\documentclass[10pt]{article}
\usepackage[left=3.8cm, right=3.8cm, top=3.8cm, bottom=2.25cm]{geometry}
\usepackage{amsmath,color}
\usepackage{amsfonts,amsthm}
\usepackage{amssymb,enumerate,enumitem,verbatim}
\usepackage[pdftex]{graphicx}
\usepackage[export]{adjustbox}
\usepackage{amsmath,setspace,scalefnt}
\usepackage[usenames,dvipsnames,svgnames,table]{xcolor}
\usepackage{amsfonts,amsthm}
\usepackage{amssymb,enumitem,verbatim}
\usepackage{graphicx}
\usepackage{tikz,caption,subcaption,thm-restate}
\usetikzlibrary{calc,decorations.pathmorphing}
\usepackage[titles]{tocloft}
\usepackage[symbol]{footmisc}
\usepackage{hyperref}
\usepackage{appendix}

\newcounter{capitalcounter}

\newcounter{claimcounter}

\newcounter{myfootnote}[page]

\newcommand{\Footnote}[1]{\stepcounter{myfootnote}\footnote{#1}}

\newtheorem{lemma}{Lemma}[section]
\newtheorem{corollary}[lemma]{Corollary}
\newtheorem{theorem}[lemma]{Theorem}

\newtheorem{prop}[lemma]{Proposition}
\newtheorem{conjecture}[lemma]{Conjecture}

\theoremstyle{definition}
\newtheorem{defn}[lemma]{Definition}

\newtheorem{claim}{Claim}

\theoremstyle{remark}
\newtheorem*{rmk}{Remark}

\newtheorem*{example}{Example}
\newtheorem*{nte}{Note}

\global\long\def\eps{\varepsilon}

\global\long\def\N{\mathbb{N}}
\global\long\def\R{\mathbb{R}}

\global\long\def\re{\begin{rmk}}
\global\long\def\mark{\end{rmk}}
\global\long\def\ex{\begin{example}}
\global\long\def\ple{\end{example}}
\global\long\def\no{\begin{nte}}
\global\long\def\ted{\end{nte}}
\global\long\def\en{\begin{compactenum}}
\global\long\def\um{\end{compactenum}}
\global\long\def\li{\begin{compactitem}}
\global\long\def\st{\end{compactitem}}
\global\long\def\de{\begin{defn}}
\global\long\def\fn{\end{defn}}
\global\long\def\cor{\begin{corollary}}
\global\long\def\ary{\end{corollary}}
\global\long\def\lem{\begin{lemma}}
\global\long\def\ma{\end{lemma}}
\global\long\def\arr{\begin{array}}
\global\long\def\ay{\end{array}}
\global\long\def\pr{\begin{proof}}
\global\long\def\oof{\end{proof}}

\newcounter{propcounter}

\newcommand{\claimproofstart}[1][Proof]{\begin{proof}[#1]
\renewcommand{\qedsymbol}{$\boxdot$}}
\newcommand\claimproofend{
\end{proof}
\renewcommand{\qedsymbol}{$\square$}}

\newcommand{\llpoly}{\stackrel{\scriptscriptstyle{\text{\sc poly}}}{\ll}}

\hypersetup{hidelinks}

\newcommand{\Alabel}[1]{%
  \textbf{A#1%
    \ifnum#1=1 *%
    \else\ifnum#1=7 *%
    \fi\fi
  }%
}

\makeatletter
\newcommand{\astar}[1]{%
  \expandafter\@astar\csname c@#1\endcsname
}
\newcommand{\@astar}[1]{%
  \@arabic#1%
  \ifnum#1=1 *%
  \else\ifnum#1=7 *%
  \fi\fi
}
\AddEnumerateCounter{\astar}{\@astar}{7*}
\makeatother

\title{A proof of Andersen's rainbow path conjecture for large $n$}
\author{Candida Bowtell\thanks{School of Mathematics, University of Birmingham, Edgbaston, Birmingham, UK. Research supported by Leverhulme Trust Early Career Fellowship ECF--2023--393. Email: {\tt c.bowtell@bham.ac.uk}}\and Richard Montgomery\thanks{Mathematics Institute, University of Warwick, Coventry, CV4 7AL, UK. Research supported by the European Research Council (ERC) under the European Union Horizon 2020 research and innovation programme (grant agreement No.\ 947978). Email: {\tt richard.montgomery@warwick.ac.uk} } \and Alp M\"uyesser\thanks{Mathematical Institute, University of Oxford, Oxford, UK. Email: {\tt alp.muyesser@new.ox.ac.uk}} \and Alexey Pokrovskiy\thanks{Department of Mathematics, University College London, London, UK. {Email}: {\tt dralexeypokrovskiy@gmail.com}}}

\begin{document}
\maketitle
\begin{abstract}
We show that, for sufficiently large $n$, every properly edge-coloured $n$-vertex complete graph contains a path with $n-1$ vertices which uses each colour at most once (that is, a rainbow path). This resolves a conjecture of Andersen from 1989 for all large $n$ and improves previous results of Alon–Pokrovskiy–Sudakov, and then Balogh–Molla, which showed that rainbow paths/cycles of length $n-O(n^{1/2}\log n)$ exist in this setting. Furthermore, with related methods, we show that, for every sufficiently large $n$, every Latin square of order $n$ contains a cycle-free transversal of order $n-2$, confirming a conjecture of Gy\'arf\'as and S\'ark\"ozy from 2014 for large $n$.
\end{abstract}


\section{Introduction}\label{sec:intro}
Edge-colourings provide a natural way to encode additional structure in graphs and arise in a wide range of applications. A key property often sought in a subgraph of a coloured graph is that it uses each colour at most once. Such subgraphs are called \textit{rainbow}. To find large rainbow subgraphs, the host graph must use many different colours. A natural way to guarantee this is to consider a \textit{proper} edge-colouring, that is, one in which no two edges of the same colour are incident to a common vertex. Henceforth, all proper colourings refer to proper edge-colourings.

Rainbow subgraph problems in properly coloured graphs can be used to reformulate a vast number of problems in design theory, graph decompositions, additive combinatorics, discrete geometry, and coding theory. These include the Ryser--Brualdi--Stein conjecture on transversals in Latin squares~\cite{Ryser,montgomery2023proof}, Ringel's conjecture on decomposing complete graphs into trees~\cite{ringel1963theory,montgomery2021proof}, the graceful labelling conjecture~\cite{rosa1966certain,gallian2022dynamic,letzter2025gracesize}, Graham's rearrangement conjecture~\cite{graham1971sums, bucic2025towards}, Jamison's direction paths conjecture~\cite{Jamison1987DirectionTrees, KleitmanPinchasi2005Caterpillar}, and several questions concerning locally correctable and locally decodable codes~\cite{alon2025essentially, alrabiah2024near}. We refer the reader to a survey by Sudakov~\cite{sudakov2024restricted} for further applications of restricted subgraph problems in coding theory (see also \cite{hsieh2025small}), and to recent surveys by Pokrovskiy~\cite{pokrovskiy2022rainbow} and Montgomery~\cite{Mysurvey} for an overview of rainbow subgraph questions in extremal combinatorics.

The topic of this paper is one of the most natural extremal questions in this area: how long a rainbow path/cycle can we find in any properly coloured $n$-vertex complete graph $K_n$? When $n$ is even, $K_n$ can be properly coloured with $n-1$ colours, so a rainbow Hamilton cycle need not always exist. In 1980, Hahn~\cite{hahn1980jeu} conjectured that any properly coloured $K_n$ contains a rainbow Hamilton path. However, a few years later, this was disproved for infinitely many values of $n$ by a now well-known construction of Maamoun and Meyniel~\cite{maamoun1984problem} given by colouring each edge within the vertex set $\mathbb{F}_2^k$ with the sum of its endpoints. In 1989, Andersen~\cite{andersen1989hamilton} conjectured that Hahn's conjecture is not far wrong, as follows.

\begin{conjecture}\label{conj:andersen}
    Any properly coloured $n$-vertex complete graph contains a rainbow path with $n-1$ vertices.
\end{conjecture}

Progress towards Conjecture~\ref{conj:andersen} has been tightly linked to progress on the corresponding problem for long rainbow cycles. In 2007, Akbari, Etesami, Mahini, and Mahmoody~\cite{akbari2007rainbow} showed that every properly coloured $K_n$ has a rainbow cycle with length $n/2-1$. After a series of improvements (see~\cite{gyarfas2014rainbow,gyarfas2011long,gebauer2012rainbow,chen2015long}), Alon, Pokrovskiy, and Sudakov~\cite{alon2017random} gave the first approximate version of Andersen's conjecture in 2017. Indeed, they showed that any properly coloured $K_n$ contains a rainbow path (and cycle) of length $n-O(n^{3/4}\log n)$.
Shortly after, Balogh and Molla~\cite{balogh2019long} observed that these methods could be pushed to improve this to $n-O(n^{1/2}\log n)$.

The difficulty of proving even an approximate form of Andersen's conjecture is somewhat surprising since the results on the corresponding problem for rainbow matchings in properly coloured complete balanced bipartite graphs have long been much stronger. This corresponding problem is the famous Ryser--Brualdi--Stein conjecture on transversals in Latin squares, with origins from 1967~\cite{Ryser}. In its equivalent form in edge-coloured graphs, it states that any complete bipartite graph with $n$ vertices in each class (the graph $K_{n,n}$) which is properly coloured with $n$ colours contains a perfect rainbow matching (i.e., $n$ vertex-disjoint edges, each of a different colour) if $n$ is odd and an $(n-1)$-edge rainbow matching if $n$ is even.
An approximate form of the Ryser--Brualdi--Stein conjecture has been known since 1978 due to independent work by Brouwer, De Vries, and Wieringa~\cite{brouwer1978lower} and Woolbright~\cite{woolbright78}, who showed in this setting that a rainbow matching with $n-\sqrt{n}$ edges always exists. Moreover, though the proof contained an error later corrected by Hatami and Shor~\cite{hatami2008lower}, Shor~\cite{shor} proved in 1982 that a rainbow matching with $n-O(\log^2 n)$ edges always exists. More recently, Keevash, Pokrovskiy, Sudakov, and Yepremyan \cite{KPSY} showed that any proper colouring of $K_{n,n}$ has a rainbow matching with $n-O(\log n/\log\log n)$ edges, and later Montgomery~\cite{montgomery2023proof} showed that when $n$ is large such a matching can have $n-1$ edges.

No such result has been adapted to find rainbow paths of comparable size (as also discussed in the surveys \cite{pokrovskiy2022rainbow, Mysurvey}). In this paper, we introduce the first methods to do so, and thereby resolve Andersen's conjecture for large $n$ in the following stronger form.
\begin{theorem}\label{thm:mainintro} There exists $n_0$ such that any properly coloured complete graph on $n\geq n_0$ vertices contains a rainbow path on $n-1$ vertices. Furthermore, if the colouring has at least $n$ colours, then there exists a rainbow Hamilton path.
\end{theorem}
Observe that if $n$ is odd, then any proper colouring of $K_n$ uses at least $n$ colours and thus must contain a rainbow Hamilton path by the above theorem (if $n$ is large). Note that this shows that Hahn's 1980 conjecture~\cite{hahn1980jeu} holds for large odd values of $n$ (despite being false in general). Furthermore, it follows immediately from the above theorem that every properly coloured complete graph $K_n$ has a Hamilton cycle with at least $n-2$ distinct colours when $n$ is large, confirming a conjecture of Akbari, Etesami, Mahini, and Mahmoody~\cite{akbari2007rainbow} for large $n$.
\par We discuss another application in which the colouring encodes geometric constraints. Jamison~\cite{Jamison1987DirectionTrees} conjectured in 1987 that, for any finite \textit{cap}, that is, a set of points $S\subseteq \mathbb{R}^2$ such that no three points are collinear, the points of $S$ can be ordered as $s_1,\ldots,s_n$ so that the lines $s_is_{i+1}$ have distinct directions; equivalently, no two lines of the form $s_is_{i+1}$ are parallel. We may construct an auxiliary properly coloured complete graph on vertex set $S$ by colouring each edge $\{s,s'\}$ by the slope of the line $ss'$. This colouring is proper since $S$ is a cap. Observe that Jamison's conjecture is equivalent to the existence of a rainbow Hamilton path in this auxiliary graph. Jamison's conjecture was settled in 2005 by Kleitman and Pinchasi~\cite{KleitmanPinchasi2005Caterpillar}.

\par Theorem~\ref{thm:mainintro} gives a different proof of Jamison's conjecture for sufficiently large point sets, since elementary arguments show that a cap with $n$ points in $\mathbb{R}^2$ determines at least $n$ distinct directions, meaning that the auxiliary graph has at least $n$ distinct colours~\cite{Jamison1986FewSlopesWithoutCollinearity}. The proof of Theorem~\ref{thm:mainintro} is significantly more complicated than the proof of Jamison's conjecture given in~\cite{KleitmanPinchasi2005Caterpillar}; however, Theorem~\ref{thm:mainintro} is also significantly more general. Indeed, Theorem~\ref{thm:mainintro} requires no geometric input beyond a lower bound on the number of directions determined by a cap. 

Our methods also apply to find long rainbow cycles, see Section~\ref{sec:andersen} for more details.

\begin{theorem}\label{thm:mainintrocycle} There exist $n_0$ and $C$ such that any properly coloured complete graph on $n\geq n_0$ vertices contains a rainbow cycle with at least $n-C$ vertices.
\end{theorem}

\subsection{The Gy\'arf\'as--S\'ark\"ozy conjecture on cycle-free transversals}
We also prove a digraph generalisation of Andersen's conjecture that links the study of rainbow paths and transversals in Latin squares. To make this connection, we first introduce some terminology. A complete digraph on $n$ vertices is a digraph with vertex set $[n]$ and edges $\vec{ij}$ for each $(i,j)\in[n]^2$ with $i\neq j$. We denote by $\overleftrightarrow{K_n}$ the digraph obtained by starting with the complete digraph on $[n]$ and adding the loops $(i,i)$ for each $i\in [n]$. A proper colouring\Footnote{A digraph $\vec{G}$ is properly coloured if there are no two in-edges of any vertex with the same colour and no two out-edges of any vertex with the same colour.} of the edges of a digraph is \emph{optimal} if each colour-class forms a $1$-factor\Footnote{Equivalently, a disjoint union of directed cycles (loops and cycles of length $2$ count as cycles), or a spanning subgraph with in-degree and out-degree $1$ at each vertex.}. Optimal proper colourings of $\overleftrightarrow{K_n}$ and Latin squares\Footnote{A Latin square is an $n\times n$ array filled with $n$ symbols so that no symbol repeats in a row or a column.} are in one-to-one correspondence, as the symbol $c$ in the $i$th row and the $j$th column can be thought of as the arc $(i,j)$ having colour $c$. A \textit{transversal} of a Latin square is a selection of entries not sharing a row, column, or a symbol. In digraph language, a transversal corresponds to a rainbow subgraph where both the in-degree and the out-degree of each vertex is at most one. We refer the reader to \cite{pokrovskiy2022rainbow} for a more detailed discussion.
\par In 2014, Gy\'arf\'as and S\'ark\"ozy~\cite{gyarfas2014rainbow} conjectured that any Latin square contains a `cycle-free' transversal of size $n-2$. In graph language, the conjecture can be equivalently stated as asserting that any $\overleftrightarrow{K_n}$ with an optimal colouring contains a rainbow subgraph with $n-2$ edges which is a union of (at most two) directed paths. This conjecture is known to hold for almost all optimal colourings due to Gould and Kelly~\cite{gould2023hamilton}, and approximately in general due to Benzing, Pokrovskiy, and Sudakov~\cite{benzing2020long} with polynomial error terms.
\par We confirm the Gy\'arf\'as--S\'ark\"ozy conjecture for large $n$ in the following stronger form.
\begin{theorem}\label{thm:GS-intro}
    There exists $n_0$ such that any optimally coloured $\overleftrightarrow{K_n}$ on $n\geq n_0$ vertices contains a rainbow directed path on $n-1$ vertices.
\end{theorem}
\par As the loops cannot be part of any directed path, the above theorem equivalently asserts that the complete digraph obtained upon removing the loops from $\overleftrightarrow{K_n}$ contains a rainbow directed path on $n-1$ vertices. We refer the reader to Theorem~\ref{thm:furthermore} for more general results concerning directed analogues of Andersen's conjecture.

\subsection{Proof methods}
As discussed earlier, rainbow matching problems are much better understood than rainbow cycle problems. In this paper, we address this gap by introducing methods that convert rainbow matchings into rainbow cycles in properly coloured graphs. To `upgrade' theorems about finding large rainbow matchings into theorems about finding large rainbow cycles, we apply such results within random subgraphs of any properly coloured complete graph $K_n$, and then merge such matchings to get a collection of rainbow paths with $n-o(n)$ edges. To connect these paths together into a long rainbow cycle, we use novel methods inspired by \textit{sorting networks}.
\par Sorting networks are a fundamental tool in computer science that give a way to sort an unordered list of $n$ integers by a sequence of comparisons and transpositions that are \textit{fixed in advance} (for an overview, see~\cite{Knuth1973}). A celebrated construction of Ajtai, Koml\'os, and Szemer\'edi~\cite{ajtai19830} establishes the existence of sorting networks which make at most $O(n\log n)$ comparisons, which is best possible due to information-theoretic constraints. The existence of efficient sorting networks has surprising applications in graph theory. The earliest such applications that we are aware of appear in works of H{\"a}ggkvist and Thomason~\cite{HaggkvistThomason1995,HaggkvistThomason1997}. More recently, sorting networks have been used to tackle several graph labelling problems~\cite{muyesser2022random} and to facilitate the construction of arbitrary bounded-degree spanning trees and Hamilton cycles in expanders~\cite{hyde2025spanning, draganic2024hamiltonicity}.      

A key novelty in our work is the construction, within an arbitrary proper colouring of a complete graph, of a structure with properties related to those of sorting networks. This structure is a novel relaxation of a sorting network, which we call a \textit{Hamilton router} (see Section~\ref{sec:overview} for an overview and see Section~\ref{sec:routers} for the definition of a Hamilton router). Hamilton routers, although unable to fully sort a given list of $n$ integers, are able to produce some ordering $(k_i)_{i\in [n]}$ so that the permutation $i\to k_i$ consists of a single cycle (i.e., has cycle type $n$). The benefit of this relaxation is that we can construct such routers which make at most $O(n)$ comparisons. This critically eases the demands on the results we use on rainbow matchings in subgraphs of properly coloured complete graphs. Moreover, we show that this relaxed property of Hamilton routers is all that is required for graph-theoretic applications for problems related to Hamilton cycles\Footnote{In particular, Hamilton routers could replace sorting networks in prior applications such as \cite{muyesser2022random, draganic2024hamiltonicity}, although for applications concerning spanning trees or more complicated cycle structures (such as those in \cite{muyesser2026cycle, muyesser2025graham, hyde2025spanning}), Hamilton routers appear to offer no additional leverage.}.

To find the large rainbow matchings required for the proof of Theorem~\ref{thm:mainintro}, we use the recent breakthrough results on the Ryser--Brualdi--Stein conjecture by Montgomery~\cite{montgomery2023proof} (more specifically, we use a technical theorem from this work as a `black box'). We note, however, that previous results from \cite{KPSY} in combination with our methods would already show, in the setting of Theorem~\ref{thm:mainintro}, the existence  of rainbow cycles of length $n-O(\log n)$, already a significant improvement over the results of Alon, Pokrovskiy, and Sudakov~\cite{alon2017random} and Balogh and Molla~\cite{balogh2019long}.

\vspace{2mm}
\smallskip

\noindent \textbf{Organisation of the rest of the paper.} In the next section, we introduce some notation, give a more detailed proof overview, and make explicit the results we require from \cite{montgomery2023proof}. Section~\ref{sec:routers} is devoted to constructing Hamilton routers and finding them in proper colourings. Section~\ref{sec:puttingtogether} proves Theorem~\ref{thm:furthermore}, which is a version of our main theorem under a typicality hypothesis. Section~\ref{sec:andersen} gives a reduction from the conjectures of Andersen and Gy\'arf\'as--S\'ark\"ozy to Theorem~\ref{thm:furthermore}, thereby completing the proofs of our main theorems modulo the proof of Theorem~\ref{thm:BPW}, which is then given in Appendix~\ref{appendix:A}.

\section{Preliminaries and proof overview}
In this section, after covering our notation in Section~\ref{sec:notation}, we give an overview of our methods in Section~\ref{sec:overview}. In Section~\ref{sec:imports}, we reintroduce and discuss the results we need from \cite{montgomery2023proof}. In Sections \ref{sec:typical} and \ref{sec:nibble}, we then reintroduce typical hypergraphs and results  on matchings in them.


\subsection{Notation}\label{sec:notation}
We use $[n]$ to denote $\{1,2,\ldots, n\}$. A graph $G$ has vertex set $V(G)$ and edge set $E(G)$, and $|G|=|V(G)|$ and $e(G)=|E(G)|$. If $G$ has an edge colouring, then $C(G)$ is the set of colours appearing on the edges of $G$, and, for each $c\in C(G)$, $E_c(G)$ is the set of edges of $G$ with colour $c$. The colour of an edge $e\in E(G)$ is $c(e)$ and its vertex set is $V(e)$. Given $A\subseteq V(G)$, $G[A]$ is the induced subgraph of $G$ with vertex set $A$ and edge set $\{uv\in E(G):u,v\in A\}$. Given, further, $B\subseteq V(G)$, $G[A,B]$ is the graph with vertex set $A\cup B$ and edge set $\{uv\in E(G):u\in A,v\in B\}$ and $e_G(A,B)=|\{(u,v):u\in A,v\in B,uv\in E(G)\}|$.

A digraph $D$ has vertex set $V(D)$ and edge set $E(D)$ consisting of edges $xy$ directed from $x$ to $y$. We say $D$ is \emph{properly edge-coloured} if there are no two in-edges of any vertex with the same colour and no two out-edges of any vertex with the same colour. Given an edge-coloured digraph $D$, sets $X,Y\subseteq V(D)$ and $C\subseteq C(D)$, $D[X,Y,C]$ is the undirected coloured bipartite graph consisting of vertex classes $X,Y$, and each edge $xy$ with colour $c$ when $x\in X$, $y\in Y$, $c\in C$ and $xy\in E(D)$.
Given $X\subset V(G)$ and $C\subset C(G)$, $G[X,C]$ denotes the induced subgraph on vertex set $X$ with edges with colours in $C$. Note that every undirected graph $G$ can be viewed as a digraph by directing edges in both directions and we use $G[X,Y,C]$ in the same way for graphs and digraphs. Except in the discussion of complete digraphs in Section~\ref{sec:intro}, all digraphs in this paper are loopless.

Given a ground set $V$ and $p\in [0,1]$, a set $A\subset V$ is \emph{$p$-random} if each element of $V$ is included in $A$ independently at random with probability $p$. Two sets $A,B$ are disjoint \emph{$p$-random} if each element of $V$ is included in $A$ with probability $p$, included in $B$ with probability $p$, and included in neither with probability $1-2p$ (so in particular $A$ and $B$ are always disjoint and we must have $p\leq 1/2$).
\par A \emph{matching} in a graph (or hypergraph) is a set of edges which share no vertices. In a coloured graph, a subgraph is \emph{$C$-rainbow} if it is rainbow with edge colours in $C$, and \emph{exactly-$C$-rainbow} if it has exactly one edge of each colour in $C$, with no other colours appearing.  A \emph{balanced} bipartite graph is one where the two vertex classes have the same size.

We use hierarchies of constants to record dependencies between the constants in our proofs. We write $\alpha \ll \beta$ to mean  that there exists some positive increasing function $f:(0,1]\to \mathbb{R}$ so that the remainder of the proof follows if $\alpha\leq f(\beta)$. We use $\alpha \llpoly \beta$ to mean that there exists some fixed $C>0$ such that the remainder of the proof follows if $\alpha \leq \beta^C/C$. Where there are several constants in the hierarchy, the constants/functions are chosen from right to left. For more details on this notation, see~\cite[Section~3.2]{montgomery2018decompositions}. We also use `big O' notation, where the functions involved are always functions of $n$. Lastly, given $a\in \R$ and $b,c\geq 0$, we say $x=(a\pm c)b$ if $(a-c)b\leq x\leq (a+c)b$.


\subsection{Overview}\label{sec:overview}

One strategy to find Hamilton cycles in an (uncoloured) $n$-vertex graph $G$ with various properties (used, for example in~\cite{FerberKronenbergLong2017,Montgomery2020,hyde2025spanning}) is to first split the vertices of $G$ randomly into sets, say $V_i$, $i\in [k]$, of equal size. If it is likely  that each bipartite induced graph $G[V_i,V_{i+1}]$, $i\in [k-1]$, contains a perfect matching, then together these matchings will form vertex-disjoint paths from $V_1$ to $V_k$ which cover the vertices of $G$. To get a Hamilton cycle, we need to join these paths together into a cycle.

One way to join the paths is to first set aside a subgraph $S$ which is capable of emulating a \emph{sorting network} in $G$ between $V_1$ and $V_k$. This would have the property that $V_1,V_k\subset V(S)$ and, given any bijection $\phi:V_1\to V_k$, there is a set of $v,\phi(v)$-paths, $v\in V_1$, in $S$ which together use every vertex of $S$. If the vertices in $V(S)\setminus (V_1\cup V_k)$ can be split as $V_2\cup\ldots\cup V_{k-1}$ as above and the corresponding matchings, and thus paths, found, then this property of $S$ can be used to join the paths together into a Hamilton cycle in any order by choosing an appropriate bijection $\phi$.
The earliest use of such arguments that we are aware of is due to H{\"a}ggkvist and Thomason~\cite{HaggkvistThomason1995,HaggkvistThomason1997}. Related approaches using sorting networks have recently been used in extremal combinatorics in~\cite{kuhn2014proof,muyesser2022random,hyde2025spanning, muyesser2026cycle}.

In our setting, we want to carry out the above strategy in properly coloured graphs to produce a long rainbow cycle. Modulo various complications that arise, we will be able to find perfect rainbow matchings between randomly chosen vertex sets in our properly coloured complete graph $G$ using results from~\cite{montgomery2023proof} (see Section~\ref{sec:imports}). Moreover, we will be able to restrict this further to random sets of colours so that the paths we create by merging matchings will be collectively rainbow.

In order to use these results as a `black box', it is very desirable to use $k=O(1)$ in the outline above. However, well-known lower bounds for the optimal depth of sorting networks (see, e.g.,~\cite{ajtai19830}) imply that any subgraph $S$ with the properties discussed above satisfies $|S|=\Omega((n/k)\log (n/k))$, so that we would need $k=\Omega(\log n)$. To avoid this, observe that in the sketch above we do not mind the order in which we join the paths together using $S$, only that they are joined into a Hamilton cycle. Thus, we will replace the sorting network emulator by a structure more specifically geared to our task. This structure we call a \emph{Hamilton router}, as defined in Section~\ref{sec:routers}.

There are two remaining points to consider: whether our Hamilton router $S$ can be found in an arbitrary properly coloured complete graph, and whether removing $S$ affects the rest of the argument in which we partition into random sets. Our construction for a Hamilton router will reduce the first question to finding certain `rainbow sorting networks of order 2', which we do via a strong theorem of Janzer~\cite{janzer2023rainbow} (see Theorem~\ref{Janzer_thm}). The second point is a serious concern, as removing the Hamilton router could badly affect how `pseudorandom' our initial graph is.
We address this issue by finding the Hamilton router in a smaller random set, so as to not spoil the typicality of the entire digraph, but only that of a smaller random subset. The leftover in this small random subset (after the router is found) is then dealt with by iterated applications of a version of the semi-random method, see Theorem~\ref{thm:RSBcoveringstephyper} and the proof of Theorem~\ref{thm:furthermore} for more details.

This outline covers the main points behind our proof; as further challenges and technicalities arise, we will explain them and how we overcome them in context throughout the paper.


\subsection{Results for rainbow matchings in bipartite coloured graphs}\label{sec:imports}
\par To make use of the coloured Hamilton router construction introduced informally in the sketch above, we need results guaranteeing that large rainbow matchings are likely to exist between randomly sampled sets of vertices and colours. We import the results we need from \cite{montgomery2023proof}.

Given an edge-coloured graph $G$, sets $X,Y\subseteq V(G)$ and $C\subseteq C(G)$ (the colour-set of $G$), $G[X,Y,C]$ is the coloured bipartite graph consisting of vertex classes $X$ and $Y$, and edges from $X$ to $Y$ with colours in $C$. We can now summarise the results we require from \cite{montgomery2023proof} in the following theorem, before discussing it further. (See Section~\ref{sec:notation} for the formal definitions of $p$-random, disjoint $p$-random, and the meaning of the symbols $\ll$ and $\llpoly$.)

In this theorem, in a common setting, we collect three results \ref{RBS}, \ref{BPW} and \ref{SPC}, which can be considered respectively to be randomised versions of the Ryser--Brualdi--Stein conjecture~\cite{Ryser,Brualdi,Stein}, a conjecture of Best, Pula, and Wanless~\cite{best2021small} that asserts that any proper colouring of $K_{n,n+1}$ has a rainbow matching of size $n$, and a result \cite[Theorem 3.2]{montgomery2023proof} asserting that, when $n$ is large, any proper colouring of $K_{n,n}$ with $100$ `spare' colours contains a rainbow matching of size $n$ (see Theorem~\ref{thm-technical-variant} in Section~\ref{appendix:A} for a precise statement).

\begin{theorem}\label{thm:BPW}Let $1/n\ll p\leq 1$, let $\eps \llpoly\log^{-1}n$, and let $D$ be a properly coloured $(n,\eps)$-typical\Footnote{See Definition~\ref{defn:digraph} for a formal definition.} digraph. Let $X,Y$ be disjoint $p$-random subsets of $V(D)$, and let $C$ be a $p$-random subset of $C(D)$, sampled independently. Then, with high probability, the following hold for any disjoint $X',Y'\subseteq V(D)$ and $C'\subseteq C(D)$ satisfying $|Q\triangle Q'|\leq \eps n$ for each $Q\in \{X,Y,C\}$.

\begin{enumerate}[label = \textup{\textbf{RBS}}]
    \item\label{RBS}  If $|X'|=|Y'|\leq |C'|$, then $D[X',Y',C']$ has a rainbow matching of size $|X'|-1$.\renewcommand{\labelenumi}{\textup{\textbf{BPW}}}
    \item\label{BPW}  If $|X'|=|Y'|+1\leq |C'|$, then $D[X',Y',C']$ has a rainbow matching of size $|X'|-1=|Y'|$.\renewcommand{\labelenumi}{\textup{\textbf{SPC}}}
    \item\label{SPC}  If $|X'|=|Y'|\leq |C'|-100$, then $D[X',Y',C']$ has a rainbow matching of size $|X'|$.
\end{enumerate}
\end{theorem}
The above theorem is not stated directly in \cite{montgomery2023proof}, though it follows from its proofs with minor modification. We discuss any required modification, and give formal proofs where necessary, in Appendix~\ref{appendix:A}. For now, let us note only that the main theorems of \cite{montgomery2023proof} are stated under a technical assumption of `proper-pseudorandomness', which holds, for example, for nearly complete balanced bipartite graphs properly coloured with around $n$ colours (as shown in \cite[Proposition 3.12]{montgomery2023proof}). Proper-pseudorandomness also holds for dense random samples of such nearly-complete graphs (as shown in \cite[Section 10]{montgomery2023proof}). Theorem~\ref{thm:BPW} is a consequence of combining the main theorems of \cite{montgomery2023proof} together with the assertion that proper-pseudorandomness is inherited upon random sampling. Our model for random sampling is slightly different from that used in \cite{montgomery2023proof}, giving rise to most of the proofs in Appendix~\ref{appendix:A}.
\vspace{2mm}



\subsection{Typical hypergraphs and digraphs}\label{sec:typical}
Here we recall some standard definitions, following \cite{KPSY,montgomery2023proof,muyesser2022random}. A correspondence between edge-coloured graphs and $3$-uniform hypergraphs is obtained as follows.

\begin{defn}\label{defn:HG}
Given a coloured bipartite graph $G$ with vertex classes $A$ and $B$, we denote by $\mathcal{H}(G)$ the 3-partite 3-uniform hypergraph with vertex classes $A$, $B$ and $C(G)$, and edge set $\{abc:a\in A,b\in B,c\in C(G),ab\in E_c(G)\}$.
\end{defn}

We  start by defining a typical bipartite graph as follows.
\begin{defn}\label{defn:typical}
A bipartite graph $H$ with vertex classes $A$ and $B$ is \emph{$(n,p,\eps)$-typical} if the following all hold. \textbf{a)} $|A|=(1\pm \eps)n$ and $|B|=(1\pm \eps)n$. \textbf{b)} For each $v\in V(H)$, $d_{H}(v)=(1\pm \eps)pn$. \textbf{c)} For each distinct $u,v\in V(H)$ with $u,v\in A$ or $u,v\in B$, we have $|N_H(u)\cap N_H(v)|=(1\pm \eps)p^2n$.
\end{defn}

If the first two conditions above hold, but not necessarily the third, then we call $H$ \emph{$(n,p,\eps)$-regular}.\vspace{1mm}

We now define typicality for $3$-uniform hypergraphs.

\begin{defn} Given a 3-uniform hypergraph $\mathcal{H}$ and any disjoint sets $X,Y\subset V(\mathcal{H})$, let $\mathcal{H}_{XY}$ be the bipartite graph with vertex classes $X$ and $Y$, with edges $xy$ exactly for those $x\in X$ and $y\in Y$ for which there exists some $z\in V(\mathcal{H})\setminus \{x,y\}$ with $\{x,y,z\}\in E(\mathcal{H})$.
\end{defn}

\begin{defn}\label{def:typicalgraph}
A simple 3-partite 3-uniform hypergraph $\mathcal{H}$ with vertex classes $A$, $B$ and $C$ is \emph{$(n,p,\eps)$-typical} if each of $\mathcal{H}_{AB}$, $\mathcal{H}_{BC}$ and $\mathcal{H}_{AC}$ is $(n,p,\eps)$-typical.
\end{defn}

We also give a digraph analogue of the above definitions.

\begin{defn}[Typical digraphs]\label{defn:digraph}
    We call a properly edge-coloured digraph without loops $(n,\eps)$\textit{-typical} if it has $(1\pm \eps )n$ vertices, each vertex has $(1\pm \eps )n$ in-neighbours and $(1\pm \eps )n$ out-neighbours and each colour appears $(1\pm \eps )n$ times.\label{defn:typical2var}
\end{defn}


\subsection{A nibble-type lemma and Chernoff's bound}\label{sec:nibble}
The R\"odl nibble, also known as the semi-random method, has played a fundamental role in the study of matchings in regular (or almost regular) hypergraphs with low codegrees~\cite{rodlandhisnibble}, and will allow us to find almost-perfect matchings in our typical hypergraphs.
For discussion of the R\"odl nibble we refer the reader to the survey by
Kang, Kelly, K\"uhn, Osthus, and Methuku~\cite{kang2021graph} and \cite[Section~4.7]{alon2004probabilistic}.
We will need a flexible nibble-like statement that, essentially, we can still use after removing some subset of a random vertex set (corresponding to the Hamilton-router $S$ in the sketch in Section~\ref{sec:overview}). Such a result was used in \cite{montgomery2018decompositions}, and then by, for example, \cite{montgomery2023proof, muyesser2022random}. After stating the result we use, we discuss its proof which can be found in Appendix~\ref{appendix:B}. Recall that a hypergraph is simple if every pair of edges intersects on at most one vertex.

\begin{theorem}\label{thm:RSBcoveringstephyper}
Let $1/n, \, \eps \llpoly \eta \llpoly  p,q\leq 1$. Let $2q/3\leq q_A,q_B,q_C\leq q$. Let $\mathcal{H}$ be a simple 3-partite 3-uniform hypergraph which is $(n,p,\eps)$-typical with vertex classes $A$, $B$ and $C$. Let $A'$ be a $q_A$-random subset of $A$, let $B'$ be a $q_B$-random subset of $B$ and let $C'$ be a $q_C$-random subset of $C$, possibly with dependencies. Then, with high probability, the following holds.

Given any sets $\bar{A}\subset A$, $\bar B\subset B$, $\bar{C}\subset C$ of size $qn$ such that $A'\cup B'\cup C'\subset \bar{A}\cup \bar{B}\cup \bar{C}$, there is a matching in $\mathcal{H}[\bar{A},\bar{B},\bar{C}]$ with at least $qn-\eta n$ edges.
\end{theorem}

In the setting of Theorem~\ref{thm:RSBcoveringstephyper}, a usual nibble-type theorem would guarantee that $\mathcal{H}[A', B',C']$ contains a nearly spanning matching -- this follows simply because $\mathcal{H}[A', B',C']$ would inherit the regularity properties of $\mathcal{H}$ with good precision, which is all that is needed to invoke a standard nibble-type theorem (see, for example,~\cite[Section~4.7]{alon2004probabilistic}). Even though $\mathcal{H}[\bar{A},\bar{B},\bar{C}]$ can have significantly more vertices than $\mathcal{H}[A', B',C']$, and even though the new vertices (e.g., those in $\bar{A}\setminus A'$) could be chosen adversarially, Theorem~\ref{thm:RSBcoveringstephyper} still guarantees a nearly perfect matching in $\mathcal{H}[\bar{A},\bar{B},\bar{C}]$. The result is very similar to  \cite[Theorem 4.1]{montgomery2023proof} except that the random sets may have some dependencies, as in \cite[Lemma 3.9]{muyesser2022random}. The proof of Theorem~\ref{thm:RSBcoveringstephyper} contains no novelty compared with \cite{montgomery2023proof, muyesser2022random}, but as it does not formally follow from either, we include a proof in Appendix~\ref{appendix:B}.

Finally, we recall the following standard version of Chernoff's bound (see, for example, \cite[Corollary 2.2 and Theorem 2.10]{janson2011random}).
\begin{lemma}[Chernoff's bound]\label{chernoff-bound}
Let $X$ be a random variable with mean $\mu$ which is binomially distributed or hypergeometrically distributed.
Then, for any $0<\gamma<1$, we have that $\mathbb{P}(|X-\mu|\geq \gamma \mu)\leq 2e^{-\mu \gamma^2/3}$.
\end{lemma}


\section{Hamilton routers}\label{sec:hamroutersintro}
In this section, we build directed coloured Hamilton routers. Later, we will apply our results for directed graphs to any properly coloured graph $G$ by replacing its edges $xy\in E(G)$ with the two directed edges $xy$ and $yx$ with colour $c(xy)$, noting that this gives a properly coloured digraph. In part for the discussion in Section~\ref{sec:overview}, we will define our key structure, a Hamilton router, in both undirected and directed forms, using `(directed)' to mean that the sentence applies in the directed and undirected cases.

\begin{defn}[Hamilton router] Suppose $A$ and $B$ are disjoint subsets of $V(S)$ with $|A|=|B|$ for some (directed) graph $S$. We call $S$ a (directed) $A,B$\textit{-Hamilton-router} if, for every bijection $\phi\,\colon A\to B$, $V(S)$ can be partitioned into a collection of (directed) paths, $\mathcal{P}_\phi$, with endpoints in $A\times B$ so that, letting $M_\phi$ be the (directed) matching $\{(\phi(a),a)\colon a\in A\}\subseteq B\times A$, the edge-set $E(\mathcal{P}_\phi)\cup M_\phi$ forms a Hamilton cycle over $V(S)$. The \textit{order} of an $A,B$-Hamilton-router is $|A|=|B|$ and the \emph{depth} of an $A,B$-Hamilton-router is $|V(S)|/|A|-1$ (i.e., the average length of a component path in any path partition $\mathcal{P}_\phi$).
\par We define $A,B,C$-Hamilton-routers analogously as edge-coloured Hamilton routers where the colour set of each path partition $\mathcal{P}_\phi$ is always $C$ and $|C|=|E(\mathcal{P}_\phi)|$ (independently of the bijection $\phi$). If $A,B,C$ are implicit, we simply call these objects \emph{coloured Hamilton routers}.
\end{defn}

\noindent \textbf{Remark 1.} $A,B$-Hamilton routers $S$ of order $2$ have a special role in our arguments. In this case, the definition is equivalent to $S$ having designated vertices $a_1,a_2,b_1,b_2$, and two pairs of paths $\{a_1\xrightarrow{P_1} b_1,a_2\xrightarrow{P_2}b_2\}$, $\{a_1\xrightarrow{Q_1} b_2,a_2\xrightarrow{Q_2}b_1\}$ where $V(S)=V(P_1)\sqcup V(P_2)=V(Q_1)\sqcup V(Q_2)$.  Throughout the paper, we will refer to Hamilton routers of order 2 as \emph{comparators}. \vspace{2mm}

\noindent \textbf{Remark 2.} A $4$-cycle yields a Hamilton router of order $2$ (where $A$ and $B$ correspond to the bipartition of the $4$-cycle). While any complete bipartite graph $K_{n,n}$ naturally provides a Hamilton router of order $n$, we will need much sparser Hamilton routers. \vspace{2mm}

\noindent \textbf{Remark 3.} It follows from the definition that, in a coloured Hamilton router, each path partition $\mathcal{P}_\phi$ is a rainbow path forest with colour set exactly $C$. Therefore, we can observe that there is no proper edge-colouring of $C_4$ that yields a coloured Hamilton router. The existence of \textit{coloured} Hamilton routers which are properly coloured is not trivial (for any order), but we will shortly construct sparse examples. \vspace{2mm}

\par We now give a simple proposition to illustrate the definition of a Hamilton router (in an uncoloured setting, for simplicity, but the proposition generalises in an obvious way to $A,B,C$-Hamilton-routers).
\begin{prop}\label{prop:basic}
    Let $D$ be a directed graph containing an $A,B$-Hamilton-router $S$. Suppose that there exists a directed path forest $\mathcal{P}$ with exactly $|A|=|B|$ paths, start-vertices in $B$, end-vertices in $A$, and internal vertices partitioning $V(D)\setminus V(S)$. Then, $D$ has a directed Hamilton cycle.
\end{prop}
\begin{proof}
    Let $\phi\colon A\to B$ be the bijection induced by the endpoints of the paths in $\mathcal{P}$  (i.e., mapping each end vertex to the corresponding start vertex). Let $\mathcal{P}_\phi$ be the directed path partition of $S$ from the definition of a Hamilton-router applied with $\phi$. Then, $\mathcal{P}_\phi\cup \mathcal{P}$ is the desired directed Hamilton cycle.
\end{proof}

\par Our first task is to build small Hamilton routers of order two. We start by showing a basic rainbow connectivity property in properly coloured typical digraphs which are close to complete (see Definition~\ref{defn:typical2var}).

\begin{lemma}[Connecting lemma]\label{Connecting_lemma}
Let $1/n, \,  \eps\leq 10^{-10}$. Let $D$ be a properly coloured $(n,\eps)$-typical digraph. Suppose there is a set $\mathcal{J}$ of at most $n/1000$ ``forbidden'' vertices and colours. Then, for any distinct $x,y\in V(D)$, there is a rainbow $x,y$-path of length $3$ whose colours and internal vertices are outside of $\mathcal{J}$.
\end{lemma}
\begin{proof}
Given distinct $x,y\in V(D)$, fix distinct $v_1, v_2 \in V(D)\setminus(\mathcal{J} \cup \{x,y\})$ for which there are $c_1,c_2\notin \mathcal{J}$ such that $xv_i \in E(D)$ and $xv_i$ has colour $c_i$ for each $i \in [2]$. Note that there are at least $(1-\eps)n-n/1000-2$ choices for $v_1$ and then $v_2$, so this is possible.

Now, for each $i\in [2]$, let $U_i$ be the out-neighbourhood of $v_i$ avoiding $\{x,y,v_1, v_2\} \cup \mathcal{J}$ and any out-edges with colour in $\{c_1, c_2\} \cup \mathcal{J}$. Furthermore, let $U_y$ be the in-neighbourhood of $y$ avoiding the vertices $\{x,y,v_1, v_2\} \cup \mathcal{J}$ and any in-edges of colours in $\{c_1,c_2\}\cup \mathcal{J}$. Note that $|V(D)\setminus U_1|, |V(D) \setminus U_2|, |V(D) \setminus U_y| \leq \eps n+n/1000+6$. Thus $|U_1 \cap U_2 \cap U_y| \geq n-3\eps n - 3n/1000-18 \geq n/2\geq 1$. Now, if there is some $v \in U_1 \cap U_2 \cap U_y$ such that $c(v_1v) \neq c(vy)$ we have our desired rainbow path of length $3$ via $xv_1vy$. Otherwise, we have that $c(v_1v)=c(vy)$ for every $v \in U_1 \cap U_2 \cap U_y$. But then, as $D$ is properly coloured, $c(v_2v) \neq c(v_1v)$ and so $xv_2vy$ is our desired rainbow $x,y$-path of length $3$ in $D$.
\end{proof}

To prove that coloured Hamilton routers of order $2$ exist, we will use the following theorem of Janzer \cite[Theorem 3.1]{janzer2023rainbow}.
\begin{theorem}[Janzer]\label{Janzer_thm}
	Let $k\geq 2$ and $s$ be positive integers. Then there exists a constant $C=C(k,s)$ with the following property. Suppose that $G=(V,E)$ is a graph with $n$ vertices and at least $Cn^{1+1/k}$ edges. Let $\sim$ be a symmetric binary relation on $V$ such that, for every $u\in V$ and $v\in V$, $v$ has at most $s$ neighbours $w\in V$ which satisfy $u\sim w$. Let $\approx$ be a binary relation on $E$ such that, for every $uv\in E$ and $w\in V$, $w$ has at most $s$ neighbours $z\in V$ which satisfy $uv\approx wz$. Then $G$ contains a $2k$-cycle $x_1x_2\dots x_{2k}$ such that $x_i\not \sim x_j$ for every $i\neq j$ and $x_ix_{i+1}\not \approx x_jx_{j+1}$ for every $i\neq j$ (where $x_{2k+1}:=x_1$).
\end{theorem}

We can now show that $A,B,C$-Hamilton-routers of order two exist in any properly coloured typical digraph which is close to complete.

\begin{lemma}\label{lem:comparators exist} Let $1/n, \, \eps \ll 1$, and let $D$ be a properly coloured $(n,\eps)$-typical digraph. Suppose there is a set $\mathcal{J}$ of at most $n/10^{10}$ forbidden vertices and colours. Then, disjointly from the vertices and colours in $\mathcal{J}$, $D$ contains a coloured Hamilton router of order $2$ and depth $9$.
\end{lemma}
\begin{proof}
Let $X, Y$ be an arbitrary nearly-balanced partition of $V(D)\setminus \mathcal{J}$.
Let $V_P=\{(x_1,x_2): x_i\in X, x_1\ne x_2\}$ and $C_P=\{(c,d): c,d\in C(D)\setminus \mathcal{J}, c\ne d\}$ be the sets of ordered pairs of distinct $X$-vertices and colours of $D\setminus \mathcal{J}$, respectively. We construct an auxiliary coloured bipartite graph $H$ with parts $V_P$ and $C_P$ whose set of colours will be $Y$, along with relations $\sim$ and $\approx$. Using Theorem~\ref{Janzer_thm}, we will then find an appropriate 6-cycle in this auxiliary graph and use this to find the comparator depicted in Figure~\ref{fig:colouredcomparator}.

For each $(x_1,x_2)\in V_P$, $(c,d)\in C_P$, and $y\in Y$, take a colour-$y$ edge $(x_1,x_2)(c,d)$ in $H$  whenever the edges $x_1y,x_2y\in E(D)$, $c(x_1y)=c$, and $c(x_2y)=d$. Define relations $\sim$ on $V(H)$ and $\approx$ on $E(H)$ as follows.
\begin{itemize}
    \item For each $u,v\in V(H)$, let $u\sim v$ if there is some vertex or colour of $D$ that is in both tuples $u$ and $v$.
    \item For each $e,f\in E(H)$, let $e\approx f$ if $e$ and $f$ have the same colour in $H$.
\end{itemize}
Note that $n^2 \leq |H| \leq 3n^2/2$,
and $e(H) \geq \sum_{y\in Y}d^-_X(y)(d^-_X(y)-1)\ge n^3/10$.
Moreover, for any $u,v\in V(H)$ there are at most $4$ neighbours $w\in N_H(v)$ with $u\sim w$ (using that $D$ is properly coloured). Indeed, for example, letting $u=(x_1,x_2)$ and $v=(c_1,c_2)$, for each $w\in N_H(v)$ with $u\sim w$ we must have that $c(wv)$ is a colour-$c_1$ or colour-$c_2$ out-neighbour of $x_1$ or $x_2$, where the other cases follow similarly. Note also that the colouring of $H$ is proper which gives that for any $uv\in E(H)$ and $w\in V(H)$ there is at most one vertex $z\in N_H(w)$ with $uv\approx zw$. By Theorem~\ref{Janzer_thm}, we get a $6$-cycle in $H$, given by $C=(u_1,u_2)(c_1, c_2)(v_1,v_2)(d_1, d_2)(w_1,w_2)(e_1, e_2)$, for which the vertices/edges are not related to each other by $\sim/\approx$ respectively.

Since the vertices are not related by $\sim$, we have that $u_1,u_2,c_1, c_2,v_1,v_2,d_1, d_2,w_1,w_2,e_1, e_2$ are all distinct. Furthermore, by the construction of $H$, they avoid $\mathcal{J}$. Let $u_3=c((e_1, e_2)(u_1,u_2))$, $u_4=c((u_1,u_2)(c_1, c_2))$, $v_3=c((c_1, c_2)(v_1,v_2))$, $v_4=c((v_1,v_2)(d_1, d_2))$, $w_3=c((d_1, d_2)(w_1,w_2))$, and $w_4=c((w_1,w_2)(e_1, e_2))$, noting that these are distinct from each other (since the edges of $C$ are unrelated to each other by $\approx$), from $u_1,u_2,v_1,v_2,w_1,w_2$ (since the $*_3, *_4$ vertices are in $Y$, while the $*_1, *_2$ vertices are in $X$) and also from $\mathcal{J}$.

Apply Lemma~\ref{Connecting_lemma} four times to get rainbow paths $P_1$ from $u_3$ to $v_1$, $P_2$ from $u_4$ to $v_2$, $P_3$ from $v_3$ to $w_1$, and $P_4$ from $v_4$ to $w_2$, each of length 3. By enlarging the set $\mathcal{J}$ at each step, we can ensure that, as well as avoiding the original set $\mathcal{J}$, these paths are vertex/colour disjoint from each other and from $u_1,u_2, u_3, u_4,v_1,v_2, v_3, v_4,w_1,w_2, w_3, w_4, c_1, c_2, d_1, d_2$, $e_1, e_2$ (except at their endpoints where necessary). Thus, we have found the following picture, which is a comparator, $\mathcal{C}$:

\begin{figure}[h]
    \centering

\tikzset{every picture/.style={line width=0.6pt}} 
\begin{tikzpicture}[x=0.75pt,y=0.75pt,yscale=-0.51,xscale=0.51]

\draw [color={rgb, 255:red, 155; green, 155; blue, 155 }  ,draw opacity=1 ][line width=1.5]    (145,41) .. controls (190.09,-2.62) and (201.09,85.38) .. (245,41) ;
\draw [color={rgb, 255:red, 155; green, 155; blue, 155 }  ,draw opacity=1 ][line width=1.5]    (145,141) .. controls (190.09,97.38) and (201.09,185.38) .. (245,141) ;
\draw [color={rgb, 255:red, 155; green, 155; blue, 155 }  ,draw opacity=1 ][line width=1.5]    (345,41) .. controls (390.09,-2.62) and (401.09,85.38) .. (445,41) ;
\draw [color={rgb, 255:red, 155; green, 155; blue, 155 }  ,draw opacity=1 ][line width=1.5]    (345,141) .. controls (390.09,97.38) and (401.09,185.38) .. (445,141) ;

\draw [color={rgb, 255:red, 255; green, 0; blue, 0 }  ,draw opacity=1 ][line width=3.75]    (45,41) -- (145,41) ;
\draw [color={rgb, 255:red, 0; green, 23; blue, 255 }  ,draw opacity=1 ][line width=3.75]    (45,41) -- (105.55,101.55) -- (145,141) ;
\draw [color={rgb, 255:red, 189; green, 16; blue, 224 }  ,draw opacity=1 ][line width=3.75]    (45,141) -- (145,141) ;
\draw [color={rgb, 255:red, 139; green, 87; blue, 42 }  ,draw opacity=1 ][line width=3.75]    (45,141) -- (145,41) ;

\draw [color={rgb, 255:red, 0; green, 5; blue, 255 }  ,draw opacity=1 ][line width=3.75]    (245,41) -- (345,41) ;
\draw [color={rgb, 255:red, 245; green, 166; blue, 35 }  ,draw opacity=1 ][line width=3.75]    (245,41) -- (345,141) ;
\draw [color={rgb, 255:red, 65; green, 117; blue, 5 }  ,draw opacity=1 ][line width=3.75]    (245,141) -- (345,141) ;
\draw [color={rgb, 255:red, 189; green, 16; blue, 224 }  ,draw opacity=1 ][line width=3.75]    (245,141) -- (345,41) ;

\draw [color={rgb, 255:red, 245; green, 166; blue, 35 }  ,draw opacity=1 ][line width=3.75]    (445,41) -- (545,41) ;
\draw [color={rgb, 255:red, 255; green, 0; blue, 0 }  ,draw opacity=1 ][line width=3.75]    (445,41) -- (545,141) ;
\draw [color={rgb, 255:red, 139; green, 87; blue, 42 }  ,draw opacity=1 ][line width=3.75]    (445,141) -- (545,141) ;
\draw [color={rgb, 255:red, 65; green, 117; blue, 5 }  ,draw opacity=1 ][line width=3.75]    (445,141) -- (545,41) ;

\draw  [fill={rgb, 255:red, 0; green, 0; blue, 0 }  ,fill opacity=1 ] (35,41) .. controls (35,35.48) and (39.48,31) .. (45,31) .. controls (50.52,31) and (55,35.48) .. (55,41) .. controls (55,46.52) and (50.52,51) .. (45,51) .. controls (39.48,51) and (35,46.52) .. (35,41) -- cycle ;
\draw  [fill={rgb, 255:red, 0; green, 0; blue, 0 }  ,fill opacity=1 ] (35,141) .. controls (35,135.48) and (39.48,131) .. (45,131) .. controls (50.52,131) and (55,135.48) .. (55,141) .. controls (55,146.52) and (50.52,151) .. (45,151) .. controls (39.48,151) and (35,146.52) .. (35,141) -- cycle ;
\draw  [fill={rgb, 255:red, 0; green, 0; blue, 0 }  ,fill opacity=1 ] (135,41) .. controls (135,35.48) and (139.48,31) .. (145,31) .. controls (150.52,31) and (155,35.48) .. (155,41) .. controls (155,46.52) and (150.52,51) .. (145,51) .. controls (139.48,51) and (135,46.52) .. (135,41) -- cycle ;
\draw  [fill={rgb, 255:red, 0; green, 0; blue, 0 }  ,fill opacity=1 ] (135,141) .. controls (135,135.48) and (139.48,131) .. (145,131) .. controls (150.52,131) and (155,135.48) .. (155,141) .. controls (155,146.52) and (150.52,151) .. (145,151) .. controls (139.48,151) and (135,146.52) .. (135,141) -- cycle ;
\draw  [fill={rgb, 255:red, 0; green, 0; blue, 0 }  ,fill opacity=1 ] (235,41) .. controls (235,35.48) and (239.48,31) .. (245,31) .. controls (250.52,31) and (255,35.48) .. (255,41) .. controls (255,46.52) and (250.52,51) .. (245,51) .. controls (239.48,51) and (235,46.52) .. (235,41) -- cycle ;
\draw  [fill={rgb, 255:red, 0; green, 0; blue, 0 }  ,fill opacity=1 ] (235,141) .. controls (235,135.48) and (239.48,131) .. (245,131) .. controls (250.52,131) and (255,135.48) .. (255,141) .. controls (255,146.52) and (250.52,151) .. (245,151) .. controls (239.48,151) and (235,146.52) .. (235,141) -- cycle ;
\draw  [fill={rgb, 255:red, 0; green, 0; blue, 0 }  ,fill opacity=1 ] (335,41) .. controls (335,35.48) and (339.48,31) .. (345,31) .. controls (350.52,31) and (355,35.48) .. (355,41) .. controls (355,46.52) and (350.52,51) .. (345,51) .. controls (339.48,51) and (335,46.52) .. (335,41) -- cycle ;
\draw  [fill={rgb, 255:red, 0; green, 0; blue, 0 }  ,fill opacity=1 ] (335,141) .. controls (335,135.48) and (339.48,131) .. (345,131) .. controls (350.52,131) and (355,135.48) .. (355,141) .. controls (355,146.52) and (350.52,151) .. (345,151) .. controls (339.48,151) and (335,146.52) .. (335,141) -- cycle ;
\draw  [fill={rgb, 255:red, 0; green, 0; blue, 0 }  ,fill opacity=1 ] (435,41) .. controls (435,35.48) and (439.48,31) .. (445,31) .. controls (450.52,31) and (455,35.48) .. (455,41) .. controls (455,46.52) and (450.52,51) .. (445,51) .. controls (439.48,51) and (435,46.52) .. (435,41) -- cycle ;
\draw  [fill={rgb, 255:red, 0; green, 0; blue, 0 }  ,fill opacity=1 ] (435,141) .. controls (435,135.48) and (439.48,131) .. (445,131) .. controls (450.52,131) and (455,135.48) .. (455,141) .. controls (455,146.52) and (450.52,151) .. (445,151) .. controls (439.48,151) and (435,146.52) .. (435,141) -- cycle ;
\draw  [fill={rgb, 255:red, 0; green, 0; blue, 0 }  ,fill opacity=1 ] (535,41) .. controls (535,35.48) and (539.48,31) .. (545,31) .. controls (550.52,31) and (555,35.48) .. (555,41) .. controls (555,46.52) and (550.52,51) .. (545,51) .. controls (539.48,51) and (535,46.52) .. (535,41) -- cycle ;
\draw  [fill={rgb, 255:red, 0; green, 0; blue, 0 }  ,fill opacity=1 ] (535,141) .. controls (535,135.48) and (539.48,131) .. (545,131) .. controls (550.52,131) and (555,135.48) .. (555,141) .. controls (555,146.52) and (550.52,151) .. (545,151) .. controls (539.48,151) and (535,146.52) .. (535,141) -- cycle ;

\draw (78,15.4) node [anchor=north west][inner sep=0.75pt]  [color={rgb, 255:red, 255; green, 0; blue, 0 }  ,opacity=1 ]  {$e_{1}$};
\draw (486,14.4-5) node [anchor=north west][inner sep=0.75pt]  [color={rgb, 255:red, 245; green, 166; blue, 35 }  ,opacity=1 ]  {$d_{1}$};
\draw (80,147.4) node [anchor=north west][inner sep=0.75pt]  [color={rgb, 255:red, 189; green, 16; blue, 224 }  ,opacity=1 ]  {$c_{2}$};
\draw (327,65.4) node [anchor=north west][inner sep=0.75pt]  [color={rgb, 255:red, 189; green, 16; blue, 224 }  ,opacity=1 ]  {$c_{2}$};
\draw (124,64.4) node [anchor=north west][inner sep=0.75pt]  [color={rgb, 255:red, 139; green, 87; blue, 42 }  ,opacity=1 ]  {$e_{2}$};
\draw (525,66.4) node [anchor=north west][inner sep=0.75pt]  [color={rgb, 255:red, 65; green, 117; blue, 5 }  ,opacity=1 ]  {$d_{2}$};
\draw (123,98.4) node [anchor=north west][inner sep=0.75pt]  [color={rgb, 255:red, 0; green, 23; blue, 255 }  ,opacity=1 ]  {$c_{1}$};
\draw (282,16.4) node [anchor=north west][inner sep=0.75pt]  [color={rgb, 255:red, 0; green, 23; blue, 255 }  ,opacity=1 ]  {$c_{1}$};
\draw (278,148.4) node [anchor=north west][inner sep=0.75pt]  [color={rgb, 255:red, 65; green, 117; blue, 5 }  ,opacity=1 ]  {$d_{2}$};
\draw (326,97.4) node [anchor=north west][inner sep=0.75pt]  [color={rgb, 255:red, 0; green, 23; blue, 255 }  ,opacity=1 ]  {$\textcolor[rgb]{0.96,0.65,0.14}{d_{1}}$};
\draw (490,148.4) node [anchor=north west][inner sep=0.75pt]  [color={rgb, 255:red, 139; green, 87; blue, 42 }  ,opacity=1 ]  {$e_{2}$};
\draw (524,98.4) node [anchor=north west][inner sep=0.75pt]  [color={rgb, 255:red, 255; green, 0; blue, 0 }  ,opacity=1 ]  {$e_{1}$};

\draw (35,10.4) node [anchor=north west][inner sep=0.75pt]    {$u_{1}$};
\draw (36,155.4) node [anchor=north west][inner sep=0.75pt]    {$u_{2}$};
\draw (239,11.34) node [anchor=north west][inner sep=0.75pt]    {$v_{1}$};
\draw (238,153.34) node [anchor=north west][inner sep=0.75pt]    {$v_{2}$};
\draw (437,10.34) node [anchor=north west][inner sep=0.75pt]    {$w_{1}$};
\draw (439,154.34) node [anchor=north west][inner sep=0.75pt]    {$w_{2}$};
\draw (136,11.34) node [anchor=north west][inner sep=0.75pt]    {$u_{3}$};
\draw (137,156.34) node [anchor=north west][inner sep=0.75pt]    {$u_{4}$};
\draw (336,10.34) node [anchor=north west][inner sep=0.75pt]    {$v_{3}$};
\draw (337,153.34) node [anchor=north west][inner sep=0.75pt]    {$v_{4}$};
\draw (537,11.34) node [anchor=north west][inner sep=0.75pt]    {$w_{3}$};
\draw (539,155.34) node [anchor=north west][inner sep=0.75pt]    {$w_{4}$};

\draw (189,13.4) node [anchor=north west][inner sep=0.75pt]  [color={rgb, 255:red, 155; green, 155; blue, 155 }  ,opacity=1 ]  {$P_{1}$};
\draw (186,151.4) node [anchor=north west][inner sep=0.75pt]  [color={rgb, 255:red, 155; green, 155; blue, 155 }  ,opacity=1 ]  {$P_{2}$};
\draw (386,12.34) node [anchor=north west][inner sep=0.75pt]  [color={rgb, 255:red, 155; green, 155; blue, 155 }  ,opacity=1 ]  {$P_{3}$};
\draw (383,150.34) node [anchor=north west][inner sep=0.75pt]  [color={rgb, 255:red, 155; green, 155; blue, 155 }  ,opacity=1 ]  {$P_{4}$};

\end{tikzpicture}
\caption{The Hamilton router of order $2$ we find for Lemma~\ref{lem:comparators exist}. All edges are directed from left to right. Note that the six horizontal edges have the same colours as the six diagonal edges; combined with $P_1,P_2,P_3,P_4$ these two sets of edges respectively give the paths $Q_{1,3},Q_{2,4}$ and $Q_{1,4},Q_{2,3}$.}
\label{fig:colouredcomparator}
\end{figure}
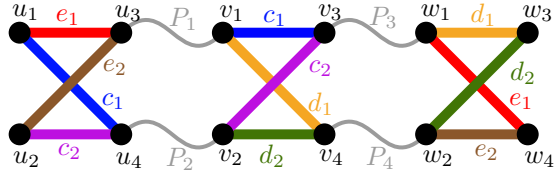

Indeed, the vertex-disjoint paths $Q_{1,3}:=u_1u_3P_1v_1v_3P_3w_1w_3$ and $Q_{2,4}:=u_2u_4P_2v_2v_4P_4w_2w_4$ are together rainbow, as are the vertex-disjoint paths $Q_{1,4}:=u_1u_4P_2v_2v_3P_3w_1w_4$ and $Q_{2,3}:=u_2u_3P_1v_1v_4P_4w_2w_3$, and we have that $V(Q_{1,4}) \cup V(Q_{2,3})=V(Q_{1,3}) \cup V(Q_{2,4})=V(\mathcal{C})$, $C(Q_{1,4}) \cup C(Q_{2,3})=C(Q_{1,3}) \cup C(Q_{2,4})=C(\mathcal{C})$, and, for each $i\in \{1,2\}$, $j\in\{3,4\}$ $Q_{i,j}$ is a $u_i,w_j$-path of length $9$. This yields a $\{u_1, u_2\}, \{w_3, w_4\}, C(\mathcal{C})$-Hamilton-router of depth 9 which avoids the vertices and colours in $\mathcal{J}$.
\end{proof}


\subsection{Hamilton routers}\label{sec:routers}

Here we build Hamilton routers of order $n$ by gluing together comparators (as depicted in Figure~\ref{fig:activated}).

\begin{defn}\label{defn:basic hamilton routers}
    We call an $A,B$-Hamilton-router $S$ a \textit{basic Hamilton router of order $n$} if $S$ can be constructed as follows. Start with a vertex set $V$ defined on a $4\times n$ grid, that is, $V:=\{v_{i,j}\,\colon (i,j)\in [4]\times [n]\}$. Let $A=\{v_{1,j}\,\colon j\in  [n]\}$ and $B=\{v_{4,j}\,\colon j\in  [n]\}$.
    \begin{itemize}
        \item For every odd $i\in [n-1]$, and for $A'=\{v_{1,i},v_{1,i+1}\}$ and $B'=\{v_{2,i},v_{2,i+1}\}$, add a graph $S^{\textrm{odd}}_i$ which is an $A',B'$-comparator of depth $9$.
        \item For every even $i\in [n-1]$, and for $A'=\{v_{3,i},v_{3,i+1}\}$ and $B'=\{v_{4,i},v_{4,i+1}\}$, add a graph $S^{\textrm{even}}_i$ which is an $A',B'$-comparator of depth $9$.
        \item For every $i\in [n]$, add a path from $v_{2,i}$ to $v_{3,i}$ of length $3$.
        \item Add an edge from $v_{3,1}$ to $v_{4,1}$. If $n$ is even, add an edge from $v_{3,n}$ to $v_{4,n}$, and otherwise add an edge from $v_{1,n}$ to $v_{2,n}$.
    \end{itemize}
\end{defn}

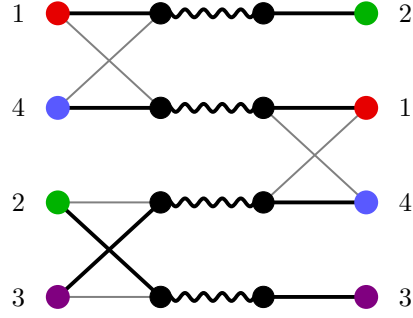
\begin{figure}
    \centering

\begin{tikzpicture}[
    x=1.7cm,
    y=1.25cm,
    straight/.style={line width=1.45pt, line cap=round},
    thinline/.style={line width=0.75pt, line cap=round,black!50},
    wavy/.style={
        line width=1.45pt,
        decorate,
        decoration={snake, amplitude=1.5pt, segment length=7pt},
        line cap=round
    },
    blackvertex/.style={circle, fill=black, draw=black, inner sep=0pt, minimum size=8pt},
    redvertex/.style={circle, fill=red!90!black, draw=none, inner sep=0pt, minimum size=9pt},
    bluevertex/.style={circle, fill=blue!65, draw=none, inner sep=0pt, minimum size=9pt},
    greenvertex/.style={circle, fill=green!70!black, draw=none, inner sep=0pt, minimum size=9pt},
    purplevertex/.style={circle, fill=violet, draw=none, inner sep=0pt, minimum size=9pt}
]

\coordinate (L1) at (0,3);
\coordinate (L2) at (0,2);
\coordinate (L3) at (0,1);
\coordinate (L4) at (0,0);

\coordinate (A1) at (0.8,3);
\coordinate (A2) at (0.8,2);
\coordinate (A3) at (0.8,1);
\coordinate (A4) at (0.8,0);

\coordinate (B1) at (1.6,3);
\coordinate (B2) at (1.6,2);
\coordinate (B3) at (1.6,1);
\coordinate (B4) at (1.6,0);

\coordinate (R1) at (2.4,3);
\coordinate (R2) at (2.4,2);
\coordinate (R3) at (2.4,1);
\coordinate (R4) at (2.4,0);

\draw[straight] (L1) -- (A1);
\draw[straight] (L2) -- (A2);
\draw[thinline] (L1) -- (A2); 
\draw[thinline] (L2) -- (A1); 

\draw[thinline] (L3) -- (A3); 
\draw[thinline] (L4) -- (A4); 
\draw[straight] (L3) -- (A4);
\draw[straight] (L4) -- (A3);

\draw[wavy] (A1) -- (B1);
\draw[wavy] (A2) -- (B2);
\draw[wavy] (A3) -- (B3);
\draw[wavy] (A4) -- (B4);

\draw[straight] (B1) -- (R1);
\draw[straight] (B4) -- (R4);

\draw[straight] (B2) -- (R2); 
\draw[straight] (B3) -- (R3); 
\draw[thinline] (B2) -- (R3); 
\draw[thinline] (B3) -- (R2); 

\node[redvertex]    at (L1) {};
\node[bluevertex]   at (L2) {};
\node[greenvertex]  at (L3) {};
\node[purplevertex] at (L4) {};

\node[greenvertex]  at (R1) {};
\node[redvertex]    at (R2) {};
\node[bluevertex]   at (R3) {};
\node[purplevertex] at (R4) {};

\node[blackvertex] at (A1) {};
\node[blackvertex] at (A2) {};
\node[blackvertex] at (A3) {};
\node[blackvertex] at (A4) {};

\node[blackvertex] at (B1) {};
\node[blackvertex] at (B2) {};
\node[blackvertex] at (B3) {};
\node[blackvertex] at (B4) {};

\node[anchor=east] at (-0.18,3) {$1$};
\node[anchor=east] at (-0.18,2) {$4$};
\node[anchor=east] at (-0.18,1) {$2$};
\node[anchor=east] at (-0.18,0) {$3$};

\node[anchor=west] at (2.58,3) {$2$};
\node[anchor=west] at (2.58,2) {$1$};
\node[anchor=west] at (2.58,1) {$4$};
\node[anchor=west] at (2.58,0) {$3$};

\end{tikzpicture}
    \caption{A basic Hamilton router of order $4$ where (for simplicity) the comparators are depicted as $4$-cycles (and all edges are directed from left to right). $A$ and $B$ are the leftmost and rightmost columns of vertices, respectively. An example bijection between $A$ and $B$ is indicated by the colour/number of the vertices. Choices of path systems for each comparator are indicated with thick lines, chosen so that identifying the vertex pairs of the same colour/number forms a cycle in the thick lines (passing through $1,2,3,4$ in order). }
    \label{fig:activated}
\end{figure}

\begin{prop}\label{prop:routers are good} Basic Hamilton routers are well defined, that is, every (directed) graph constructed as in Definition~\ref{defn:basic hamilton routers} is an $A,B$-Hamilton-router.
\end{prop}
\begin{proof}
\par Let $S$ be the graph described in Definition~\ref{defn:basic hamilton routers}. Take any bijection $\phi\colon A\to B$, and let $M_{\phi}$ be the corresponding perfect matching from $B$ to $A$. First observe that $S$ contains a path partition $\mathcal{P}_\mathrm{id}$ of $n$ paths where, for every $(v_{1,j}, v_{4,j})$, $j \in [n]$, there is a path with those endpoints. This is simply by invoking the definition of each of the component comparators with the bijection that keeps the second subscript of the vertices the same. Observe that $\mathcal{P}_\mathrm{id}\cup M_{\phi}$ is a collection of disjoint cycles partitioning $V(S)$. If there is exactly one cycle, then we have the desired Hamilton cycle. So let us suppose that we have cycles $C_1,\ldots, C_k$ for $k\geq 2$. For each $i\in [k]$, let $I(C_i)\subseteq [n]$ denote the set of $j\in [n]$ such that there is some $\ell\in [4]$ with $v_{\ell, j}\in V(C_i)$ --- these are simply the second subscripts (or the $y$-coordinates along the grid) the cycle $C_i$ travels through. Thus, the sets $I(C_i)$, $i\in [k]$, partition $[n]$. We argue that there is a subset of comparators whose path partitions we can `toggle' to merge all of these $k$-many cycles into a Hamilton cycle.
\par Consider an auxiliary graph $G$ on $[n]$ with an edge $\{i,j\}$ if, for some $\ell\in [4]$, there is an $A',B'$-comparator of $A'=\{v_{\ell,i},v_{\ell,j}\}$ and $B'=\{v_{\ell+1,i},v_{\ell+1,j}\}$. Observe that this auxiliary graph is simply a path $1-2-3-\cdots - n$ by our construction. In particular, $G$ is connected. Now we consider the graph $G'$ obtained from $G$ by contracting $I(C_i)$ to a single vertex for each $i \in [k]$. The graph $G'$ must also be connected, as vertex contractions preserve connectivity. Take a spanning tree $T$ of $G'$. Note that each edge of $G'$ corresponds to a set of edges in $G$, and no edge in $G$ appears in more than one of these sets (as this would imply that there are at least two indices $i,j \in [k]$ for which some vertex from $V(G)$ is in $C_i$ and $C_j$, contradicting that these cycles are vertex disjoint). So, for each edge $e$ of $T$, we can pick arbitrarily a distinct edge $e_G$ of $G$ that contracts to $e$ in $G'$. Note that the edges $e_G$, $e\in E(T)$, correspond to a set of $e(T)$ distinct comparators. 

Now, by toggling the path partitions for a comparator corresponding to an edge $e_G$, we can merge the two cycles that travel through the endpoints of $e_G$. Doing this for all such $e_G$, $e\in E(T)$, therefore yields the desired Hamilton cycle. \end{proof}


\subsection{Routers in typical coloured digraphs}
We can now combine the existence of efficient routing networks (Proposition~\ref{prop:routers are good}), comparators (Lemma~\ref{lem:comparators exist}), and short connecting paths (Lemma~\ref{Connecting_lemma}) to prove the following key result.
\begin{lemma}\label{lem: routers in digraphs}
    Let $1/n, \, \eps\ll 1$, let $\eps \llpoly  q \llpoly w \llpoly p\leq 1$, and let $D$ be a properly coloured $(n,\eps)$-typical digraph. Let $X$ be a $p$-random subset of $V(D)$ and let $C$ be a $p$-random subset of $C(D)$, sampled independently. Then, with high probability, the following holds.
    \par Let $A,B\subseteq V(D)$ be disjoint subsets of size $qn$, and let $\mathcal{J}$ be a forbidden set of at most $w n$ vertices and colours which avoids $A \cup B$.  Let $m\in \mathbb{N}$ with $40qn\leq m\leq 100qn$. Then, there exists an $A,B,C'$-Hamilton-router $S$ with $V(S)\subseteq A\cup B\cup (X\setminus \mathcal{J})$, $C'\subseteq C\setminus \mathcal{J}$, and  $|V(S)\setminus (A\cup B)|=m$.
\end{lemma}
\begin{proof}
    By invoking Lemma~\ref{lem:comparators exist} repeatedly, we can find a family $\mathcal{F}$ in $D$ of vertex- and colour-disjoint comparators of depth $9$ with $|\mathcal{F}|\geq n/10^{13}$. Similarly, by Lemma~\ref{Connecting_lemma}, for each distinct $x,y\in V(D)$ we can find a family $\mathcal{F}_{xy}$ of rainbow $x,y$-paths of length $3$ that are colour-disjoint and internally vertex-disjoint  with $|\mathcal{F}_{xy}|\geq n/10^{5}$. Consider the $p$-random sets $X$ and $C$. By Chernoff's bound (Lemma~\ref{chernoff-bound}), with probability at least $1-e^{-\sqrt{n}}$, there are at least $p^{100}n/10^{15}$ members of $\mathcal{F}$ whose vertices are in $X$ and colours are in $C$. Similarly, by Chernoff's bound, for each distinct $x,y\in V(D)$, with probability at least $1-e^{-\sqrt{n}}$, there exist at least $p^{10}n/10^{6}$ members of $\mathcal{F}_{xy}$ whose internal vertices are in $X$ and colours are in $C$. By a union bound, all of the $n(n-1)+1$ events mentioned hold together with high probability. 
    Fix, then, some choice of $X$ and $C$ for which these events all hold.
    \par
    Let $A, B\subset V(D)$ be disjoint subsets of size $qn$ and let $\mathcal{J}$ have size at most $wn$. Let $S'$ be some basic (uncoloured, directed) $X_A,X_B$-Hamilton-router of order $qn$, whose existence is guaranteed by Proposition~\ref{prop:routers are good}, for some disjoint $X_A, X_B \subseteq X \setminus (\mathcal{J} \cup A \cup B)$ and observe that $S'$ has $4qn+2qn+(qn-1)(2\cdot 8)=22qn-16$ vertices.
    Take an arbitrary connecting path of length $3$ in $S'$ and subdivide it the appropriate number of times to obtain $\bar{S}$, which is an $X_A,X_B$-Hamilton-router of order $qn$ with precisely $m-4qn$ vertices. (It is clear that subdividing a connecting path does not break the Hamilton router property, and requires subdividing $m-26qn+16 \leq 75qn$ times.) We will now `embed' $\bar{S}$ on $D[X,C]$, disjointly from $\mathcal{J} \cup A \cup B$, mapping the comparators in $\bar{S}$ to comparators in $D$ and embedding the paths between them, so that the embedding of $\bar{S}$ constitutes a coloured Hamilton router, i.e., an $X_A,X_B,\bar{C}$-Hamilton-router for some $\bar{C}\subseteq C\setminus \mathcal{J}$.
    \par As $p^{100}n/10^{15} - wn - 2qn\gg qn$, we may arbitrarily embed the $(qn-1)$ comparators in $\bar{S}$ using coloured comparators in $D[X,C]$ that are vertex- and colour-disjoint from $\mathcal{J} \cup A \cup B$. We may then easily extend the embedding to include the disjoint edges $v_{3,1}v_{4,1}$ and either $v_{3,n}v_{4,n}$ or $v_{1,n}v_{2,n}$ as appropriate (depending on the parity of $qn$). It remains to find the appropriate connecting paths between the various already embedded comparators, edges and vertices in $A \cup B$. We can achieve this by iteratively invoking the guaranteed number of appropriate connecting paths of length $3$, which exist between every pair of vertices in $D$, as $p^{10}n/10^{6}- 2wn \gg 3qn$. To embed the unique connecting path in $\bar{S}$ of length $\ell := m-26qn+19\leq 100qn$, we may start by finding a rainbow path of length $\ell-3$ starting from the corresponding endpoint (this can be achieved greedily, avoiding the used colours and vertices as well as those in $\mathcal{J}$, as $D[X,C]$ has minimum degree at least $p^{10}n/10^{6}\gg 3wn$), and then invoking the abundance of remaining colour- and internally-vertex-disjoint connecting paths of length $3$ one last time. Let $\phi$ be the resulting embedding of $\bar{S}$.
    \par Then, we can extend $\phi$ to find the desired $A,B, C'$-Hamilton router $S$ of order $qn$ as follows. We shall pair up the vertices in $A$ and $X_A$ and the vertices in $X_B$ and $B$, and find vertex- and colour-disjoint rainbow paths of length three between each of the pairs, such that the colours are all distinct and the set of colours used, $\hat{C}$ say, is a subset of colours from $C \setminus (\mathcal{J} \cup C(\phi(\bar{S})))$, and the internal vertices are all distinct and the set of internal vertices used, $\hat{V}$ say, lies in $X \setminus (\mathcal{J} \cup V(\phi(\bar{S})))$. Setting $C'={C(\phi(\bar{S}))} \cup \hat{C}$ and $V(S)=A\cup B\cup V(\phi(\bar{S})) \cup \hat{V}$ yields our desired $A,B,C'$-Hamilton router.
\end{proof}


\section{Putting everything together to prove Theorem~\ref{thm:furthermore}}\label{sec:puttingtogether}

Our goal now is to prove Theorem~\ref{thm:furthermore}, stated below in a form that aids with the stability analysis performed in the next section to prove Theorem~\ref{thm:mainintro}. To prove Theorem~\ref{thm:furthermore}, we use the outline in Section~\ref{sec:overview}. In that outline, we took a digraph $\bar{D}$ and, setting aside some Hamilton-router, randomly partitioned its remaining vertices as $V_1\cup \dots\cup V_k$ for some $k$ and remaining colours as $C_1\cup\dots\cup C_{k-1}$. In order to use Theorem~\ref{thm:BPW}, we wish to have $k=O(1)$, and thus we will use a Hamilton router of order $\Omega(n)$. There is, however, a problem. Essentially, by Lemma~\ref{lem: routers in digraphs}, we can find such a router but this will disrupt the randomness of any resulting partition with an error term of $\Omega(1)$, whereas to apply Theorem~\ref{thm:BPW} we can only tolerate an error term $\eps\ll \log^{-1}n$.

To get around this, instead, we will partition the vertices/colours outside of the router into twice as many sets, $V_1\cup \dots\cup V_k\cup V_1'\cup\dots \cup V_k'$ and $C_1\cup \dots\cup C_{k-1}\cup C_1'\cup\dots \cup C_{k-1}'$, respectively (see \eqref{eq:vxpartition} and \eqref{eq:colpartition}). Then, $V_1,\dots,V_k,C_1,\dots,C_{k-1}$ can be random sets, while $V_1',\dots,V'_k,C'_1,\dots,C_{k-1}'$ incorporate the disruption from finding the router. Using Theorem~\ref{thm:RSBcoveringstephyper}, we can then find large rainbow matchings in $\bar{D}[V'_i,V_{i+1}',C'_i]$ for each $i\in [k-1]$. The union of these matchings forms rainbow paths of length $k-1$. We can then distribute the remaining vertices among the sets $V_1,\dots,V_k$. Afterwards, using Theorem~\ref{thm:BPW}, we can find very large rainbow matchings in $\bar{D}[V_i,V_{i+1},C_i]$ for each $i\in [k-1]$. Finally, for example in case \ref{case:cycle} of Theorem~\ref{thm:furthermore}, the union of all of these matchings will then be a set of rainbow paths that we can connect together into a rainbow Hamilton cycle using the router.
Let us recap a key point: while the matchings found via Theorem~\ref{thm:RSBcoveringstephyper} are not as large, the theorem tolerates a larger deviation from randomly chosen sets and produces large enough matchings to bring our error term within the scope of Theorem~\ref{thm:BPW}.

\begin{theorem}\label{thm:furthermore}
Let $1/n, \, \eps\llpoly\log^{-1}n$, and let $D$ be a properly coloured $(n,\eps)$-typical digraph\Footnote{Recall Definition~\ref{defn:digraph}.} with $n$ vertices and $m \geq n-1$ colours.
Then, the following hold.
\begin{enumerate}[label = {({\arabic{enumi}}})]
    \item\label{case:n-1} If $m=n-1$, then $D$ contains a directed rainbow path on $n-1$ vertices.
    \item\label{case:n} If $m \geq n$, then $D$ contains a directed rainbow Hamilton path.
    \item\label{case:cycle} For all $m \geq n-1$ and $V' \subseteq V(D)$ with $|V'|= \min\{n, m-100\}$, $D$ contains a directed rainbow cycle covering $V'$.
\end{enumerate}
Furthermore, for each distinct $x,y\in V(D)$, if we add $xy$ to $D$ with a colour not in $C(D)$, we can find in $D$ a rainbow Hamilton path using $xy$ and a rainbow directed cycle of length $n-100$ using $xy$.
\end{theorem}

\begin{proof}
The proofs of \ref{case:n-1}--\ref{case:cycle} only diverge towards their end. For \ref{case:cycle}, let $V' \subseteq V(D)$  have size $\min\{n, m-100\}$, let $D'=D[V']$ and let $n'=|V'|\geq n-101$. Note that, then, $D'$ is $(n', 2\eps)$-typical and properly edge-coloured with at least $n'+100$ colours.
To prove \ref{case:n-1} or \ref{case:n}, let $(\bar{n}, \bar{D})=(n,D)$, and to prove \ref{case:cycle} let $(\bar{n}, \bar{D})=(n',D')$. 

Take $q,w,p$ with $1/{n}\ll q\ll w \ll p\leq 10^{-5}$ and so that $t:=(1-p)/q$ is an integer. Further, let $\eta$ and $\tilde{\eps}$ satisfy
$$\eps \llpoly \eta \llpoly \tilde{\eps} \llpoly\log^{-1}{n}. $$
Then, let
\begin{equation}\label{eq:vxpartition}
R_V, V_1, V_1', V_2, V_2',\dots, V_{t}, V_{t}'
\end{equation} be a random partition of $V(\bar{D})$ for which $R_V$ is $p$-random and the other sets are $q/2$-random (using that $2(q/2)t + p =1$). In particular, any pair of these random sets has the distribution of a pair of disjoint random sets. Let
\begin{equation}\label{eq:colpartition}
R_C,C_1, C_1',C_2,C_2',\dots, C_{t-1},C_{t-1}', C_{\emptyset}
\end{equation}
be a random partition of $C(\bar{D})$ for which $R_C$ is $p$-random, each set $C_i,C_i'$, $i\in [t-1]$, is $q/2$-random, and $C_\emptyset$ is $q$-random.
\par Let $M$ be a value closest to $50q\bar{n}$ so that $\ell:=(\bar{n}-M)/t$ is an integer, noting that $M=50q\bar{n}\pm t$. Let $\tilde q=\ell/\bar{n}$ and $\bar q=(\ell-q\bar{n}/2)/\bar{n}$. We will show that the following inequalities hold.
\begin{equation}\label{eq:inequalities}
q + p/10t\leq \tilde q\leq 1.01q\;\;\;\;\;\;\;\; 49\tilde q \bar{n}\leq M \leq 51\tilde q \bar{n}\;\;\;\;\;\;\;\; q/2\leq \bar q \leq 1.02q/2
\end{equation}
Indeed, first note that $\tilde q=\ell/\bar{n}\leq 1/t=q/(1-p)\leq 1.01q$, where we used that $p\leq 10^{-5}$. Furthermore, $\tilde q=\ell/\bar{n}=(1-M/\bar{n})/t\geq (1-50q-t/\bar{n})/t=q+(p-50q-t/\bar{n})/t\geq q+p/2t$ (as $p\gg q$), and thus the first inequality at \eqref{eq:inequalities} holds. Next, the second inequality at \eqref{eq:inequalities} holds as $M=50q\bar{n}\pm t$, $t\leq 1/q$ and $1/n\ll q$. Finally, using $\ell=(\bar{n}-M)/t$, we have $\bar{q}=(1/t)-(q/2)\pm (M/(t\bar{n}))=(q/2)+(pq/(1-p))\pm (100q/t)$. Thus, the third inequality at \eqref{eq:inequalities} holds as $t\geq 1/(2q)$ and $1/n\ll q\ll p$.

We will now show that the following properties each hold with high probability, and therefore they hold simultaneously with high probability.
\stepcounter{propcounter}
\begin{enumerate}[label = {\textbf{\Alph{propcounter}\arabic{enumi}}}]
\item\label{chernoff} For each $c\in \{p,q/2,q\}$, every $c$-random set at \eqref{eq:vxpartition}  has $c\bar{n} \pm \bar{n}^{0.6}$ elements, and every $c$-random set at \eqref{eq:colpartition} has $c|C(\bar{D})|\pm \bar{n}^{0.6}=(1\pm 2\eps)c\bar{n}$ elements.
\item\label{sorting} For any disjoint subsets $A,B\subseteq V(\bar{D})$, each of size $\tilde{q}\bar{n}$ and any set $\mathcal{J}$ of at most $w \bar{n}$ colours and vertices not in $A \cup B$, there exists an $A,B,C'$-Hamilton-router $S$ with $V(S)\subseteq A\cup B\cup (R_V\setminus \mathcal{J})$, $C'\subseteq R_C\setminus \mathcal{J}$, and  $|V(S)\setminus (A\cup B)|=M$.

\item\label{typical1} For each $1\leq i<j\leq t$ and $k\in [t-1]$, and any disjoint $X',Y'\subseteq V(\bar{D})$ and $C'\subseteq C(\bar{D})$ satisfying $|V_i\triangle X'|,|V_j\triangle Y'|,|C_k\triangle C'|\leq \tilde{\eps} \bar{n}$, \ref{RBS}, \ref{BPW} and \ref{SPC} hold in Theorem~\ref{thm:BPW}.

\item\label{shrinkfirst}  For every $1\leq i<j\leq t$ and $k\in [t-1]$, disjoint sets $\bar  V_i',\bar V_j'\subset V(\bar{D})$ and $\bar C_k'\subseteq C(\bar{D})$ with $|\bar  V_i'|=\bar{q}\bar{n}\pm \eta \bar{n}$, $|\bar V_j'|=\bar{q}\bar{n}\pm \eta \bar{n}$, $|\bar{C}_k'|=\bar{q}\bar{n}\pm \eta \bar{n}$ and $|V_i'\setminus \bar  V_i'|,|V_j'\setminus \bar  V_j'|, |C_k'\setminus \bar  C_k'|\leq {\eta} \bar{n}$, $\bar{D}[\bar  V_i', \bar V_j', \bar C_k']$ has a rainbow matching of size $\bar q \bar{n} -10\eta \bar{n}$.
\end{enumerate}

Indeed, first note that \ref{chernoff} holds with high probability by a simple application of Chernoff's bound (Lemma~\ref{chernoff-bound}) and a union bound. Furthermore, \ref{sorting} holds with high probability by \eqref{eq:inequalities} and Lemma~\ref{lem: routers in digraphs} applied with
$(R_V,R_C,\tilde q,w,p,\eps)$ in place of $(X,C,q,w,p, \eps)$. Next, \ref{typical1} holds for each such $1\leq i<j\leq t$ and $k\in [t-1]$ with high probability (as $\bar{n}\to\infty$) by Theorem~\ref{thm:BPW} applied with $(\bar{n},q/2,\tilde{\eps},V_i,V_j,C_k)$ in place of $(n,p,\eps,X,Y,C)$, using that $\tilde{\eps}\llpoly \log^{-1}\bar{n}$ and that the $(\bar{n}, \eps)$ or $(\bar{n}, 2\eps)$-typical digraph $\bar{D}$ we started with is also $(\bar{n}, \tilde{\eps})$-typical as $\eps \llpoly \tilde{\eps}$.
Thus, as $1/\bar{n}\ll 1/t$, by a union bound \ref{typical1} holds with high probability. Finally, we will show that \ref{shrinkfirst} holds with high probability by Theorem~\ref{thm:RSBcoveringstephyper}. For this, consider an auxiliary tripartite hypergraph $\mathcal{H}$
on parts $A,B,C$ where $A$ and $B$ are copies of $V(\bar{D})$, $C$ is $C(\bar{D})$, and we have an edge $\{a,b,c\}$ exactly when there is an edge $ab$ with colour $c$ in $\bar{D}$. As $\bar{D}$ is $(\bar{n},2\eps)$-typical, this forms an $(\bar{n},1,5\eps)$-typical hypergraph.
Observe that, for each $1\leq i<j\leq t$ and $k\in [t-1]$, the random bipartite edge-coloured digraph $\bar{D}[V_i', V_j', C_k']$ corresponds to an induced random sub-hypergraph $\mathcal{H}[V_i', V_j', C_k']$ where the random sets have parameter $q/2$.
Therefore, by Theorem~\ref{thm:RSBcoveringstephyper} with $(q/2,q/2,q/2,\bar{q},\eta,V_i', V_j', C_k')$ in place of $(q_A,q_B,q_C,{q},\eta,A',B',C')$ (using that $2\bar{q}/3\leq q/2\leq \bar{q}$ holds by the third inequality at \eqref{eq:inequalities}), with high probability we have that,
for any disjoint vertex sets $\bar  V_i', \bar V_j'$ and any colour set $\bar C_k'$, with $V_i'\subset \bar V_i'$, $V_j'\subset \bar V_j'$, $C_k'\subset \bar C_k'$ and $|\bar V_i'|=|\bar V_j'|=|\bar{C}'_k|=\bar{q}\bar{n}$,
there is a matching of size $\bar q \bar{n} -\eta \bar{n}$ in $\bar{D}[\bar  V_i', \bar V_j', \bar C_k']$. Note that, furthermore, by applying this to $(\bar{V}_i',\bar{V}_j',\bar{C}_k')$ with up to $3\eta \bar{n}$ elements moved in or out of each set, if this holds, then we have that \ref{shrinkfirst} holds for $i,j$ and $k$.
As this holds with high probability for each $1\leq i<j\leq t$ and $k\in [t-1]$ as $\bar{n}\to \infty$, and $1/n\ll 1/t$, a union bound implies that \ref{shrinkfirst} holds with high probability.

Therefore, we can now fix an outcome of our random partitions so that \ref{chernoff}--\ref{shrinkfirst} hold. Move some vertices from $R_V$ to $V_1'$ and $V_t'$ so that $V_1\cup V_1'$ and $V_t\cup V_t'$ both have size exactly $\ell$, where $\ell\leq \bar{n}/t\leq q\bar{n}/(1-p)\leq 2q\bar{n}$ and
\begin{equation}\label{eq:ellbound}
\ell=\frac{\bar{n}-M}{t}=\frac{q(\bar{n}-M)}{1-p}\overset{\eqref{eq:inequalities}}\geq q\bar{n}+2n^{0.6}.
\end{equation}
Thus, this can be achieved by \ref{chernoff}, as $q\ll p$, by moving at most $2q\bar{n}\ll w\bar{n}$ elements. By \ref{sorting}, we may find a $(V_t\cup V_t', V_1\cup V_1', C')$-Hamilton-router $S$, such that $V(S)\setminus (V_t\cup V_t'\cup V_1\cup V_1')\subset R_V$ (excluding the $\ll w\bar{n}$ vertices moved to $V_1'$ and $V_t'$), such that $C'\subseteq R_C$ and $|V(S)\setminus (V_t\cup V_t'\cup V_1\cup V_1')|=M$. 
\par Take all the remaining vertices in $R_V$ outside of $V(S)$ and distribute them among the sets $V_i'$, $2\leq i\leq t-1$, so that $|V_i\cup V_i'|=\ell$ for each $2\leq i\leq t-1$. This is possible as $\bar{n}=M+t\ell$ and, by \ref{chernoff} and \eqref{eq:ellbound}, $|V_i\cup V_i'|\leq q\bar{n}+2\bar{n}^{0.6}\leq \ell$. Furthermore, for each $i\in [t]$, we have by \ref{chernoff} that $|V_i'|=\ell-|V_i|=\bar{q}\bar{n}\pm n^{0.8}$.

\begin{figure}[h]
    \hspace{-1cm}\begin{minipage}{1.1\textwidth}\begin{minipage}[t]{0.5cm}\textbf{a)}\end{minipage}\begin{minipage}[t]{0.33\textwidth}\includegraphics[scale=0.6,trim={3.9cm 0cm 4cm 0cm},clip,valign=t]{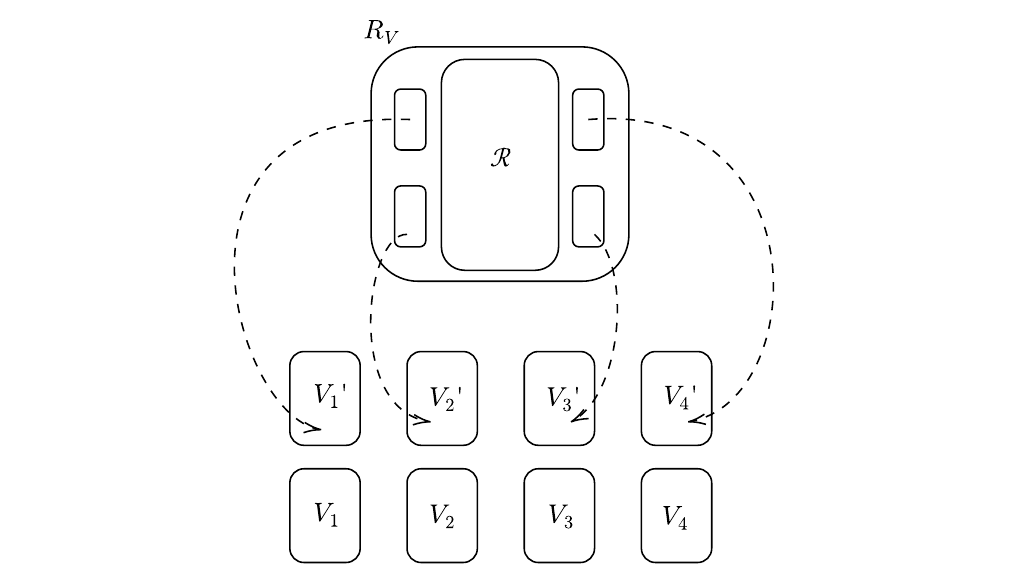}\end{minipage}\;\begin{minipage}[t]{0.5cm}\textbf{b)}\!\!\!\end{minipage}\includegraphics[scale=0.6,trim={4.1cm 0cm 4.1cm 0cm},clip,valign=t]{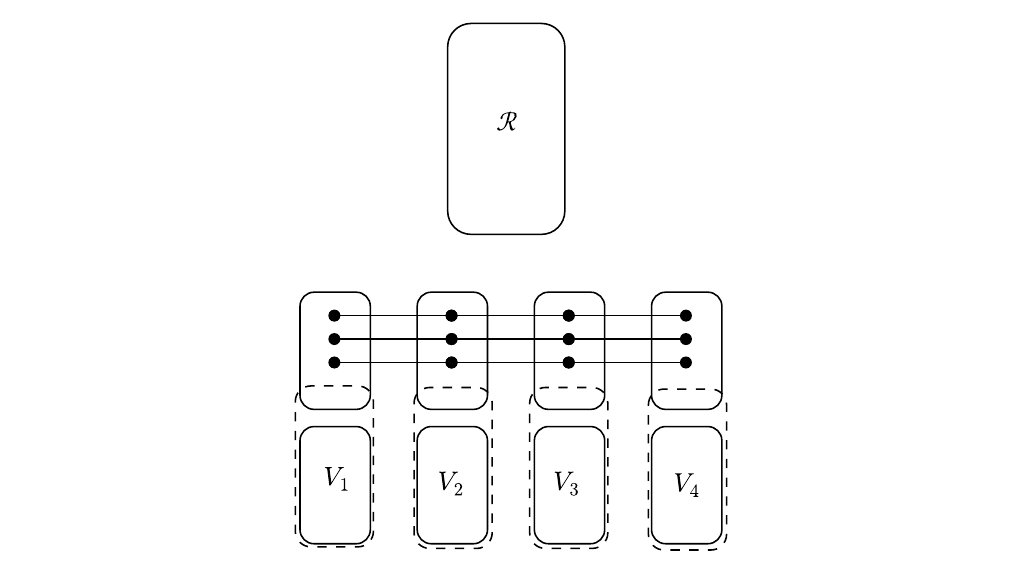}\begin{minipage}[t]{0.5cm}\textbf{c)}\end{minipage}\includegraphics[scale=0.6,trim={3.9cm 1cm 4cm 0cm},clip,valign=t]{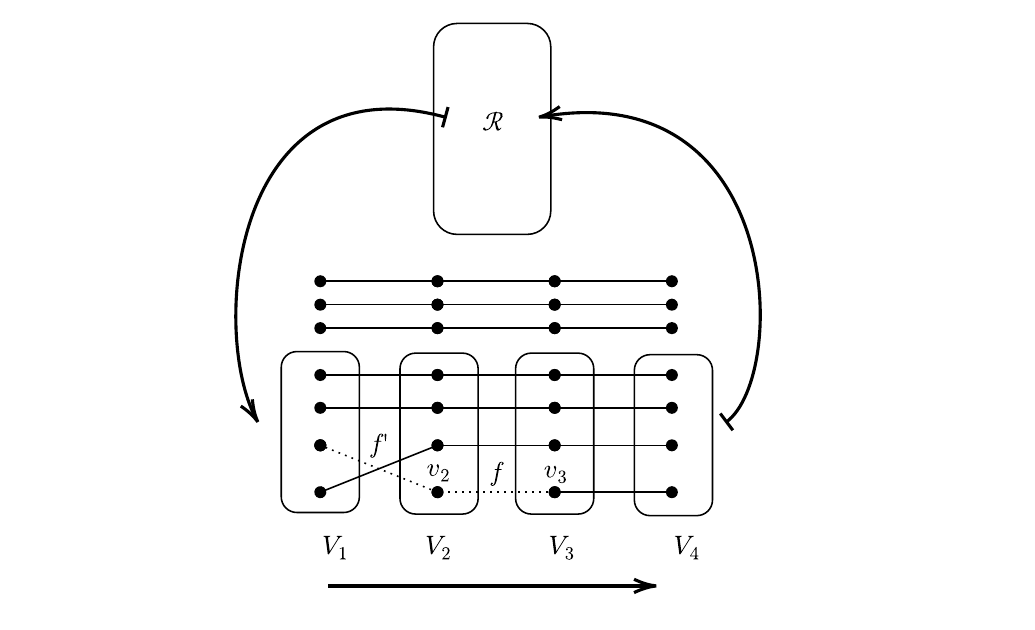}
    \end{minipage}
    \caption{Steps in the proof of Theorem~\ref{thm:furthermore}. \textbf{a)} The vertex partition at \eqref{eq:vxpartition} and the reallocation steps illustrated with $t=4$. \textbf{b)} After almost perfect matchings are found between the sets $V_i'$ and $V_{i+1}'$, $i\in [t-1]$, the leftover in each $V_i'$ is transferred to $V_i$. \textbf{c)} The final steps of the proof, in case \ref{case:n-1}. If $f$ and $f'$ are present, then there is a perfect matching between $V_i\cup V_i'$ and $V_{i+1}\cup V_{i+1}'$ for each $i\in [t-1]$. This induces a spanning linear forest on $\bigcup_{i\in [t]} (V_i\cup V_i')$ with endpoints in $V_1\cup V_1'$ and $V_t\cup V_t'$. Identifying the endpoints of the same paths, the routing network $S$ can merge these paths into a single (Hamilton) cycle. Deleting $f$ and $f'$ yields a path on $n-1$ vertices, as $f$ and $f'$ are both incident to $v_2$.}
    \label{fig:overallproof}
\end{figure}

\par Working analogously with the colour partition, using \ref{chernoff}, reallocate the colours in $C_{\emptyset}$ and $R_C\setminus C(S)$ among the sets $C_i'$, $i\in [t-1]$, so that $|C_i\cup C_i'|=\ell+100$ for each $2\leq i\leq t-1$ and $|C_1\cup C_1'|\geq \ell -200t$. This is possible by \ref{chernoff} and \eqref{eq:ellbound}, and as $|C(S)|=M + \ell$ and $m-(M + \ell)-(t-2)(\ell+100) = m-\bar{n}+\ell-100(t-2)$. Note that, as $|C(\bar{D})|\leq (1+\eps)\bar{n}$, $|C_1\cup C'_1|\leq \ell+\eps n$. Then, for each $2\leq i\leq t-1$, we have by \ref{chernoff} that $|C_i'|=\ell+100-|C_i|= \bar{q}n\pm 2\eps \bar{n}$.

\par We can now invoke \ref{shrinkfirst} for $\bar{D}[V_{i}', V_{i+1}', C_i']$ for each $i\in [t-1]$. 
Together, the rainbow matchings found give a rainbow directed path forest in which at most $20(t-1)\eta \bar{n}$ of the vertices in $V_{t-1}'\cup V_{t-2}'\cup \dots\cup V_2'$ do not have in- and out-degree 1. Therefore, letting $F$ be the forest of those paths with length $t-1$, for each $i\in [t]$, all but at most $20(t-1)\eta\bar{n}$ vertices in $V_i'$ lie in $F$. 
Similarly, for each $i\in [t-1]$, all but
at most $20(t-1)\eta\bar{n}$ colours in $C'_i$ are used on $F$. Let us denote by $r$ the number of paths in $F$.

For each $i\in [t]$, move the vertices in $V_i'\setminus V(F)$ to $V_i$ and, for each $j\in [t-1]$, move the colours in $C_j'\setminus C(F)$ to $C_j$. This ensures that each set $V_i$ and $C_j$, $i\in [t]$ and $j\in [t-1]$, has symmetric difference at most $20(t-1)\eta \bar{n}\ll \tilde{\eps} \bar{n}$ from its purely random variant sampled in the beginning.
Note, furthermore, that, as we had $|V_i'\cup V_i|=\ell$ for each $i\in [t]$ before this move, we now have that $|V_i|=|V_j|$ for each $i,j\in [t]$. Similarly, for each $2\leq i\leq t-1$, as we had $|C_i'\cup C_i|= |V_i'\cup V_i|+100$, we now have that $|C_i|=|V_i|+100$.
\par For convenience, let us introduce $s:= \ell -r$, noting that with the exchanges performed thus far, we have that $|V_i|=s$ for each $i\in [t]$.
\par By \ref{typical1} (in particular, using \ref{SPC} from Theorem~\ref{thm:BPW}), for each $2\leq i \leq t-1$ there is a perfect rainbow matching in $\bar{D}[V_i, V_{i+1}, C_i]$ from $V_{i}$ into $V_{i+1}$. Let $F'$ be the rainbow path forest formed by the union of all these matchings. Note that, for each $2\leq i\leq t-1$, there are exactly $100$ colours in $C_i$ not used on $F'$. Move all of these unused colours into $C_1$. Note that, now,
\begin{align}
|C_1|
    &= m-|C(S)|-|C(F)|-(t-2)s \notag\\
    &= m-(M+\ell)-r(t-1)-(t-2)s \notag\\
    &= m-M-(t-1)\ell-r \notag\\
    &= m+(\ell-r)-(M+t\ell) \notag\\
    &= m+s-\bar n.\notag
\end{align}

\par Note also that $V_1$ and $V_2$ both have precisely $s$ vertices whose in and out degrees are not both $1$ in $F'$.
\par We now find a perfect rainbow matching in $\bar{D}[V_1,V_2,C_1]$ from $V_1$ to $V_2$, perhaps including up to two dummy edges, $f$ and $f'$. Which result we invoke from \ref{typical1} to do this, and how many dummy edges we use, varies depending on which of \ref{case:n-1}--\ref{case:cycle} we are proving.

In case \ref{case:n-1}, we have  $(\bar{n}, \bar{D})=(n, D)$ and $m\geq n-1$, and thus $|C_1|\geq s-1$. Arbitrarily, pick an edge $(v_2,v_3)\in V_2\times V_3$ from $F'$, delete this edge from $F'$, and add the colour of $v_2v_3$ to $C_1$, so that, now, $|C_1|\geq s$.
Then, by \ref{typical1} (in particular, using \ref{BPW} in Theorem~\ref{thm:BPW}), we have that $D[V_1, V_2\setminus\{v_2\}, C_1]$ contains a matching of size $|V_2\setminus\{v_2\}|$. Add this matching to $F'$ and let $v_1$ be the vertex in $V_1$ not in $V(F')$. Let $f$ and $f'$ be the edges $v_2v_3$ and $v_1v_2$, each with a different new (dummy) colour, and add $f$ and $f'$ to $F'$.

In case \ref{case:n}, we have  $(\bar{n}, \bar{D})=(n, D)$ and $m\geq n$, and thus $|C_1|\geq s$. Pick $v_2\in V_2$ arbitrarily. Then, by \ref{typical1} (in particular, using \ref{BPW} in Theorem~\ref{thm:BPW}), we have that $D[V_1, V_2\setminus\{v_2\}, C_1]$ contains a matching of size $|V_2\setminus\{v_2\}|$. Add this matching to $F'$ and let $v_1$ be the vertex in $V_1$ not in $V(F')$. Let $f$ be the edge $v_1v_2$ with a new colour, and add $f$ to $F'$. Let $f'=\emptyset$.

In case \ref{case:cycle}, we have  $(\bar{n}, \bar{D})=(n', D')$ and $m\geq n'+100$, and thus $|C_1|\geq s+100$. Then, by \ref{typical1} (in particular, using \ref{SPC} in Theorem~\ref{thm:BPW}), we have that $D[V_1, V_2, C_1]$ contains a matching of size $|V_2|$. Add this matching to $F'$, and let $f=f'=\emptyset$.

In all cases, we have that $F\cup F'\subset \bar{D}+f+f'$ is a rainbow collection of $\ell$ vertex-disjoint directed paths of length $t-1$ from $V_1\cup V_1'$ to $V_{t}\cup V_t'$ intersecting
$V(S)$ precisely in $(V_1\cup V_1')\cup(V_t\cup V_t')$, with these vertices as the endpoints of its paths, and its internal vertices are precisely $V(\bar{D})\setminus V(S)$ and its non-dummy colours are contained in $C(\bar{D})\setminus C(S)$. By mapping each vertex in $V_t\cup V_t'$ to the vertex in $V_1\cup V_1'$ on the same path in $F\cup F'$, and using the Hamilton router property of $S$, we have that $\bar{D}+f+f'$ contains a directed rainbow Hamilton cycle, $S_0$ say.

In case \ref{case:n-1}, as  $(\bar{n}, \bar{D})=(n, D)$, and $f$ and $f'$ share the vertex $v_2$, we have that $S_0-f-f'$ (and hence $\bar{D}$) contains a directed rainbow path on $n-1$ vertices. In case \ref{case:n},  as $(\bar{n}, \bar{D})=(n, D)$, and $f'=\emptyset$, we have that $S_0-f-f'$ (and hence $\bar{D}$) contains a directed rainbow Hamilton path. In case \ref{case:cycle}, as $(\bar{n}, \bar{D})=(n', D')$, $D'=D[V']$, and $f=f'=\emptyset$, we have that $S_0$ is a directed rainbow cycle in $D$ with vertex set $V'$.

\par Finally, the `furthermore' part of the theorem for each distinct $x,y\in V(D)$ follows easily by taking an instance of our random partitions at \eqref{eq:vxpartition} and \eqref{eq:colpartition} so that, furthermore, $x\in V_3$ and $y\in V_4$ (noting that this holds with probability $(q/2)^2=\Omega(1)$) and, for case \ref{case:cycle}, taking an arbitrary $V'\subset V(D)$ with $x,y\in V'$ and $|V'|=n-100$. Then, in the step where we find a matching between $V_3$ and $V_4$, we choose a matching that uses $xy$ (by applying Theorem~\ref{thm:BPW} \textbf{SPC} to $D[V_3\setminus \{x\}, V_4\setminus \{y\},C_3]$). Note that the spare colour provided by this dummy edge then implies that at the final step we have in case \ref{case:n-1} or \ref{case:n} that $|C_1|\geq s$, and in case \ref{case:cycle} that $|C_1|\geq s+100$, and thus we can find a directed rainbow Hamilton path containing $xy$ and a directed rainbow cycle of length $n-100$ containing $xy$, as required.
\end{proof}


\section{Conjectures of Andersen and Gy\'arf\'as--S\'ark\"ozy for large $n$}\label{sec:andersen}
\par In this section, we use Theorem~\ref{thm:furthermore} to prove Theorems \ref{thm:mainintro}, \ref{thm:mainintrocycle}, and \ref{thm:GS-intro}. We start with the former two, deferring the third until later in the section.

\par Recall that Andersen's conjecture (Conjecture \ref{conj:andersen}) asserts that every properly coloured copy $G$ of $K_n$ admits a rainbow path on $n-1$ vertices. In particular, this includes many cases where $G$ is not $(n,\eps)$-typical for even $\eps=1/2$. On the other hand, many colourings of $G$ which are far from being $(n,\eps)$-typical for some sufficiently small $\eps>0$ are covered by the following result of Montgomery, Pokrovskiy, and Sudakov~\cite[Theorem~1.10]{montgomery2018decompositions} which finds, moreover, a rainbow Hamilton cycle.

\begin{theorem}[Montgomery--Pokrovskiy--Sudakov]\label{thm:MPS}
There is an $\alpha > 0$ so that the following holds for all $1 > \eps \geq n^{-\alpha}/\alpha$. Let $K_n$ be properly coloured with at most $(1 - \eps)n$ colours having more than $(1 - \eps)n/2$ edges. Then, $K_n$ has a rainbow Hamilton cycle.
\end{theorem}

We will use Theorem~\ref{thm:furthermore}, in combination with a stability analysis to bridge the gap between $(n,\eps)$-typicality and the condition in Theorem~\ref{thm:MPS}, to prove the following result.

\begin{theorem}\label{thm:reduction}
Let $1/n, \, \eps \llpoly \log^{-1}n$. Let $K_n$ be properly coloured using $m \geq n-1$ colours, with more than $(1-\eps)n$ colours appearing on more than $(1-\eps)n/2$ edges. Then $K_n$ admits a rainbow path on at least $\min\{n,m\}$ vertices, and a rainbow cycle missing at most $101$ vertices.
\end{theorem}

Note that both Theorems~\ref{thm:mainintro} and \ref{thm:mainintrocycle} follow directly from Theorem~\ref{thm:MPS} and Theorem~\ref{thm:reduction}. It remains only to prove Theorem~\ref{thm:reduction}. We call a colouring of $K_n$ satisfying the first condition of Theorem~\ref{thm:reduction} (i.e., more than $(1-\eps)n$ colours appearing on more than $(1-\eps)n/2$ edges) an {\it $\eps$-near-optimal} colouring. In such a colouring, we will first use the following lemma to find a short rainbow path $\hat{P}$, from, say, $x$ to $y$, such that $K_n-(V(\hat{P})\setminus \{x,y\})$ contains a spanning $(n,\eta)$-typical subgraph with many colours (see Definition~\ref{def:typicalgraph}), where $\eps\llpoly \eta$, to which we can then apply Theorem~\ref{thm:furthermore} by replacing each edge by the two possible directed edges on its vertex set with the same colour as the original edge.

\begin{lemma}\label{lem:rare-cols}
    Let $1/n, \, \eps \llpoly \eta \llpoly\log^{-1}n$. Let $G$ be an $n$-vertex complete graph which is properly coloured with an $\eps$-near-optimal colouring with at least $n$ colours.
    Then, there exist $x,y \in V(G)$ and a rainbow $xy$-path $P$ for which $G-(V(P)\setminus \{x,y\})$ contains an $(n,\eta)$-typical spanning subgraph $H$ with colours in $C(G)\setminus C(P)$ such that $|C(H)|\geq |H|-1$.
\end{lemma}

\begin{proof}
Let $V=V(G)$ and $C=C(G)$, so that $|C|\geq n=|V|$. Let
\[
\hat{C}=\{c \in C:|E_{c}(G)|<(1-\eps)n/2\}\;\;\;\text{ and }\;\;\; \hat{V}:=\{v \in V: |\{e\in E(G):v\in V(e),c(e) \in \hat{C}\}|>\sqrt{\eps}n\}.
\]
Let $C'=C\setminus \hat{C}$ so that, as the colouring of $G$ is $\eps$-near-optimal, $|{C}'|>(1-\eps)n$. Since there are  at most $\binom{n}{2}-|C'|(1-\eps)n/2\leq \eps n^2$ edges in $G$ with a colour in $\hat{C}$, we have $|\hat{V}| < 2\sqrt{\eps}n$.

We will find a rainbow path $P$ in $G$ with some endpoints $x, y$ such that
\stepcounter{propcounter}
\begin{enumerate}[label = {\textbf{\Alph{propcounter}\arabic{enumi}}}]
\item $\hat{V} \subseteq V(P)\setminus\{x,y\}$,\label{path_vertices}
\item $P$ uses at least $n-|C'|$ colours from $\hat{C}$, and\label{Pcolnum}
\item $|P| \leq \eta n/3$.\label{Plength}\label{path_xy}
\end{enumerate}

Given such a path, we can easily find the required subgraph $H\subseteq G$ to go with it. Indeed, let $H$ be the graph with $V(H)=(V(G)\setminus V(P))\cup\{x, y\}$ and edge set $E(H)=\{e \in E(G[V(H)]):c(e) \in C'\setminus C(P)\}$. Then, $H$ has $n-|P|+2\geq (1-\eta)n$ vertices by \ref{Plength}, while every colour in $C(H)\subseteq C'$ appears at least $(1-\eps)n/2-|P|\geq (1-\eta)n/2$ times (by the definition of $C'$ and by \ref{Plength}) and at most $n/2$ times as $|H|\leq n$. Furthermore, for each $v\in V(H)$, $v\notin \hat{V}$ by \ref{path_vertices}, so therefore $d_H(v)\geq n-\sqrt{\eps}n-|P|\geq (1-\eta)n$, while we have $d_H(v)\leq |H|\leq n$. Thus, $H$ is $(n, \eta)$-typical. Finally, note that, by \ref{Pcolnum}, $|C(H)|\geq |C'|-e(P)+(n-|C'|)=n-|P|+1=|H|-1$, and thus $H$ is as desired.

It remains only to find a rainbow path $P$ in $G$ with some endpoints $x, y$ such that \ref{path_vertices}--\ref{path_xy} hold. We need to take particular care over colours which appear very few times. Let, then, $\bar{C}$ be the set of colours which each have fewer than $n/100$ edges. Let $\bar{r}=\max\{0,n-|C(G)\setminus \bar{C}|\}\leq \eps n$.

\begin{claim}\label{claim:initialpathforest}
There is a rainbow path forest $F\subseteq G$ with $C(F)\subseteq \bar{C}$ and $e(F)\geq \bar{r}$.
\end{claim}
\claimproofstart[Proof of Claim~\ref{claim:initialpathforest}]
Let $F$ be a rainbow path forest in $G$ with no isolated vertices which maximises $e(F)$ subject to $C(F)\subseteq \bar{C}$. Suppose, for a contradiction, that $e(F)<\bar{r}$. In particular, then, we have $0<\bar{r}=n-|C(G) \setminus \bar{C}|\leq |C(G)|-|C(G) \setminus \bar{C}|=|\bar{C}|$.
If $\bar{r}\leq 2$, then (using $|\bar{C}|\geq 2$) taking any $\bar{r}$ edges with different colours in $\bar{C}$ gives the required path forest. Thus, we may assume that $\bar{r}\geq 3$.
Furthermore, as $\bar{r}=n-|C(G) \setminus \bar{C}|$, and $G$ is complete, every vertex in $G$ is adjacent to at least $(n-1)-|C(G) \setminus \bar{C}|=\bar{r}-1$ edges with colour in $\bar{C}$.

Let $V_F=V(F)$ and $C_F=C(F)$.
Any edge in $G$ with colour in $\bar{C}\setminus C_F$ and both endpoints in $V\setminus V_F$ can be added to $F$ to  contradict the maximality of $e(F)$. Thus, all the edges with colour from $\bar{C}\setminus C_F$ with one endpoint in $V\setminus V_F$ must have their other endpoint in $V_F$. Since every colour in $\bar{C}$ appears on at most $n/100$ edges in $G$, there are at most $|C_F|\cdot n/100 \leq (\bar{r}-1)\cdot n/100$ edges with colour in $C_F$.
As every vertex in $V\setminus V_F$ is adjacent to at least $\bar{r}-1\geq 2\bar{r}/3$ edges with colour in $\bar{C}$, and $|V\setminus V_F|\geq (1-2\eps) n$ (as $|V_F|\leq 2\bar{r}$), the number of edges with colour in $\bar{C}\setminus C_F$ in $G[V\setminus V_F,V_F]$ is at least
\begin{equation}\label{eq:VVPVPedges}
(\bar{r}-1)\cdot |V\setminus V_F|-2\cdot (\bar{r}-1)\cdot \frac{n}{100}\geq \frac{2\bar{r}}{3}\left((1-2\eps)n-\frac{n}{50}\right)\geq \frac{3\bar{r}n}{5}.
\end{equation}

Now, observe that any vertex in $V_F$ which is an endpoint of a path in $F$ has (due to the maximality of $e(F)$) no neighbouring edges to $V\setminus V_F$ with colour in $\bar{C}\setminus C_F$. Take a matching $M$ in $E(F)$ of at most $e(F)/2\leq \bar{r}/2$ which covers the interior vertices of the paths in $F$. Observe that, for each $xy\in M$, there are at most $n$ edges from $\{x,y\}$ to $V\setminus V_F$ with colour in $\bar{C}\setminus C_F$. Otherwise, assuming $x$ is incident to at most half of these edges, it is easy to pick $x'\in V\setminus V_F$ and then $y'\in V\setminus V_F$ so that $\{xx',yy'\}$ is a rainbow matching with colours in $\bar{C}\setminus C_F$, whereupon $F-xy+xx'+yy'$ contradicts the maximality of $e(F)$. Therefore, the number of edges with colour in $\bar{C}\setminus C_F$ in $G[V\setminus V_F,V_F]$ is at most $|M|\cdot n\leq \bar{r}n/2$. This contradicts \eqref{eq:VVPVPedges}.
\claimproofend

Let $r^*=\max\{0,n-|C'|\}\leq \eps n$. As $\bar{C} \subseteq \hat{C}=C\setminus C'$, and every colour in $\hat{C}\setminus \bar{C}$ has at least $n/100$ edges in $G$, by Claim~\ref{claim:initialpathforest}, we can greedily extend the path forest found there to a path forest $F\subseteq G$ with $C(F)\subseteq \hat{C}$, $e(F)\geq r^*$ and no isolated vertices. Letting $F'$ be $F$ with any vertices in $\hat{V}\setminus V(F)$ added as isolated vertices, we get a path forest $F'\subseteq G$ with $r^*$ edges and at most $2r^*+|\hat{V}|\leq 3\sqrt{\eps}n$ vertices, such that $F'$ is rainbow with colours in $\hat{C}$ and such that $\hat{V}\subseteq V(F')$.

Let $r'\leq 3\sqrt{\eps}n$ be the number of paths in $F'$ (including paths with 1 vertex) and label vertices $u_i,v_i$, $i\in [r']$, so that $F'$ consists of one $u_iv_i$-path for each $i\in [r']$. Using that $|F'|\leq 3\sqrt{\eps}n$, pick distinct $v_0,u_{r'+1}\in V\setminus V(F')$, and let $x=v_0$ and $y=u_{r'+1}$.
For each $i\in [r'+1]$ in turn, pick a vertex
\[
w_i\in V\setminus \{v_0,v_1,\dots,v_{r'},u_1,u_2,\dots,u_{r'+1},w_1,w_2,\dots,w_{i-1}\}
\]
such that $v_{i-1}w_i$ and $w_iu_i$ have colours not in $C(F)\cup \{c(v_{j-1}w_j),c(w_ju_j):j\in [i-1]\}$, noting that this is possible as these conditions rule out at most $|F|+2(|C(F)|+2r')\leq 15\sqrt{\eps}n$ vertices in $V$ for $w_i$.
Finally, then, let $P$ be the path formed by adding the edges $v_{j-1}w_j,w_ju_j$, $j\in [r'+1]$, to $F'$, noting that this is a rainbow $xy$-path by construction.

As $P$ has at most $|F|+2(|C(F)|+2r')\leq 15\sqrt{\eps}n$ vertices, we have that \ref{path_xy} holds. As $\{x,y\}= \{v_0,u_{r'+1}\}\subseteq V\setminus V(F')\subseteq V\setminus \hat{V}$ and $\hat{V}\subseteq V(F')\subseteq V(P)$, we have that \ref{path_vertices} holds. Finally, as $F'\subseteq P$, $P$ uses at least $r^*\geq n-|C'|$ colours from $\hat{C}$, and thus \ref{Pcolnum} holds. Thus, $P$ satisfies \ref{path_vertices}--\ref{path_xy}, as required.
\end{proof}

It remains only to combine Lemma~\ref{lem:rare-cols} and Theorem~\ref{thm:furthermore} to prove Theorem \ref{thm:reduction}.

\begin{proof}[Proof of Theorem~\ref{thm:reduction}]
Let $1/n, \, \eps \llpoly \log^{-1}n$, and let $G$ be a properly coloured copy of $K_n$ which has an $\eps$-near-optimal colouring with $m\geq n-1$ colours. If $m=n-1$, then every colour appears exactly $n/2$ times and thus the colouring is $(n,\eps)$-typical, and so Theorem~\ref{thm:reduction} follows directly from Theorem~\ref{thm:furthermore}(1) and (3) (applied by first replacing each edge $xy\in E(G)$ by the directed edges $xy$ and $yx$, both with colour $c(xy)$).

Suppose, then, that $m\geq n$. Let $\eta$ be such that $\eps \llpoly \eta \llpoly\log^{-1}n$. 
By Lemma \ref{lem:rare-cols}, we can find $x,y\in V(G)$ and a rainbow $xy$-path $P$ in $G$ along with an $(n,\eta)$-typical subgraph $H\subseteq G$ with $V(H)=(V(G)\setminus V(P)) \cup \{x,y\}$, $C(H)$ is disjoint from $C(P)$, and $|C(H)|\geq |H|-1$. Add a dummy edge between $x$ and $y$ of a colour distinct from all the colours in $C(H)$. By the `furthermore' part of Theorem~\ref{thm:furthermore}, we can find in $H$ a rainbow Hamilton path, and a rainbow cycle of length $|H|-101$, using the dummy edge $xy$. The path and cycle in $H$ combine with $P$, respectively, to form a rainbow Hamilton path and a rainbow cycle of length $n-101$ in $G$, as desired.  \end{proof}

It remains to prove Theorem~\ref{thm:GS-intro}.
\begin{proof}[Proof of Theorem~\ref{thm:GS-intro}]
    Let $\overleftrightarrow{K_n}$ be optimally coloured, suppose $n$ is sufficiently large and take some $\eps$ such that $\eps\llpoly \log^{-1}n$. Denote by $G$ the complete digraph obtained by removing the loops from $\overleftrightarrow{K_n}$. Each colour class in $\overleftrightarrow{K_n}$ is a $1$-factor, and there are exactly $n$ colours. Let $\mathcal{D}$ be the set of colours appearing on the diagonal (as loops) at least $\eps n$ times, noting $|\mathcal{D}|\leq 1/\eps$. If $|\mathcal{D}|\leq 1$, then $G$, upon the removal of the diagonal, contains at least $n-1$ colours and is $(n,\eps)$-typical, so we can conclude immediately from Theorem~\ref{thm:furthermore}(1). Suppose, then, that $|\mathcal{D}|\geq 2$, in which case each colour in $\mathcal{D}$ appears as a loop at most $(1-\eps)n$ times, and thus appears on non-loop edges at least $\eps n$ times. As $\eps n \gg 1/\eps$, we can greedily find a (directed) exactly-$\mathcal{D}$-rainbow matching $M$ in $G$. By applying Lemma~\ref{Connecting_lemma} $|\mathcal{D}|-1$ times, we can find a directed rainbow path $P$ of length at most $10/\eps$ which extends the matching $M$. Note that, by deleting the colours and the internal vertices of $P$ and the colours in $\mathcal{D}$ from $G$, we obtain an $(n',2\eps)$-typical digraph $G'$ with $n'$ vertices and at least $n'-1$ colours. Let $(x,y)$ be the endpoints of $P$, and let us add the dummy edge from $x\to y$ to $G'$ with a new colour. By the `furthermore' part of Theorem~\ref{thm:furthermore}, we find a rainbow Hamilton path using the dummy edge, which implies the existence of a rainbow Hamilton path in $G$. Therefore, we can conclude the proof in the $|\mathcal{D}|\geq 2$ case also (in fact, in this case, we found a longer path than necessary).
\end{proof}
 \par \noindent \textbf{Acknowledgements.} We would like to thank Alistair Benford for valuable comments on an early draft of this manuscript. The third author thanks Cosmin Pohoata for pointing him towards Jamison's conjecture.
 \medskip
 \par \noindent \textbf{Statement of AI use.} ChatGPT 5.6 was used to proofread this paper. All of the ideas and writing are due entirely to the authors.


\bibliographystyle{abbrv}
\bibliography{rbs}
\appendix{}

\section{Proof of Theorem~\ref{thm:BPW}}\label{appendix:A}


Here we give a proof of Theorem~\ref{thm:BPW}. To start, we recall the notion of \textit{proper-pseudorandomness} from \cite{montgomery2023proof}. Roughly speaking, an $(n,p,\eps)$-properly-pseudorandom bipartite graph is a properly coloured pseudorandom graph for which the proof strategy in \cite{montgomery2023proof} goes through. The strategy in \cite{montgomery2023proof} is based on the absorption method, hence the proof requires working with various small subgraphs (gadgets) with certain properties. The formal definition of proper-pseudorandomness bounds the size of these small subgraph families.


\subsection{Proper-pseudorandomness}\label{sec:proper-pseud}
\par In this section, we give a list of seven properties that make up the definition of proper-pseudorandomness that are listed as \textbf{F1}--\textbf{F7} in \cite{montgomery2023proof}, with a few slight alterations as we outline below.
\par Firstly, in its definition of proper-pseudorandomness, \cite{montgomery2023proof} uses a parameter $\alpha=p^{12}/10^{100}$. This parameter is chosen to be sufficiently small as a function of $p$ that the two main coloured bipartite graphs used in \cite{montgomery2023proof} are $(n,p,\eps)$-properly-pseudorandom under this definition, where the values $p=1$ and $p=1/3$ are used respectively in \cite[Section~3.6]{montgomery2023proof} and \cite[Section~10]{montgomery2023proof}. The parameter $\alpha$ is only chosen to be a fixed function of $p$ to reduce the variables involved in the definition of proper-pseudorandomness.  We will need to take a smaller value of $\alpha$, to record the properties of our bipartite coloured graphs. The proofs of the results from \cite{montgomery2023proof} (recalled in Section~\ref{sec:recallthms}) can be seen to follow with only the most minor adjustments for any constants $0<\alpha \ll p$, as long as $n$ is sufficiently large.
\par We also provide a slightly altered definition of \ref{prop-pseud-add-new-2} compared to \cite[\textbf{F7}]{montgomery2023proof} in that our \ref{prop-pseud-add-new-2} here assumes a slightly weaker property. As will be reflected in the published version of \cite{montgomery2023proof}, the proof in \cite{montgomery2023proof}, as written, uses only this weaker property. The stronger property originally stated in \cite[\textbf{F7}]{montgomery2023proof} is presented here as \ref{prop-alt}, and as we will discuss in Section~\ref{sec:recallthms}, if this stronger property is assumed, one can derive a slightly strengthened version of the main theorem of \cite{montgomery2023proof}.
\par Another very slight alteration is in the property \ref{prop-pseud-add-2} below. In \cite{montgomery2023proof}, a stronger \ref{prop-pseud-add-2} is stated that does not need to designate, given $c_0$, two exceptional vertices $u_0$ and $v_0$ for which the property in \ref{prop-pseud-add-2} may not hold. However, the proofs in \cite{montgomery2023proof} go through with the weaker version of \ref{prop-pseud-add-2} we state here. Indeed, \cite{montgomery2023proof} uses only a single value of $c_0$ as an `identity colour' throughout the proof, so $c_0$ is fixed in the very beginning. To see that the weaker \ref{prop-pseud-add-2} does not create any difficulties, observe that, given $c_0$, we can greedily find a rainbow matching of size $2$ (not using $c_0$) saturating the exceptional vertices $u_0$ and $v_0$. Deleting the vertices and colours of this matching, we see that the stronger version of \ref{prop-pseud-add-2} is restored in the remainder graph (with a new value of $\alpha':=\alpha/2$), and the proof in \cite{montgomery2023proof} goes through exactly as written (the only two instances in which \ref{prop-pseud-add-2} is used in \cite{montgomery2023proof} are in the proofs of \cite[Claims~6~and~7]{montgomery2023proof}).

Let $G$ be a properly coloured bipartite graph, with vertex classes $A$ and $B$. Below we define certain properties parametrised by $n,p,\eps$, and $\alpha:=p^{10^{10}}/10^{10^{10}}$, where an edge coloured graph is said to be \emph{exactly-$C$-rainbow} if it uses every colour in $C$ exactly once.
\stepcounter{propcounter}
\begin{enumerate}[label = {\textbf{\Alph{propcounter}\arabic{enumi}}}]
\item $|A|=|B|=n$ and $n\leq |C(G)|\leq (1+\eps)n$. \label{prop-pseud-basic-1}

\item $\mathcal{H}(G)$ is $(n,p,\eps)$-typical (as defined in Section~\ref{sec:typical}).\label{prop-pseud-basic-2new}

\item For each $c\in C(G)$ and $e\in E_c(G)$, for all but at most $\sqrt{n}$ edges $f\in E_c(G)\setminus \{e\}$, there are at least $\alpha n^2$ pairs $(S_1,S_2)$ such that $S_1$ and $S_2$ are vertex-disjoint rainbow 4-cycles, $e\in E(S_1)$, $f\in E(S_2)$, and
the colour sets of the neighbouring edges of $e$ in $S_1$ and the neighbouring edges of $f$ in $S_2$ are the same.\label{prop-pseud-abs-prime}
\item For any $c_0\in C(G)$ there exist $u_0\in A$ and $v_0\in B$ such that the following holds for all $u\in A\setminus\{u_0\}$ and $v\in B\setminus\{v_0\}$. There are disjoint sets ${V}_1,\ldots,{V}_{\alpha n}\subset V(G)\setminus \{u,v\}$ and disjoint sets $C_1,\ldots, C_{\alpha n}$ in $C(G)\setminus\{c_0\}$ such that,
for each $i\in [\alpha n]$, $|{V}_i|=4$, $|C_i|=3$, $G[V_i]$ contains 2 colour-$c_0$ edges and $G[\{u,v\}\cup V_i]$ contains an exactly-$C_i$-rainbow matching in $E(G)\setminus \{uv\}$.\label{prop-pseud-add-2}
\item For each distinct $c_0,d\in C(G)$, there are disjoint sets $V_1,\ldots,V_{\alpha n/12}$ in $V(G)$ and disjoint sets $C_1,\ldots,$ $C_{\alpha n/12}$ in $C(G)\setminus\{c_0,d\}$,
so that, for each $i\in [\alpha n/12]$, $|V_i|=8$, $|C_i|=3$, and $G[V_i]$ contains a matching of $4$ colour-$c_0$ edges
and an exactly-$(C_i\cup \{d\})$-rainbow matching.\label{prop-pseud-add-3}

\item  For any $c_0\in C(G)$, $0\leq k\leq 20$, and any $\bar{C}\subset C(G)\setminus \{c_0\}$ with $|\bar{C}|\geq 5k$, there are disjoint sets $\bar{V}_1,\ldots,\bar{V}_{\alpha n}\subset V(G)$
such that,
for each $i\in [\alpha n]$, $|\bar{V}_i|=2k+2$ and $G[\bar{V}_i]$ contains both a matching of $k+1$ colour-$c_0$ edges and a $\bar{C}$-rainbow matching with $k$ edges. \label{prop-pseud-add-new-1}
\item Setting $k=100$, for each $c_0\in C(G)$, there is some $r\in \N$ and disjoint sets ${V}_1,\ldots,{V}_{r}$ in $V(G)$ and disjoint sets $C_1,\ldots,C_{r}$ in $C(G)\setminus\{c_0\}$
 with $|V_i|=2k$ and $|C_i|=k$ for each $i\in [r]$ such that $G[V_i]$ contains an exactly-$C_i$-rainbow matching and a perfect matching of colour-$c_0$ edges, and the following holds. For every $\bar{C}\subset C(G)\setminus \{c_0\}$ with $|\bar{C}|\leq k$, for at least $\alpha^2n$ values of $i\in [r]$, there are vertex-disjoint sets $\bar{V}_1,\ldots,\bar{V}_{\alpha n}\subset V(G)$ such that,
for each $j\in [\alpha n]$, $|\bar{V}_j|=2k+2|\bar{C}|$ and $G[\bar{V}_j]$ contains both a matching of $k+|\bar{C}|$ colour-$c_0$ edges and a $(\bar{C}\cup C_i)$-rainbow matching with $k+|\bar{C}|-1$ edges.
\label{prop-pseud-add-new-2}

\end{enumerate}

\begin{defn}[Proper-pseudorandomness]\label{defn:pseud}
    Let $G$ be a properly coloured bipartite graph, with vertex classes $A$ and $B$. We call $G$ \emph{$(n,p,\eps)$-properly-pseudorandom} if the properties \ref{prop-pseud-basic-1} through \ref{prop-pseud-add-new-2} all hold.
\end{defn}

The following weaker variant of \ref{prop-pseud-basic-1} and the stronger variant of \ref{prop-pseud-add-new-2} will also be important to record.

\begin{enumerate}[label = {{\textbf{C1*}}}]
\item $|A|=|B|=n$ and $n-1\leq |C(G)|\leq (1+\eps)n$. \label{prop-alt-basic}
\end{enumerate}

\begin{enumerate}[label = {{\textbf{C7*}}}]
\item Setting $k=100$, for each $c_0\in C(G)$, there is some $r\in \N$ and disjoint sets ${V}_1,\ldots,{V}_{r}$ in $V(G)$ and disjoint sets $C_1,\ldots,C_{r}$ in $C(G)\setminus\{c_0\}$
 with $|V_i|=2k$ and $|C_i|=k$ for each $i\in [r]$ such that $G[V_i]$ contains an exactly-$C_i$-rainbow matching and a perfect matching of colour-$c_0$ edges, and the following holds. For every $\bar{C}\subset C(G)\setminus \{c_0\}$ with $|\bar{C}|\leq k$, for at least $\alpha^2n$ values of $i\in [r]$, there are vertex-disjoint sets $\bar{V}_1,\ldots,\bar{V}_{\alpha n}\subset V(G)$ such that,
for each $j\in [\alpha n]$, $|\bar{V}_j|=2k+2|\bar{C}|+2$ and $G[\bar{V}_j]$ contains both a matching of $k+|\bar{C}|+1$ colour-$c_0$ edges and a $(\bar{C}\cup C_i)$-rainbow matching with $k+|\bar{C}|$ edges. \label{prop-alt}
\end{enumerate}

\begin{defn}[Strong-proper-pseudorandomness]
    Let $G$ be a properly coloured bipartite graph, with vertex classes $A$ and $B$. We call $G$ \emph{$(n,p,\eps)$-strongly-properly-pseudorandom} if the properties \ref{prop-alt-basic}, \ref{prop-pseud-basic-2new} to \ref{prop-pseud-add-new-1}, and \ref{prop-alt} all hold.
\end{defn}


\subsection{Rainbow matchings in pseudorandom properly-coloured bipartite graphs}
\label{sec:recallthms}

Here we state the three theorems from \cite{montgomery2023proof} that we require, which all find large rainbow matchings in suitably pseudorandom properly coloured 
bipartite graphs. The first is the main technical result of \cite{montgomery2023proof}, which finds an $(n-1)$-edge rainbow matching when the graph contains $n$ vertices in each part and at least $n$ colours. The second takes a slightly stronger pseudorandom condition, so that such a matching can be found using only $n-1$ colours (i.e., so that all the colours may be used exactly once).
The last such result finds a slightly larger $n$-edge rainbow matching under weaker pseudorandom conditions when slightly more colours are available.

\begin{theorem}[{\cite[Theorem 3.1]{montgomery2023proof}}]\label{thm-technical} Let $1/n\ll p\leq 1$ and let $\eps\llpoly\log^{-1}n$. Then, any $(n,p,\eps)$-properly-pseudorandom bipartite graph $G$ contains a rainbow matching with $n-1$ edges.
\end{theorem}
\par As already remarked earlier in Section~\ref{sec:proper-pseud}, \cite[Theorem 3.1]{montgomery2023proof}, as written, technically assumes the stronger property \ref{prop-alt} instead of \ref{prop-pseud-add-new-2} as we state here, but the proof of \cite[Theorem 3.1]{montgomery2023proof} works exactly as written with the weaker \ref{prop-pseud-add-new-2} as we formally assert above (and as will be reflected in the published version of \cite{montgomery2023proof}). We also remark that \cite[Theorem 3.1]{montgomery2023proof} also makes the additional assumption that $1/n\llpoly \eps $, but this is simply to simplify the presentation of the hierarchy and is not a required property. Indeed, a smaller value of $\eps$ only serves to strengthen the typicality properties of the graph, so the formulations we give above and below that omit this assumption are equivalent to the formulations in \cite{montgomery2023proof}.
\par We also observe from the proof of \cite[Theorem~3.1]{montgomery2023proof} (and as will be updated in the published version of that proof~\cite{montgomerypersonal}) that it is sufficient for the graph $G$ to have only at least $n-1$ colours, as opposed to at least $n$ colours as claimed there, under our strengthened proper-pseudorandomness. That is, we can weaken \ref{prop-pseud-basic-1} to \ref{prop-alt-basic} if we strengthen \ref{prop-pseud-add-new-2} to \ref{prop-alt} in compensation. To explain this further, we first state the key property in the proof used to find a matching saturating all remaining colours (from \cite[Section 9.4]{montgomery2023proof}).
\begin{enumerate}[label = {{\textbf{W3}}}]
\item For every $\bar{C}\subset C(G)\setminus \{c_0\}$ with $|\bar{C}|=k$, for at least $\alpha^2n$ values of $i\in [r]$, there are vertex-disjoint sets $\bar{V}_1,\ldots,\bar{V}_{\alpha n}\subset V$ such that,
for each $j\in [\alpha n]$, $|\bar{V}_j|=4k+2$  and $G[\bar{V}_j]$ contains both a matching of $2k+1$ colour-$c_0$ edges and a $(\bar{C}\cup C_i)$-rainbow matching with $2k$ edges.
\end{enumerate}
As the rainbow matching found using this property uses exactly all the colours in $\bar{C}\cup C_i$, this allows, with very minor adjustments, the overall $(n-1)$-edge rainbow matching found to use all of the colours if necessary, so that only at least $n-1$ colours overall are needed.

\begin{theorem}\label{thm-strong-technical} Let $1/n\ll p\leq 1$ and let $ \eps\llpoly\log^{-1}n$. Then, any $(n,p,\eps)$-strongly-properly-pseudorandom bipartite graph contains a rainbow matching with $n-1$ edges.
\end{theorem}
Below is the final theorem we need that says we can find a matching saturating the vertices if there are enough `spare' colours.

\begin{theorem}[{\cite[Theorem 3.2]{montgomery2023proof}}]\label{thm-technical-variant} Let $1/n\ll p\leq 1$ and let $\eps\llpoly\log^{-1}n$. Then, any $(n,p,\eps)$-properly-pseudorandom bipartite graph $G$ with at least $n+100$ colours has a rainbow matching with $n$ edges.
\end{theorem}

\subsection{Proof of Theorem~\ref{thm:BPW}} \label{sec_2.1proof}
The proof of Theorem~\ref{thm:BPW} now boils down to verifying the notions of proper-pseudorandomness within the relevant subgraphs involving random sets of vertices and colours. First, we restate Theorem~\ref{thm:BPW} for convenience.

\begin{theorem}\label{thm:BPWapp}Let $1/n\ll p\leq 1$, $\eps \llpoly\log^{-1}n$, and let $D$ be a properly coloured $(n,\eps)$-typical digraph. Let $X,Y$ be disjoint $p$-random subsets of $V(D)$, and let $C$ be a $p$-random subset of the colours $C(D)$, sampled independently. Then, with high probability, the following hold for any choice of disjoint $X',Y'\subseteq V(D)$ and $C'\subseteq C(D)$ satisfying $|Q\triangle Q'|\leq \eps n$ for each $Q\in \{X,Y,C\}$.

\begin{enumerate}[label = \textup{\textbf{RBS}}]
    \item\label{RBS-app}  If $|X'|=|Y'|\leq |C'|$, then $D[X',Y',C']$ has a rainbow matching of size $|X'|-1$.\renewcommand{\labelenumi}{\textup{\textbf{BPW}}}
    \item\label{BPW-app}  If $|X'|=|Y'|+1\leq |C'|$, then $D[X',Y',C']$ has a rainbow matching of size $|X'|-1=|Y'|$.\renewcommand{\labelenumi}{\textup{\textbf{SPC}}}
    \item\label{SPC-app}  If $|X'|=|Y'|\leq |C'|-100$, then $D[X',Y',C']$ has a rainbow matching of size $|X'|$.
\end{enumerate}
\end{theorem}

Below is the lemma that records all of the proper-pseudorandomness inheritance that we require. Given an edge-coloured digraph $D$, and sets $X,Y\subseteq V(D)$ and $C\subseteq C(D)$, $D[X,C,Y]$ is the undirected coloured bipartite graph consisting of vertex classes $X,C$, and an edge $xc$ with $x\in X$ and $c\in C$ of colour $y\in Y$ whenever $xy$ is an edge of colour $c$ in $D$. The graph $D[X,Y,C]$ is defined in Section~\ref{sec:notation}.

\begin{lemma}\label{lem:pseud is inherited appendix} Let $1/n\ll p\leq 1$, and let $\eps \llpoly \tilde \eps\llpoly\log^{-1}n$, and let $D$ be a properly coloured $(n,\eps)$-typical digraph. Let $X,Y$ be disjoint $p$-random subsets of $V(D)$, and let $C$ be a $p$-random subset of the colours $C(D)$, sampled independently. Then, with high probability, the following holds.
\begin{enumerate}[label={\textup{\textbf{PrP}}}]
    \item\label{PrP}  For any disjoint $X',Y'\subseteq V(D)$ and $C'\subseteq C(D)$ where, for each $Q\in \{X,Y,C\}$, $|Q\triangle Q'|\leq \eps n$,
    \begin{enumerate}[label={\textup{\textbf{\arabic{enumii}}}},leftmargin=0pt]
    \item if $|X'|=|C'|\leq |Y'|+1,$ then $D[X',C',Y']$ is $(|X'|,p,\tilde \eps)$-strongly-properly-pseudorandom, and \label{PrP-1}
    \item if $|X'|=|Y'|\leq |C'|,$ then  $D[X',Y', C']$ is $(|X'|,p,\tilde \eps)$-properly-pseudorandom.\label{PrP-2}
    \end{enumerate}
\end{enumerate}
\end{lemma}
We now prove Theorem~\ref{thm:BPWapp} subject only to the proof of Lemma~\ref{lem:pseud is inherited appendix}, which is in Section~\ref{sec_lemmaA.7}.
\begin{proof}[Proof of Theorem~\ref{thm:BPWapp}] Let $\tilde{\eps}$ satisfy $\eps\llpoly \tilde{\eps}\llpoly \log^{-1}n$. By Lemma~\ref{lem:pseud is inherited appendix}, with high probability, $X,Y,C$ satisfy \ref{PrP}.
For any such $X,Y,C$ satisfying \ref{PrP}, we now show that the conclusion of Theorem~\ref{thm:BPWapp} holds. Firstly, using \ref{PrP}\textbf{.}\ref{PrP-2}, we have that \ref{RBS-app} and \ref{SPC-app} are direct consequences of Theorem~\ref{thm-technical} and Theorem~\ref{thm-technical-variant}, respectively. Then, using the simple observation that $D[X',Y',C']$ has a rainbow matching of size $|Y'|$ if and only if $D[X',C',Y']$ has a rainbow matching of size $|Y'|$, \ref{BPW-app} follows from \ref{PrP}\textbf{.}\ref{PrP-1} together with Theorem~\ref{thm-strong-technical} applied to $D[X',C',Y']$.
\end{proof}

The remaining task is to prove Lemma~\ref{lem:pseud is inherited appendix}.


\smallskip

\subsection{Proof of Lemma~\ref{lem:pseud is inherited appendix}}\label{sec_lemmaA.7}
We will now prove Lemma~\ref{lem:pseud is inherited appendix}. For this, let $1/n\ll p\leq 1$, and let $\eps \llpoly \tilde \eps\llpoly\log^{-1}n$, and let $D$ be a properly coloured $(n,\eps)$-typical digraph. Let $X,Y$ be disjoint $p$-random subsets of $V(D)$, and let $C$ be a $p$-random subset of the colours $C(D)$, sampled independently.
Our aim is to prove that \ref{PrP} holds with high probability.
Where  $X',Y',C'$ are as in \ref{PrP}, this concerns two different graphs, $D[X',C',Y']$ and $D[X',Y',C']$. The second of these is the most natural interpretation of the relevant coloured graph, matching the use in, e.g., \cite{KPSY,montgomery2023proof}. The first hypergraph, $D[X',C',Y']$, however, flips the role of the colours and the second vertex set. For convenience then, we will restate the properties we need for $D[X',C',Y']$ to be $(|X'|,p,\tilde \eps)$-strongly-properly-pseudorandom in the following more natural form. Note that it would suffice to replace `$n$' with `$|X'|$' in \ref{prop-pseud-abs-prime-flip} to \ref{prop-alt-flip}. However, since with high probability we have that $|X'| \sim pn$, and we have that $\alpha\leq p^{1000}/10^{1000}$, we can prove the stated claims (which are technically stronger by a factor of $p$), reducing slightly the notation.

\stepcounter{propcounter}
\begin{enumerate}[label={\textbf{\Alph{propcounter}\astar*}}]
\item $|X'|=|C'|$ and $|X'|-1\leq |Y'| \leq (1+\tilde \eps)|X'|$. \label{prop-pseud-basic-1-flip}

\item $\mathcal{H}(D[X',C',Y'])$ is $(|X'|,p,\tilde \eps)$-typical (recall Section~\ref{sec:typical}).\label{prop-pseud-basic-2new-flip}

\item For each $y\in Y'$ and $e\in E_y(D[X', C', Y'])$, for all but at most $\sqrt{n}$ edges $f\in E_y(D[X', C', Y'])\setminus \{e\}$, there are at least $\alpha n^2$ pairs $(S_1,S_2)$ such that $S_1$ and $S_2$ are vertex-disjoint rainbow 4-cycles with $e\in E(S_1)$ and $f\in E(S_2)$, and
the colour sets of the neighbouring edges of $e$ in $S_1$ and the neighbouring edges of $f$ in $S_2$ are the same.\label{prop-pseud-abs-prime-flip}
\item For any $y_0 \in Y'$ there exist $x_0\in X'$ and $c_0\in C'$ such that the following holds for all $x\in X'\setminus\{x_0\}$ and $c \in C'\setminus\{c_0\}$. There are disjoint sets ${X}_1,\ldots,{X}_{\alpha n}\subset X'\setminus \{x\}$, $C_1,\ldots,C_{\alpha n}$ in $C'\setminus\{c\}$ and $Y_1,\ldots,Y_{\alpha n}$ in $Y'\setminus\{y_0\}$ such that,
for each $i\in [\alpha n]$, $|{X}_i|=|{C_i}|=2$, $|Y_i|=3$, $D[X_i, C_i, \{y_0\}]$ contains 2 colour-$y_0$ edges and $D[X_i\cup \{x\}, C_i \cup \{c\}, Y_i]$ contains an exactly-$Y_i$-rainbow matching in $E(D[X', C', Y'])\setminus \{xc\}$.\label{prop-pseud-add-2-flip}
\item For each distinct $y_0,y'\in Y'$, there are disjoint sets $X_1,\ldots,X_{\alpha n/12}$ in $X'$, $C_1,\ldots,C_{\alpha n/12}$ in $C'$, and $Y_1,\ldots,Y_{\alpha n/12}$ in $Y'\setminus\{y_0,y'\}$
so that, for each $i\in [\alpha n/12]$, $|X_i|=|C_i|=4$ and $|Y_i|=3$, and $D[X_i, C_i, Y_i \cup \{y_0, y'\}]$ contains a matching of $4$ colour-$y_0$ edges
and an exactly-$(Y_i\cup \{y'\})$-rainbow matching.\label{prop-pseud-add-3-flip}

\item
For any $y_0\in Y'$, $0\leq k\leq 20$, and any $\bar{Y}\subset Y'\setminus \{y_0\}$ with $|\bar{Y}| \geq 5k$, there are disjoint sets $\bar{X}_1,\ldots,\bar{X}_{\alpha n}\subset X'$ and $\bar{C}_1,\ldots,\bar{C}_{\alpha n}\subset C'$
such that,
for each $i\in [\alpha n]$, $|\bar{X}_i|=|\bar{C}_i|=k+1$ and $D[\bar{X}_i, \bar{C}_i, \bar{Y} \cup \{y_0\}]$ contains both a matching of $k+1$ colour-$y_0$ edges and a $\bar{Y}$-rainbow matching with $k$ edges. \label{prop-pseud-add-new-1-flip}

\item Setting $k=100$, for each $y_0\in Y'$, there is some $r\in \N$ and disjoint sets ${X}_1,\ldots,{X}_{r}$ in $X'$, $C_1,\ldots,C_{r}$ in $C'$ and $Y_1,\ldots,Y_{r}$ in $Y'\setminus\{y_0\}$
 with $|X_i|=|C_i|=|Y_i|=k$ for each $i\in [r]$ such that $D[X_i, C_i, Y_i \cup \{y_0\}]$ contains an exactly-$Y_i$-rainbow matching and a perfect matching of colour-$y_0$ edges, and the following holds. For every $\bar{Y}\subset Y'\setminus \{y_0\}$ with $|\bar{Y}|\leq k$, for at least $\alpha^2n$ values of $i\in [r]$, there are vertex-disjoint sets $\bar{X}_1,\ldots,\bar{X}_{\alpha n}\subset X'$ and $\bar{C}_1,\ldots,\bar{C}_{\alpha n}\subset C'$ such that,
for each $j\in [\alpha n]$, $|\bar{X}_j|=|\bar{C}_j|=k+|\bar{Y}|+1$ and $D[\bar{X}_j, \bar{C}_j, \bar{Y} \cup Y_i \cup \{y_0\}]$ contains both a matching of $k+|\bar{Y}|+1$ colour-$y_0$ edges and a $(\bar{Y}\cup Y_i)$-rainbow matching with $k+|\bar{Y}|$ edges. \label{prop-alt-flip}
\end{enumerate}

We will separate the proof that \ref{PrP} holds with high probability into showing that 14 different statements hold with high probability corresponding to the 7 properties for proper-pseudorandomness and the 7 properties for strong-proper-pseudorandomness. For example, we will use \textbf{\ref{PrP}.\ref{PrP-1}.\ref{prop-pseud-basic-1}} to refer to the statement that ``For any disjoint $X',Y'\subseteq V(D)$ and $C'\subseteq C(D)$ where, for each $Q\in \{X,Y,C\}$, $|Q\triangle Q'|\leq \eps n$, if $|X'|=|C'|\leq |Y'|+1,$ then \ref{prop-pseud-basic-1-flip} holds for $D[X',C',Y']$ to be $(|X'|,p,\tilde \eps)$-strongly-properly-pseudorandom.''
We will prove that \textbf{\ref{PrP}.\ref{PrP-1}} holds with high probability in Section~\ref{sec_flip} by showing that each of the relevant statements corresponding to \ref{prop-pseud-basic-1-flip}, \ref{prop-pseud-basic-2new-flip}--\ref{prop-pseud-add-new-1-flip}, and \ref{prop-alt-flip} holds with high probability. Similarly, then,
we will prove that \textbf{\ref{PrP}.\ref{PrP-2}} holds with high probability in Section~\ref{sec_nonflip} by showing that each of the relevant statements corresponding to \ref{prop-pseud-basic-1}--\ref{prop-pseud-add-new-2} holds with high probability.

\par We begin by briefly discussing \ref{prop-pseud-basic-1}, \ref{prop-pseud-basic-2new}, \ref{prop-pseud-basic-1-flip} and \ref{prop-pseud-basic-2new-flip}, which are the most straightforward and standard properties to verify. First, note that, with high probability, $X$, $Y$ and $C$ are such that, for any disjoint $X',Y'\subseteq V(D)$ and $C'\subseteq C(D)$ where, for each $Q\in \{X,Y,C\}$, $|Q\triangle Q'|\leq \eps n$, if $|X'|=|C'|\leq |Y'|+1,$ then \ref{prop-pseud-basic-1-flip} holds, i.e., that $|X'|=|C'|$ and $|X'|-1\leq |Y'|\leq (1+\tilde{\eps})|X'|$.
In this inequality, the first part follows directly by the assumption on $X'$ and $C'$. As $\eps\llpoly\tilde{\eps}$, the second part is true if $|Q|\leq (1+\tilde{\eps}/2)|Q'|$ for each $Q,Q'\in \{X,Y,C\}$, which holds with high probability by a simple application of Lemma~\ref{chernoff-bound}. \ref{prop-pseud-basic-1} follows similarly.

\par Following Keevash, Pokrovskiy, Sudakov, and Yepremyan~\cite[Section~6]{KPSY}, it is straightforward to see that, with high probability, we have $|X|,|Y|,|C|=(p\pm \eps )n$ and, $\mathcal{H}(D[X',Y',C'])$ and $\mathcal{H}(D[X',C',Y'])$ are $(pn,p,\tilde \eps)$-typical for any choice of $X',Y',C'$ that is an $\eps$-small perturbation of $X,Y,C$. Thus, with high probability, $X,Y,C$ are so that \ref{prop-pseud-basic-2new} and \ref{prop-pseud-basic-2new-flip} hold for the hypergraphs associated to $D[X',Y',C']$ and $D[X',C',Y']$ for any choice of $X',Y',C'$.

\par Thus, in the remainder of the section, we focus only on the remaining five properties for each of the two models.

\par For both models, it is useful to make the following observations.  Since $D$ is $(n, \eps)$-typical, we have that $|C(D)|=(1 \pm 3\eps)n$. In particular, for every pair of vertices $u,v \in V(D)$ the number of colours that both $u$ and $v$ see as in- or out-edges (or $u$ as in/out and $v$ as out/in), is at least $(1-5\eps)n$. For every pair of colours $c_1, c_2 \in C(D)$ the number of vertices that have in- or out-edges in both colours $c_1$ and $c_2$ is at least $(1-3\eps)n$. We also have that the number of non-edges in $D$ is at most $3\eps n^2$.

\smallskip

\subsubsection{Proof that \ref{PrP}\textbf{.}\ref{PrP-1} holds with high probability}
\label{sec_flip}


\smallskip

 \noindent\textbf{\ref{PrP}.\ref{PrP-1}.\ref{prop-pseud-abs-prime-flip}:}
 We will prove something stronger that does not require having $\sqrt{n}$ exceptional edges $f$. Fix $y \in Y'$ and disjoint edges $e=x_1c_1$ and $f=x_2c_2$ in $E_y(D[X', C', Y'])$. Note that the existence of edges $e$ and $f$ in $D[X', C', Y']$ corresponds to the existence of edges $x_1y$ and $x_2y$ in $D$ such that $c_D(x_1y)=c_1$ and $c_D(x_2y)=c_2$. Working with this correspondence, to find the required structures in $D[X', C', Y']$, it suffices to find, in $D[X', Y', C']$, a collection of $\alpha n^2$ structures consisting of a rainbow $4$-cycle $x_1yx_2y''$ with edges of distinct colours $c_1,c_2,c_2',c_1'$ respectively, together with a path of length $4$ on vertices $y_1x_1'y'x_2'y_2$ with edges in distinct colours $c_1', c_1, c_2, c_2'$ respectively, where it is also permissible if $y_1=y_2$, yielding a $4$-cycle.  That is, it suffices to show that there are at least $\alpha n^2$ tuples of elements $(y', y'', y_1, y_2, x_1', c_1', x_2', c_2')$ such that $y', y'', y_1, y_2 \in Y'$, $x_1', x_2' \in X'$ and $c_1', c_2' \in C'$ such that all elements are distinct, and distinct from $x_1, x_2, c_1, c_2, y$, with the exception that $y_1=y_2$ is not forbidden, and the following all hold:
 \begin{itemize}
     \item $x_1'y', x_2'y' \in D$ with $c_D(x_1'y')=c_1$ and $c_D(x_2'y')=c_2$.
     \item $x_1y'', x_2y'' \in D$ with $c_D(x_1y'')=c_1'$ and $c_D(x_2y'')=c_2'$.
     \item $x_1'y_1, x_2'y_2 \in D$ with $c_D(x_1'y_1)=c_1'$ and $c_D(x_2'y_2)=c_2'$.
 \end{itemize}
We start by counting the number of these tuples which appear in $D$ (that is, in place of $(X', Y', C')$, we take $(V(D), V(D), C(D))$), and show, with high probability, that sufficiently many of these do in fact appear in $(X', Y', C')$. Note that the `$x$' vertices and `$y$' vertices are all distinct, thus when sampling so that $X$ and $Y$ are disjoint, each of these is equally likely to appear. Now, in $D$, there are at least $(1-3\eps)n-3$ vertices $y' \notin \{y, x_1, x_2\}$ with neighbours in both colours $c_1$ and $c_2$, yielding distinct vertices $x_1'$ and $x_2'$ respectively. Moreover, having fixed these, let $\bar{c}_1:=c(x_1'x_2')$ and $\bar{c}_2:=c(x_2'x_1')$. There are at least $(1-3\eps)n-20$ vertices $y'' \notin\{x_1, x_2, x_1', x_2', y, y'\}$ such that $x_1y''$ and $x_2y''$ are edges in $D$ with distinct colours $c_1', c_2' \notin \{c_1, c_2, \bar{c}_1, \bar{c}_2\}$. Given such a pair of colours $(c_1', c_2')$ note that there is no choice of $y''$ yielding the pair of colours $(c_1', d)$, or $(d, c_2')$ for any other colour $d$. Now, say such a pair $(c_1', c_2')$ is {\it bad} if for some $z \in \{x_1, x_2, y\}$ we have that $c(x_1'z)=c_1'$ or $c(x_2'z)=c_2'$, or $x_1'$ has no out-edge in colour $c_1'$, or $x_2'$ has no out-edge in colour $c_2'$. Since $|C(D)| \leq (1+3\eps)n$ as $D$ is $(n, \eps)$-typical, every vertex has an in-edge in all but at most $4\eps n$ colours in $D$, and so there are at most $8\eps n+6$ bad colour pairs. These correspond to at most $8\eps n + 6$ bad choices of $y''$, leaving at least $(1-11\eps)n-26$ good choices. That is, in $D$, there are at least $(1-15\eps)n^2$ choices for the desired tuples $(y', y'', y_1, y_2, x_1', c_1', x_2', c_2')$.

Thus, the expected number of tuples of the desired type appearing in $D[X,C,Y]$ is at least $(1-15\eps)p^8n^2$. Note that any vertex or colour (not including the fixed colour $y$ and fixed vertices $x_1, x_2, c_1, c_2$) appears in at most $O(n)$ of these tuples. Thus, by Azuma's inequality, with probability $1-\exp(-\Omega(\alpha n))$, there are at least $(1-16\eps)p^8n^2$  pairs of disjoint rainbow $4$-cycles of the desired type in $D[X,C,Y]$.

Moreover, again using that any vertex or colour (not including the fixed colour $y$ and fixed vertices $x_1, x_2, c_1, c_2$) appears in at most $O(n)$ of the desired tuples, when moving from $X, C$ and $Y$ to $X', C'$ and $Y'$ respectively, we have that at most $O(\eps n^2)$ of these have a colour in $Y\setminus Y'$, at most $O(\eps n^2)$ of these have a vertex in $X\setminus X'$ and at most $O(\eps n^2)$ of these use a vertex in $C \setminus C'$. Thus, since $\eps \ll p$, with high probability, we have at least $p^8n^2/2 \geq \alpha n^2$ such cycles for each $y \in Y'$ and pair $e,f \in D[X',C',\{y\}]$. Taking a union bound over all $y \in Y'$ and $(x_1, c_1), (x_2, c_2) \in X' \times C'$ gives the desired result.

 \smallskip

 \noindent\textbf{\ref{PrP}.\ref{PrP-1}.\ref{prop-pseud-add-2-flip}}: Note that we actually prove something slightly stronger, where $c$ can be any element from $C'$. In particular, given $y_0 \in Y'$, we only need to define $x_0 \in X'$ as, in the setting of $D[X', C', Y']$, we do not need to avoid any $c_0 \in C'$. Let $y_0 \in V(D)$. We start by observing that there exist at most one vertex $x_0 \in V(D)\setminus \{y_0\}$ with $t \geq 2n/3$ neighbours $u_1, \ldots, u_t$ such that $c(u_iy_0)=c(x_0u_i)$ for every $i \in [t]$. Indeed this follows by noting that if there are at least two such vertices, this contradicts that $D$ is properly coloured.

 Now, given $y_0 \in Y'$, fix $x \in X'\setminus \{x_0\}$, and $c \in C'$, noting that $x$ and $y_0$ are distinct. To find the required structures in $D[X', C', Y']$, it suffices to find $\alpha n$ disjoint tuples of distinct elements $(x_1, x_2, c_1, c_2, y_1, y_2, y_3) \in (X')^2 \times (C')^2 \times (Y')^3$ such that, in $D[X',Y',C']$, $y_2x_1y_0x_2y_3$ is a path with edges in colours $c_2, c_1, c_2, c$ respectively, and $xy_1$ is an edge in colour $c_1$.

 To do this, we first show that, given any $\bar{X} \subseteq X'\setminus\{x\}, \bar{Y} \subseteq Y'\setminus\{y_0\}$ and $\bar{C} \subseteq C'\setminus \{c\}$ with $|\bar{X}|, |\bar{C}|, |\bar{Y}| \leq 3\sqrt{\alpha} n$, we can find the desired structure with distinct elements $(x_1, x_2, c_1, c_2, y_1, y_2, y_3) \in (V(D)\setminus (\bar{X} \cup \{x\}))^2 \times (C(D)\setminus (\bar{C} \cup \{c\}))^2 \times (V(D) \setminus (\bar{Y} \cup \{y_0\}))^3$. This ensures that there are at least $\sqrt{\alpha} n$ disjoint collections of the desired type in $D$.

We find the desired structure constructively. We start by finding a suitable $c_1 \in C\setminus(\bar{C} \cup \{c\})$ which fixes $x_1 \neq y_1$ with $x_1, y_1 \in V(D) \setminus (\bar{X} \cup \bar{Y} \cup \{y_0, x\})$ such that $c(xy_1)=c(x_1y_0)=c_1$. Note that by our initial observation, in $D$ there are at least $n/3-O(\sqrt{\alpha} n)$ such choices for $c_1$. Of these, also by our initial observation, at most one such choice fixes $x_1$ in such a way that, when trying to choose $c_2, x_2, y_2$ we have more than $2n/3$ choices forcing $x_2=y_2$. Moreover, there is at most one other choice of $c_1$ fixing $x_1$ so that for at least $2n/3$ choices of $c_2$ we find that the resulting $x_2, y_2$ satisfy $c(x_2y_2)=c$. In all other cases we are left with at least $n/3-O(\sqrt{\alpha} n)$ choices of a colour $c_2 \in C\setminus(\bar{C} \cup \{c, c_1\})$ fixing vertices $x_2 \neq y_2$ with $x_2, y_2 \in V(D) \setminus (\bar{X} \cup \bar{Y} \cup \{y_0, y_1, x, x_1\})$ such that $c(x_1y_2)=c(x_2y_0)=c_2$ and $c(x_2y_2) \neq c$. That is, we have at least $n/4$ choices for $c_1$ which result in at least $n/4$ choices for $c_2\in C\setminus(\bar{C} \cup \{c, c_1\})$ fixing vertices $x_2 \neq y_2$ with $x_2, y_2 \in V(D) \setminus (\bar{X} \cup \bar{Y} \cup \{y_0, y_1, x, x_1\})$ such that $c(x_1y_2)=c(x_2y_0)=c_2$ and $c(x_2y_2) \neq c$. Fixing one such choice of $c_1$, each such subsequent choice of $c_2$ fixes a distinct vertex $x_2 \in V(D) \setminus (\bar{X} \cup \bar{Y} \cup \{y_0, y_1, x, x_1\})$ (also with $x_2 \neq y_2$ and $c(x_2y_2) \neq c$). At most $O(\sqrt{\alpha}n)$ of these choices result in a vertex $x_2$ which either has no out-neighbour in colour $c$, or has an out-neighbour in colour $c$ which lies in $\bar{X} \cup \bar{Y} \cup \{x, x_1, y_0, y_1\}$, leaving many choices for $c_2$ which result in an edge $x_2y_3$ such that $c(x_2y_3)=c$ and $y_3 \in V(D)\setminus (\bar{X} \cup \bar{Y} \cup \{y_0, y_1, y_2, x, x_1, x_2\})$. In particular, this yields the desired structure.

 Now, given that we have at least $\sqrt{\alpha}n$ such disjoint tuples in $D$, we see that the expected number found satisfying $x_1, x_2 \in X$, $y_1, y_2, y_3 \in Y$ and $c_1, c_2 \in C$ is at least $p^7\sqrt{\alpha}n$. By Chernoff's bound, it follows that,  with probability at least $1-\exp(-\Omega(\sqrt{n}))$, at least $(1-\eps)p^7\sqrt{\alpha}n \geq 2\alpha n$ appear in $D[X,Y,C]$. Moreover, since each of these structures is disjoint, at most $3\eps n \leq \alpha n$ can be removed by the perturbations of $X, Y$ and $C$ to obtain $X', Y'$ and $C'$. Recall that these $\alpha n$ structures found in $D[X', Y', C']$ correspond precisely to $\alpha n$ disjoint structures of the required type found in $D[X', C', Y']$. Thus taking a union bound over all $y_0 \in Y'$, $x \in X'\setminus \{x_0\}$ and $c \in C'$ yields the result.

 \smallskip

  \noindent\textbf{\ref{PrP}.\ref{PrP-1}.\ref{prop-pseud-add-new-1-flip}}: We start with the following stronger version of \ref{prop-pseud-add-new-1-flip}.
\begin{enumerate}[label = \textbf{\Alph{propcounter}6'}]
\item For any $y_0\in Y'$, $0\leq k\leq 200$, and any $\bar{Y}\subset Y'\setminus \{y_0\}$ with $|\bar{Y}|=k$, there are disjoint sets $\bar{X}_1,\ldots,\bar{X}_{\alpha n}\subset X'$ and $\bar{C}_1,\ldots,\bar{C}_{\alpha n}\subset C'$ such that,
for each $i\in [\alpha n]$, $|\bar{X}_i|=|\bar{C}_i|=k+1$ and $D[\bar{X}_i, \bar{C}_i, \bar{Y} \cup \{y_0\}]$ contains both a matching of $k+1$ colour-$y_0$ edges and a $\bar{Y}$-rainbow matching with $k$ edges.\label{prop-pseud-add-new-1-flip-strong}
\end{enumerate}

 Fix $y_0 \in Y'$ and $\bar{Y}=\{y_1, \ldots, y_k\} \subseteq Y' \setminus \{y_0\}$. Translating from $D[X', C', Y']$ to $D[X', Y', C']$, to prove \ref{prop-pseud-add-new-1-flip-strong} it is sufficient to show that, given any subsets $\bar{X} \subseteq X'$ and $\bar{C} \subseteq C'$, with $|\bar{X}|, |\bar{C}| \leq (k+1)\sqrt{\alpha} n$, there are distinct vertices $\{x_1, \ldots, x_{k+1}\} \in X' \setminus \bar{X}$, and colours $\{c_1, \ldots, c_{k+1}\} \in C'\setminus \bar{C}$ such that $D[X',Y',C']$ has a tree rooted at $y_0$ with edges $x_iy_0$ in colour $c_i$  for each $i \in [k+1]$ and $x_iy_i$ has colour $c_{i+1}$ for every $i \in [k]$.

 We do this by first showing there are sufficiently many choices in $D$. That is, given that there is at least one choice of distinct vertices $\{x_1, \ldots, x_{k+1}\} \in V(D) \setminus (\bar{X} \cup \bar{Y} \cup \{y_0\})$, and colours $\{c_1, \ldots, c_{k+1}\} \in C(D)\setminus \bar{C}$, it follows that we can find at least $\sqrt{\alpha}n$ of the required structures, each appropriately disjoint from one another, in $D$. Note that the desired structure is fixed by an appropriate choice of $x_1$, and any vertex $v \in V(D)$ or colour $c \in C(D)$ appears in at most $k$ of the distinct structures given by different choices of $x_1$. In particular, fixing $x_1$ fixes edges $x_1y_0$ and $x_1y_1$ in colours $c_1$ and $c_2$ respectively, which in turn fixes $x_2$ so that $x_2y_0$ has colour $c_2$, which fixes $c_3$ to be the colour of the edge $x_2y_2$, and so on. Thus it remains to count the number of appropriate choices of $x_1$.

 We do this by induction on $k$, starting with $k=1$. In fact we prove that, for each $k$, there are at least $(1-7k^2\sqrt{\alpha})n$ valid choices for $x_1$.

 To this end, note that, in $D$, $y_0$ and $y_1$ have at least $(1-3\eps)n$ common in-neighbours. Of these, at most $2k\sqrt{\alpha} n$ are in $\bar{X} \cup \bar{Y} \cup \{y_0\}$, making them bad choices. Of the remaining choices, at most $2k\sqrt{\alpha} n$ yield a pair of colours $(c_1, c_2)$ (on the edges $x_1y_0$ and $x_1y_1$) such that $\{c_1, c_2\} \cap  \bar{C} \neq \emptyset$, and at most one choice yields $c_2=c(y_1y_0)$. Now for the remaining at least $(1- 3\eps - 4k\sqrt{\alpha})n-1$ choices for $x_1$, each yields an ordered pair $(c_1, c_2)$ such that any colour in $C(D) \setminus \bar{C}$ appears at most once as the first element and at most once as the second element. (That is, given $x_1$ yielding $(c_1, c_2)$ such that $c(x_1y_0)=c_1$ and $c(x_1y_1)=c_2$, there is no $x_1'$ yielding $(c_1', c_2)$ with   $c(x_1'y_0)=c_1'$ and $c(x_1'y_1)=c_2$, or $(c_1, c_2')$ with $c(x_1'y_0)=c_1$ and $c(x_1'y_1)=c_2'$. This all follows since $D$ is properly edge-coloured.) Thus, any of these pairs for which there is a $c_2$-colour in-edge incident to $y_0$ avoiding $\bar{X} \cup \bar{Y}$ yields the desired structure. (Note that such an in-edge cannot have its other endpoint as $x_1$ since $c(x_1y_0)=c_1 \neq c_2$, and $c_2$ was chosen so that $c(y_1y_0) \neq c_2$.) Since there are at most $4\eps n$ colours that any vertex does not see, this rules out at most $4\eps n + 2k\sqrt{\alpha} n$ additional pairs $(c_1, c_2)$ and so at most $4\eps n+2k\sqrt{\alpha}n$ choices for $x_1$. Thus we are left with at least $(1- 7\eps - 6k\sqrt{\alpha})n-1 \geq (1-7\sqrt{\alpha})n$ choices for $x_1$ when $k=1$, as required.

 Suppose now that the statement holds for all $k'<k$, and $k \leq 200$. Fix $y_k \in \bar{Y}$ and let $\bar{Y}'=\bar{Y}\setminus\{y_k\}$. By induction there are at least $(1-7(k-1)^2\sqrt{\alpha})n$ choices for $x_1$ leading to valid distinct choices of $x_2, \ldots, x_{k}$ and $c_1, \ldots, c_k$ such that $c(x_iy_0)=c_i$ for every $i \in [k]$ and $c(x_iy_i)=c_{i+1}$ for every $i \in [k-1]$. Of these, they only do not extend to a valid structure for $\bar{Y}$ if either (i) $x_i = y_k$ for some $i \in [k]$, which can happen for $O(k)=O(1)$ such choices of $x_1$, (ii) $x_ky_k \notin E(D)$ which can happen for at most $\eps n$ choices of $x_1$, (iii) $c(x_ky_k) \in \bar{C} \cup \{c_1, \ldots, c_k\}$, which can happen for at most $2k\sqrt{\alpha}n$, (iv) $c(x_ky_k)=c(y_ky_0)$ which can happen for at most one of these, (v) $y_0$ has no incident in-edge in colour $c(x_ky_k)$, which can happen for at most $4\eps n$ choices of $x_1$ or (vi) the incident in-edge to $y_0$ in colour $c(x_ky_k)$ has its other endpoint in $\bar{Y} \cup \bar{X}$, which can happen for at most $2k\sqrt{\alpha}n$ choices. In total this excludes an additional at most $5k\sqrt{\alpha}n$ choices leaving at least $(1-7k^2\sqrt{\alpha})n$ valid choices, as desired.

 Now, this shows that in $D$, there are at least $\sqrt{\alpha} n$ choices of disjoint sets $\bar{X_i}$ in $V(D) \setminus (\bar{Y} \cup \{y_0\})$ and $\bar{C_i}$ in $C(D)$ of the desired form.
 The probability that each of these appears in $D[X,(\bar{Y} \cup \{y_0\}),C]$ is then $p^{2k+2}$ and this is independent since they are all disjoint.
 So by Chernoff's bound, with high probability there are at least $\sqrt{\alpha} p^{2k+2} n/2$ such sets in $D[X,(\bar{Y} \cup \{y_0\}),C]$. Moreover, modifying $X, Y$ and $C$ to $X', Y'$ and $C'$ by $\eps n$ elements, maintaining disjointness between $X'$ and $Y'$ can only affect $2\eps n$ of these structures, leaving at least $\sqrt{\alpha} p^{2k} n/2-2\eps n \geq \alpha n$ such sets in $D[X',(\bar{Y} \cup \{y_0\}),C']$, each of which yields the desired structures in $D[X', C', Y']$. Taking a union bound over all choices for $\{y_0\}$ and $\bar{Y}$, the statement still holds.

 Note that the proof of \ref{prop-pseud-add-new-1-flip} is complete, since it follows directly from \ref{prop-pseud-add-new-1-flip-strong}. Having established \ref{prop-pseud-add-new-1-flip-strong}, the proofs of \ref{prop-pseud-add-3-flip} and \ref{prop-alt-flip} follow almost verbatim those of the equivalent properties in \cite[Proposition~3.12]{montgomery2023proof}.

\smallskip

  \noindent\textbf{\ref{PrP}.\ref{PrP-1}.\ref{prop-pseud-add-3-flip}}: Let $y_0, y' \in Y'$ be distinct. Let $s$ be the largest integer for which there are disjoint sets $X_1, \ldots, X_s \subseteq X'$ and $C_1, \ldots, C_s \subseteq C'$ with $|X_i|=|C_i|=4$ and $Y_1, \ldots, Y_s \subseteq Y' \setminus \{y_0, y'\}$ with $|Y_i|=3$ for all $i \in [s]$, for which $D[X_i, C_i, Y_i \cup \{y_0, y'\}]$ contains both a matching of $4$ colour $y_0$ edges and an exactly-$Y_i \cup \{y'\}$-rainbow matching. Suppose for a contradiction that $s < \alpha n/12$, and let $\bar{X}=\cup_{i \in [s]} X_i$, $\bar{C}=\cup_{i \in [s]} C_i$ and $\bar{Y}=\cup_{i \in [s]} Y_i$, so that $|\bar{X}|, |\bar{C}|< \alpha n/3$ and $|\bar{Y}| \leq \alpha n/4$. Now from \ref{prop-pseud-add-new-1-flip-strong} with $k=1$ and $\bar{Y}=\{y'\}$, we can find sets $\hat{X} \subseteq X' \setminus (\bar{X} \cup \{x_0\})$ and $\hat{C} \subseteq C'\setminus (\bar{C} \cup \{c_0\})$ with $|\hat{X}|=|\hat{C}|=2$, where $x_0, c_0$ are as in \ref{prop-pseud-add-2-flip}, so that $D[\hat{X}, \hat{C}, \{y_0, y'\}]$ contains two edges of colour $y_0$ and one of colour $y'$. Let $x \in \hat{X}$ and $c \in \hat{C}$ be such that the edge in $D[\hat{X}\setminus\{x\}, \hat{C}\setminus\{c\}, \{y_0, y'\}]$ has colour $y'$. Now, using \ref{prop-pseud-add-2-flip}, there are choices for sets $\hat{X}' \subseteq X' \setminus (\bar{X} \cup \hat{X})$, $\hat{C}' \subseteq C' \setminus (\bar{C} \cup \hat{C})$ with $|\hat{X}'|=|\hat{C}'|=2$ and $\hat{Y}' \subseteq Y' \setminus (\bar{Y} \cup \{y_0, y'\})$ with $|\hat{Y'}|=3$ such that $D[\hat{X}', \hat{C}', \{y_0\}]$ has a perfect matching and $D[\hat{X}' \cup \{x\}, \hat{C}' \cup \{c\}, \hat{Y}']$ contains an exactly-$\hat{Y}'$-rainbow matching avoiding using the edge $xc$. Then letting $X_{s+1}=\hat{X} \cup \hat{X}'$, $C_{s+1}=\hat{C} \cup \hat{C}'$ and $Y_{s+1}=\hat{Y}'$ contradicts the choice of $s$ showing that $s \geq \alpha n/12$.

\smallskip

  \noindent\textbf{\ref{PrP}.\ref{PrP-1}.\ref{prop-alt-flip}}: Set $k=100$ and let $y_0 \in Y'$. Take the largest $r \in \mathbb{N}$ for which there are disjoint sets $X_1, \ldots, X_r \subseteq X'$, $C_1, \ldots, C_r$ in $C'$ and $Y_1, \ldots, Y_r$ in $Y' \setminus \{y_0\}$ with $|X_i|=|C_i|=|Y_i|=k$ for each $i \in [r]$, such that $D[X_i, C_i, Y_i \cup \{y_0\}]$ contains an exactly-$Y_i$-rainbow matching and a perfect matching of colour $y_0$ edges. Fix such sets $X_i, C_i, Y_i$, $i \in [r]$. We first show that $r \geq 2\alpha^2n$. Suppose not, for a contradiction. Let $s \leq 25$ be maximal such that there are disjoint sets $X_1', \ldots, X_s'$ in $X' \setminus (X_1 \cup \ldots \cup X_r)$, $C_1', \ldots, C_s'$ in $C' \setminus (C_1 \cup \ldots \cup C_r)$ and $Y_1', \ldots, Y_s'$ in $Y' \setminus (Y_1 \cup \ldots \cup Y_r \cup \{y_0\})$ with $|X_i'|=|C_i'|=|Y_i'|=4$ for each $i \in [s]$, so that $D[X_i', C_i', Y_i' \cup \{y_0\}]$ contains a matching of $4$ colour-$y_0$ edges and an exactly-$Y_i'$-rainbow matching. Note that if $s<25$ then $|\cup_{i \in [r]} X_i|,|\cup_{i \in [r]} C_i|,|\cup_{i \in [r]} Y_i|\leq 2k\alpha^2n \leq \alpha n/2$, and so \ref{prop-pseud-add-3-flip} immediately gives the required contradiction.  Thus $s=25$. But now, letting $X_{r+1}=\cup_{i \in [s]} X_i'$, $C_{r+1}=\cup_{i \in [s]} C_i'$ and $Y_{r+1}=\cup_{i \in [s]} Y_i'$ gives a contradiction to the choice of $r$, so we have that $r \geq 2\alpha^2n$.

  Now we show that the sets $X_i, C_i$ and $Y_i$ have the desired property. Indeed, let $\bar{Y} \subseteq Y' \setminus \{y_0\}$, with $|\bar{Y}| \leq k$. Note that there are at least $\alpha^2n$ values of $i \in [r]$ for which $\bar{Y} \cap Y_i = \emptyset$, and by \ref{prop-pseud-add-new-1-flip-strong} there are disjoint sets $\hat{X}_1, \ldots \hat{X}_{\alpha n} \subseteq X'$ and $\hat{C}_1, \ldots \hat{C}_{\alpha n} \subseteq C'$ with $|\hat{X}_i|=|\hat{C_i}|=k+|\bar{Y}|+1$ and $D[\hat{X}_i, \hat{C}_i, \bar{Y} \cup Y_i \cup \{y_0\}]$ contains both a matching of $k+|\bar{Y}|+1$ colour-$y_0$ edges and a $(\bar{Y}\cup Y_i)$-rainbow matching with $k+ |\bar{Y}|$ edges, completing the proof.

\subsubsection{Proof that \ref{PrP}.\ref{PrP-2} holds with high probability}\label{sec_nonflip}

Since $D$ is $(n, \eps)$-typical, we have that $|C(D)|=(1 \pm 2\eps)n$. In particular, for every pair of vertices $u,v \in V(D)$ the number of colours that both $u$ and $v$ see as edges is at least $(1-4\eps)n$. We also have that the number of non-edges in $D$ is at most $2\eps n^2$.

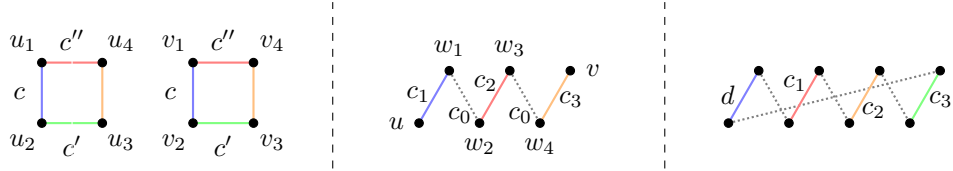
\begin{figure}[h]
\begin{center}\begin{tikzpicture}[scale=0.8]
\def\vxrad{0.07cm}
\def\horunit{1.2}
\def\edgelength{0.4}
\def\betweenrows{0.5}


\def\gapratio{1}
\def\biggergap{2}

\foreach \num in {0,1}
{
\coordinate (B\num) at ($(0,0)+\num*4*\gapratio*(1,0)+\gapratio*(-0.5,0.5)$);
\coordinate (C\num) at ($(0,0)+\num*4*\gapratio*(1,0)+\gapratio*(0.5,0.5)$);
\coordinate (E\num) at ($0.5*(B\num)+0.5*(C\num)+0.86602540*\gapratio*(0,1)$);
}

\coordinate (C1) at ($0.5*(B0)+0.5*(C0)-(0,1)$);
\coordinate (C2) at ($(C1)+(0,1)$);
\coordinate (C3) at ($0.5*\gapratio*(1,0)$);
\coordinate (C4) at ($-1*(C3)$);
\coordinate (X1) at (B0);
\coordinate (X2) at (C0);
\coordinate (X3) at ($(B0)-(0,1)$);
\coordinate (X4) at ($(C0)-(0,1)$);

\def\bit{0.375}
\draw ($(C2)+(0,\bit)$) node {$c''$};
\draw ($(C1)+(0,-\bit)$) node {$c'$};
\draw ($(C4)-(\bit,0)$) node {$c$};

\draw [white] ($(C3)+2.5*(\gapratio,0)$) -- ($(C3)+2.5*(\gapratio,0)$);


\def\sm{0.2};

\foreach \num/\numm/\coll/\fillcoll in {0/0/red!50/red!25,0/1/blue!50/blue!25,0/2/green!50/green!25,0/3/orange!50/orange!25}
{
\begin{scope}[transform canvas={rotate=\numm*90}]
{
\draw [thick,\coll] (C\num) -- (B\num);
}
\end{scope}
}
\foreach \num in {0}
{
\foreach \numm in {0,1,2,3}
{
\begin{scope}[transform canvas={rotate=\numm*90}]
{
\draw [fill] (B\num) circle [radius=\vxrad];
}
\end{scope}
}
}

\begin{scope}[transform canvas={shift={(2*\gapratio,0)}}]
\foreach \num/\numm/\coll/\fillcoll in {0/0/red!50/red!25,0/1/blue!50/blue!25,0/2/green!50/green!25,0/3/orange!50/orange!25}
{
\begin{scope}[transform canvas={rotate=\numm*90}]
{
\draw [thick,\coll] (C\num) -- (B\num);
}
\end{scope}
}
\foreach \num in {0}
{
\foreach \numm in {0,1,2,3}
{
\begin{scope}[transform canvas={rotate=\numm*90}]
{
\draw [fill] (B\num) circle [radius=\vxrad];
}
\end{scope}
}
}
\draw ($(C2)+(0,\bit)$) node {$c''$};
\draw ($(C1)+(0,-\bit)$) node {$c'$};
\draw ($(C4)-(\bit,0)$) node {$c$};
\end{scope}

\draw [white] (0,1.25) -- (0,-1.25);

\def\bitt{0.625};
\draw ($(X1)+\bitt*(-0.5,0.5)$) node {$u_1$};
\draw ($(X2)+\bitt*(0.5,0.5)$) node {$u_4$};
\draw ($(X3)+\bitt*(-0.5,-0.5)$) node {$u_2$};
\draw ($(X4)+\bitt*(0.5,-0.5)$) node {$u_3$};

\begin{scope}[transform canvas={shift={(2*\gapratio,0)}}]
\draw ($(X1)+\bitt*(-0.5,0.5)$) node {$v_1$};
\draw ($(X3)+\bitt*(-0.5,-0.5)$) node {$v_2$};
\draw ($(X4)+\bitt*(0.5,-0.5)$) node {$v_3$};
\draw ($(X2)+\bitt*(0.5,0.5)$) node {$v_4$};
\end{scope}

\end{tikzpicture}\hspace{0.3cm}
\begin{tikzpicture}
\draw [white] (-0.5,0) -- (0.5,0);
\draw [dashed] (0,-1.2) -- (0,1);.2
\end{tikzpicture}
\begin{tikzpicture}[scale=0.8]
\def\vxrad{0.07cm}
\def\horunit{1.2}
\def\edgelength{0.4}
\def\betweenrows{0.5}


\def\gapratio{1}
\def\biggergap{2}

\def\hggt{2.5}

\def\hanglow{-1.75}
\def\widdd{9}
\draw [white] (\widdd,-\hanglow) -- (\widdd,-\hanglow-\hggt);

\foreach \num in {1}
{
\coordinate (A\num) at ($(0,0)+\num*4*\gapratio*(1,0)$);
\coordinate (B\num) at ($(0,0)+\num*4*\gapratio*(1,0)+\gapratio*(1,0)$);
\coordinate (C\num) at ($(0,0)+\num*4*\gapratio*(1,0)+\gapratio*(2,0)$);
\coordinate (D\num) at ($(0,0)+\num*4*\gapratio*(1,0)+\gapratio*(3,0)$);
\coordinate (E\num) at ($0.5*(B\num)+0.5*(C\num)+0.86602540*\gapratio*(0,1)$);
\coordinate (F\num) at ($0.5*(B\num)+0.5*(C\num)-0.86602540*\gapratio*(0,1)$);
}
\foreach \num in {2}
{
\coordinate (B\num) at ($(0,0)+\num*4*\gapratio*(1,0)+\gapratio*(0,0)$);
\coordinate (D\num) at ($(0,0)+\num*4*\gapratio*(1,0)+\gapratio*(1,0)$);
\coordinate (A\num) at ($0.5*(D1)+0.5*(B\num)+0.86602540*\gapratio*(0,1)$);
\coordinate (C\num) at ($0.5*(B\num)+0.5*(D\num)+0.86602540*\gapratio*(0,1)$);
\coordinate (E\num) at ($0.5*(B\num)+0.5*(C\num)+0.86602540*\gapratio*(0,1)$);
\coordinate (F\num) at ($0.5*(B\num)+0.5*(C\num)-0.86602540*\gapratio*(0,1)$);
}
\coordinate (G1) at ($0.5*(D1)+0.5*(A2)+0.86602540*\gapratio*(0,1)$);
\coordinate (G2) at ($0.5*(D1)+0.5*(A2)+0.86602540*\gapratio*(0,1)+4*\gapratio*(1,0)$);
\coordinate (A30) at ($(A2)+2*\gapratio*(1,0)$);
\foreach \num in {3,4}
{
\coordinate (A\num) at ($(0,0)+\num*4*\gapratio*(1,0)+\biggergap*(1,0)$);
\coordinate (B\num) at ($(0,0)+\num*4*\gapratio*(1,0)+\gapratio*(1,0)+\biggergap*(1,0)$);
\coordinate (C\num) at ($(0,0)+\num*4*\gapratio*(1,0)+\gapratio*(2,0)+\biggergap*(1,0)$);
\coordinate (D\num) at ($(0,0)+\num*4*\gapratio*(1,0)+\gapratio*(3,0)+\biggergap*(1,0)$);
\coordinate (E\num) at ($0.5*(B\num)+0.5*(C\num)+0.86602540*\gapratio*(0,1)$);
\coordinate (F\num) at ($0.5*(B\num)+0.5*(C\num)-0.86602540*\gapratio*(0,1)$);
}

\coordinate (G3) at ($0.5*(D3)+0.5*(A4)+0.86602540*\gapratio*(0,1)$);

\def\sm{0.2};
\foreach \num in {2}
{
\draw [thick,black!50,fill=black!25,densely dotted] (A\num) -- (B\num);
\draw [thick,black!50,fill=black!25,densely dotted] (C\num) -- (D\num);
}

\foreach \num in {2}
{
\draw [thick,red!50,fill=red!25] (C\num) -- (B\num);
}

\foreach \num/\numplus in {1/2}
{
\draw [thick,blue!50,fill=blue!25] (A\numplus) -- (D\num);
}

\foreach \num/\numplus in {2/3}
{
\draw [thick,orange!50,fill=orange!25] (A30) -- (D\num);
}

\def\bit{0.375}
\draw ($0.5*(D1)+0.5*(A2)-(0.7*\bit,-0.1*\bit)$) node {$c_1$};
\draw ($0.5*(B2)+0.5*(C2)-(0.4*\bit,-0.6*\bit)$) node {$c_2$};
\draw ($0.5*(A30)+0.5*(D2)+(0.7*\bit,-0.1*\bit)$) node {$c_3$};
\draw ($0.5*(A2)+0.5*(B2)-(0.2*\bit,0.8*\bit)$) node {$c_0$};
\draw ($0.5*(C2)+0.5*(D2)-(0.2*\bit,0.8*\bit)$) node {$c_0$};

\draw ($(D1)+(-\bit,0)$) node {$u$};
\draw ($(B2)-(0,\bit)$) node {$w_2$};
\draw ($(A2)+(0,\bit)$) node {$w_1$};
\draw ($(C2)+(0,\bit)$) node {$w_3$};
\draw ($(D2)-(0,\bit)$) node {$w_4$};
\draw ($(A30)+(\bit,0)$) node {$v$};

\foreach \lett in {A,B,C,D}
\foreach \num in {2}
{
\draw [fill] (\lett\num) circle [radius=\vxrad];
}
\foreach \lett in {D}
\foreach \num in {1}
{
\draw [fill] (\lett\num) circle [radius=\vxrad];
}


\draw [fill] (A30) circle [radius=\vxrad];

\end{tikzpicture}
\begin{tikzpicture}
\draw [white] (-0.5,0) -- (0.5,0);
\draw [dashed] (0,-1.2) -- (0,1);.2
\end{tikzpicture}
\begin{tikzpicture}[scale=0.8]
\def\vxrad{0.07cm}
\def\horunit{1.2}
\def\edgelength{0.4}
\def\betweenrows{0.5}


\def\gapratio{1}
\def\biggergap{2}

\def\hggt{2.5}

\def\hanglow{-1.75}
\def\widdd{9}
\draw [white] (\widdd,-\hanglow) -- (\widdd,-\hanglow-\hggt);

\foreach \num in {1}
{
\coordinate (A\num) at ($(0,0)+\num*4*\gapratio*(1,0)$);
\coordinate (B\num) at ($(0,0)+\num*4*\gapratio*(1,0)+\gapratio*(1,0)$);
\coordinate (C\num) at ($(0,0)+\num*4*\gapratio*(1,0)+\gapratio*(2,0)$);
\coordinate (D\num) at ($(0,0)+\num*4*\gapratio*(1,0)+\gapratio*(3,0)$);
\coordinate (E\num) at ($0.5*(B\num)+0.5*(C\num)+0.86602540*\gapratio*(0,1)$);
\coordinate (F\num) at ($0.5*(B\num)+0.5*(C\num)-0.86602540*\gapratio*(0,1)$);
}
\foreach \num in {2}
{
\coordinate (B\num) at ($(0,0)+\num*4*\gapratio*(1,0)+\gapratio*(0,0)$);
\coordinate (D\num) at ($(0,0)+\num*4*\gapratio*(1,0)+\gapratio*(1,0)$);
\coordinate (A\num) at ($0.5*(D1)+0.5*(B\num)+0.86602540*\gapratio*(0,1)$);
\coordinate (A30) at ($(A2)+2*\gapratio*(1,0)$);
\coordinate (C\num) at ($0.5*(B\num)+0.5*(D\num)+0.86602540*\gapratio*(0,1)$);
\coordinate (E\num) at ($(0,0)+\num*4*\gapratio*(1,0)+\gapratio*(2,0)$);
\coordinate (F\num) at ($(A2)+3*\gapratio*(1,0)$);
}
\coordinate (G1) at ($0.5*(D1)+0.5*(A2)+0.86602540*\gapratio*(0,1)$);
\coordinate (G2) at ($0.5*(D1)+0.5*(A2)+0.86602540*\gapratio*(0,1)+4*\gapratio*(1,0)$);

\foreach \num in {3,4}
{
\coordinate (A\num) at ($(0,0)+\num*4*\gapratio*(1,0)+\biggergap*(1,0)$);
\coordinate (B\num) at ($(0,0)+\num*4*\gapratio*(1,0)+\gapratio*(1,0)+\biggergap*(1,0)$);
\coordinate (C\num) at ($(0,0)+\num*4*\gapratio*(1,0)+\gapratio*(2,0)+\biggergap*(1,0)$);
\coordinate (D\num) at ($(0,0)+\num*4*\gapratio*(1,0)+\gapratio*(3,0)+\biggergap*(1,0)$);
\coordinate (E\num) at ($0.5*(B\num)+0.5*(C\num)+0.86602540*\gapratio*(0,1)$);
\coordinate (F\num) at ($0.5*(B\num)+0.5*(C\num)-0.86602540*\gapratio*(0,1)$);
}

\coordinate (G3) at ($0.5*(D3)+0.5*(A4)+0.86602540*\gapratio*(0,1)$);

\def\sm{0.2};
\foreach \num in {2}
{
\draw [thick,black!50,fill=black!25,densely dotted] (A\num) -- (B\num);
\draw [thick,black!50,fill=black!25,densely dotted] (C\num) -- (D\num);
\draw [thick,black!50,fill=black!25,densely dotted] (A30) -- (E\num);
\draw [thick,black!50,fill=black!25,densely dotted] (D1) -- (F\num);
}

\foreach \num in {2}
{
\draw [thick,red!50,fill=red!25] (C\num) -- (B\num);
}

\foreach \num/\numplus in {1/2}
{
\draw [thick,blue!50,fill=blue!25] (A\numplus) -- (D\num);
}

\foreach \num/\numplus in {2/3}
{
\draw [thick,orange!50,fill=orange!25] (A30) -- (D\num);
}

\foreach \num/\numplus in {2/3}
{
\draw [thick,green!50,fill=green!25] (E\num) -- (F\num);
}

\def\bit{0.375}
\draw ($0.5*(D1)+0.5*(A2)-(0.7*\bit,-0.1*\bit)$) node {$d$};
\draw ($0.5*(B2)+0.5*(C2)-(0.4*\bit,-0.6*\bit)$) node {$c_1$};
\draw ($0.5*(A30)+0.5*(D2)+(0.4*\bit,-0.6*\bit)$) node {$c_2$};
\draw ($0.5*(A30)+0.5*(D2)+(0.7*\bit,-0.1*\bit)+(\gapratio,0)$) node {$c_3$};


\foreach \lett in {A,B,C,D,E,F}
\foreach \num in {2}
{
\draw [fill] (\lett\num) circle [radius=\vxrad];
}
\foreach \lett/\num in {D/1}
{
\draw [fill] (\lett\num) circle [radius=\vxrad];
}


\draw [fill] (A30) circle [radius=\vxrad];

\end{tikzpicture}
\end{center}

\vspace{-0.4cm}

\caption{Relevant structures for \ref{prop-pseud-abs-prime}, \ref{prop-pseud-add-2} and \ref{prop-pseud-add-3} respectively in the definition of proper-pseudorandomness (see Definition~\ref{defn:pseud}). In the picture for \ref{prop-pseud-abs-prime}, we have $e=u_1u_2$ and $f=v_1v_2$. The edges $u_4u_3$ and $v_4v_3$ are depicted to have the same colour, but \ref{prop-pseud-abs-prime} does not require this. In the picture for \ref{prop-pseud-add-2}, we would take, for some $i$, $V_i=\{w_1,w_2,w_3,w_4\}$ and $C_i=\{c_1,c_2,c_3\}$.
In the picture for \ref{prop-pseud-add-3}, the dotted lines have colour $c_0$. In \ref{prop-pseud-add-3} we would take, for some $i$, $V_i$ to be the set of pictured vertices and $C_i=\{c_1,c_2,c_3\}$.   This figure first appeared in \cite{montgomery2023proof}.
}\label{fig:pseudpics}
\end{figure}
\smallskip
Observe that all subgraphs depicted above in Figure~\ref{fig:pseudpics} are bipartite.
In order to verify that the random sample $D[X,Y,C]$ contains such subgraphs, we will show that $D$ contains \textit{antidirected} versions of the above subgraphs, meaning versions obtained by fixing a bipartition of the subgraph, and directing all edges from one part to the other.

 \smallskip

 \noindent\textbf{\ref{PrP}.\ref{PrP-2}.\ref{prop-pseud-abs-prime}}: As in Section \ref{sec_flip}, we prove something stronger that does not require having $\sqrt{n}$ exceptional edges $f$. For fixed choices of $c$ and two distinct edges $e,f\in E_c(D)$ with $e=\{u_1,u_2\}$ and $f=\{v_1,v_2\}$, observe that there exist at least $(1-8\eps)n^2$ distinct choices of $u_3,u_4,v_3,v_4,c',c''$ that give two antidirected cycles (not necessarily vertex-disjoint) as depicted in Figure~\ref{fig:pseudpics} where $c,c'$ and $c''$ are distinct. Note also that fixing any of the variables $u_3,u_4,v_3,v_4,c',c''$ to be a particular value, there are at most $n$ distinct values of the remaining variables that yield $4$-cycles of the prescribed colour type. Furthermore, we observe that there exist at most $\binom{8}{2}n$ choices of $u_3,u_4,v_3,v_4,c',c''$ with the prescribed colour type that additionally satisfy that two distinct vertex variables have the same value. As a consequence, we have that there are at least $(1-10\eps)n^2$ distinct choices of $u_3,u_4,v_3,v_4,c',c''$ giving two vertex-disjoint cycles as depicted in Figure~\ref{fig:pseudpics} that are antidirected. Hence, McDiarmid's inequality implies that $p^6(1-10\eps)n^2/10$ such distinct choices of these two $4$-cycles $S_1,S_2$ exist in $D[X\cup\{u_1\} ,Y\cup \{u_2\},C\cup \{c\}]$ with odd-indexed vertices belonging to $X$ and even-indexed vertices belonging to $Y$ with exponentially high probability (because the expectation is of order $n^2$, the random variable is determined by $\leq 2n$ independent trials, and the random variable is $n$-Lipschitz). In particular, we can union bound this event over the choices of $c,e,f$ to conclude it holds with high probability for all $c,e,f$ simultaneously. Fix some choice of $X,Y,C$ so that this event holds. For any sets $X', Y', C'$ that have symmetric difference with $X,Y,C$ bounded by $\eps n$, and for any choices of $e,f,c$, only $3\eps n^2$ of the choices of $S_1,S_2$ can have a coordinate meeting $X\setminus X'$, $Y\setminus Y'$, or $C\setminus C'$. As $p^6(1-10\eps)n^2/10\geq \alpha n^2$, this concludes the proof.

 \smallskip

 \noindent\textbf{\ref{PrP}.\ref{PrP-2}.\ref{prop-pseud-add-2}}: Given distinct $u,v\in V(D)$ and $c_0\in C(D)$, it suffices to show that there are at least $n/10^{10}$ $5$-edge antidirected paths in $D$ with colour sequence as depicted in Figure~\ref{fig:pseudpics} where the paths are pairwise internally vertex-disjoint and colour-disjoint. The desired assertion about the small perturbations of the random sets then follows by applying Chernoff's bound and a union bound over choices of $u,v,c_0$, as before.
 \par To find the desired disjoint antidirected paths, first observe that there are at least $(1-10\eps)n^2$ choices of $w_1,w_2,w_3,w_4,c_1,c_2,c_3$ inducing a $5$-edge antidirected \textit{walk} between $u$ and $v$ with colour sequence $c_1,c_0,c_2,c_0,c_3$. Also observe that only $6n$ of such choices of $5$-edge walks can have an `undesirable colour-repetition' (meaning $c_i=c_0$ for some $i>0$ or $c_i=c_j$ for some $0<i<j$), and only $\binom{6}{2}n$ of the choices can have a vertex-repetition, so $(1-11\eps)n^2$ of these walks are indeed $5$-edge paths with $c_1,c_2,c_3$ being all distinct. Also observe that if we fix any of the variables $w_1,w_2,w_3,w_4,c_1,c_2,c_3$ to equal a particular value of $V(D)$ or $C(D)$, there exist at most $n$ choices of the remainder of the variables so that the variables collectively form a $5$-edge walk of the prescribed pattern. This implies that a maximal disjoint collection of $5$-edge antidirected paths of the desired form has the required size, because any such collection spanning less than $100n/10^{10}$ vertices and colours can then meet at most $4000n^2/10^{10}\ll (1-11\eps)n^2$ of the $5$-edge paths of the desired form.

 \smallskip

 \noindent\textbf{\ref{PrP}.\ref{PrP-2}.\ref{prop-pseud-add-3}}: Given $c_0,d\in C(D)$ the assertion boils down to showing that there are at least $n/10^{10}$ antidirected $8$-cycles in $D$ with `colour-type' $d,c_0,c_1,c_0,c_2,c_0,c_3,c_0$ (as indicated in Figure~\ref{fig:pseudpics}) that are pairwise vertex-disjoint and colour-disjoint except for on $c_0$ and $d$. The first observation is that there are at least $(1-20\eps)n^2$ closed walks on $8$ edges with the desired colour type. As before, if any of the colours $c_1,c_2,c_3$ or any of the vertices are fixed to have a particular value, the number of such closed walks on $8$ edges is at most $n$. Furthermore, observe also that the number of such closed walks that are not cycles can be at most $\binom{8}{2}n$, so there are at least $(1-30\eps)n^2$ choices of $c_1,c_2,c_3$ (which determines the vertices) for which the closed walks are in fact antidirected $8$-cycles. A maximal collection of vertex-disjoint and colour-disjoint (except for $c_0,d$) such $8$-cycles then meets the requirements, as if the maximal collection spanned less than $16n/10^{10}$ vertices and colours, it could meet at most $16^3n^2/10^{10}$ of the other valid choices of $8$-cycles, so then we could extend the collection.

 \smallskip
 \smallskip

For the remainder of the proof, set $\alpha'=10^{-2000}$.
 \smallskip

\noindent \textbf{\ref{PrP}.\ref{PrP-2}.\ref{prop-pseud-add-new-1}}: It suffices to show that the property holds with probability $1-O(n^{-100})$ for each fixed value of $c_0$, $0\leq k\leq 20$, and $\bar{C}\subseteq C(D)\setminus \{c_0\}$ with $|\bar{C}|=5k$, as then we can take a union bound over all combinations of $c_0$ and $\bar{C}$. We will show that, in $D$, there exist vertex-disjoint sets $\bar{V_1},\ldots, \bar{V}_{\alpha'n}$, each of size $2k+2$, where $D[\bar{V}_i]$ contains a spanning antidirected path with colour sequence alternating between a distinct element of $\bar{C}$ and $c_0$, starting and ending with a colour-$c_0$ edge. Note that $D[\bar{V}_i]$ then has a $c_0$-perfect matching as well as a $\bar{C}$-rainbow matching (missing just the endpoints). Furthermore, showing this is enough for \ref{prop-pseud-add-new-1}  as each such spanning path is in particular bipartite, and hence can get feasibly sampled into $D[X,Y,\{c_0\}\cup C]$, so at least $\alpha' p^{2k+2}n/10$ such subsets get sampled into $D[X,Y,\{c_0\}\cup C]$ with probability at least $1-e^{-\sqrt{n}}$. Deleting $\eps n $ elements from $X$ or $Y$ in $D[X',Y',\{c_0\}\cup C]$ can change the count of the sampled $\bar{V_i}s$ by at most $\eps n$, so their count overall would still be $\gg \alpha n$, as desired.
\par To argue that $\bar{V_1},\ldots, \bar{V}_{\alpha'n}$ exist, take a maximal collection of $\bar{V_i}$ and suppose towards a contradiction that $|\bigcup\bar{V_i}|< (2k+2)\alpha'n$. We show how to extend the collection. Call a vertex \textit{deficient} if there exists some colour in $\{c_0\}\cup \bar{C}$ such that the vertex has not got both an in- and an out-neighbour in that colour. Note in $D$ there are at most $2(5k+1)\cdot 2\eps n$ deficient vertices, call them $B$. Let $D'$ be the spanning subgraph of $D$ with edges of colour $\{c_0\}\cup \bar{C}$. The number of vertices at $D'$-distance at most $3k$ to a vertex in $B\cup \bigcup\bar{V_i}$ is at most $$(2(5k+1))^{3k}(4(5k+1)\eps n+ (2k+2)\alpha'n)\leq n/2.$$

\par Each of the remaining at least $n/2$ vertices meets an in-edge and an out-edge of every colour in $\{c_0\} \cup \bar{C}$, with the other endpoint of the each such edge not being in $B\cup \bigcup\bar{V_i}$ (else it would be at distance less than $3k$). Let $v_0$ be any such vertex. We can build a $c_0/\bar{C}$-alternating antidirected path $P$ covering $2k+2$ vertices and starting at $v_0$, which is disjoint from $\bigcup\bar{V_i}$ as follows. Without loss of generality, we will build such a path starting with an out-edge. By choice of $v_0$, we know that $v_0$ has a $c_0$-colour out-neighbour $u_0$. From $u_0$, over $k$ steps, at every step we want to add to the end of the current path an antidirected path of length two, so that the first edge is an in-edge in some colour from $\bar{C}$ which has not already been used, the second is an out-edge of colour $c_0$, and the vertices are distinct from all previous vertices used on the path (so that the antidirected path property is maintained). If we complete all $k$ steps of the process then we have found a path $P$ which allows us to extend the collection $\{\bar{V_i}\}$. Note that, at each step, we may assume that the vertex from which we are extending the path is not in $B\cup \bigcup\bar{V_i}$ (as otherwise it would be at distance less than $3k$ from $v_0$). Thus it has in-edges in all of the at least $5k$ colours from $\bar{C}$. Now, note that at most $k-1$ colours from $\bar{C}$ have been used so far in the path and thus at least $4k+1$ are still available. Let $V'$ be the set of in-neighbours of colours still available. Since the path also contains at most $2k$ vertices, at least $2k+1$ of the vertices in $V'$ have not already been used, call this set $V''$. Moreover, we know that all vertices in $V'$ are not in $B\cup \bigcup\bar{V_i}$ (as, yet again, if they were then they are at distance less than $3k$ from $v_0$). Looking at the vertices in $V''$ we know that each of these has a distinct $c_0$ out-neighbour. At most $k$ of these out-neighbours can lie on the existing path. Thus there is at least one choice of vertex in $V''$ which extends the path as desired. This allows us to extend the collection $\{\bar{V}_i\}$, contradicting its maximality and thus concluding the proof.

\smallskip

\noindent \textbf{\ref{PrP}.\ref{PrP-2}.\ref{prop-pseud-add-new-2}}: Set $k=100$ and fix $c_0$. We will check \ref{prop-pseud-add-new-2} holds with probability $1-O(n^{-2})$ and then we can conclude by a union bound. As before, we will first check the deterministic statement with $\alpha'$ playing the role of $\alpha$, finding the desired bipartite, antidirected subgraphs, and then the assertion for the $p$-random sets just follows by applying Chernoff's bound.
\par Set $r=10^{-10}n$. We find disjoint random vertex sets $V_1,\ldots, V_r\subseteq V(D)$ of size $2k$ and disjoint random colour sets $C_1,\ldots, C_r\subseteq C(D)\setminus \{c_0\}$ of size $k$ which, for each pair $(V_i, C_i)$, produce a random $c_0-C_i$ alternating antidirected cycle on vertex set $V_i$. The random process for obtaining these sets is defined iteratively. 
Suppose the sets $V_i,C_i$ are defined for $i<j\leq r$. We define $V_j,C_j$ as follows. First, pick a random vertex $v\in V(D)\setminus \bigcup_{i<j}V_i$ who has a $c_0$-out-neighbour outside $\bigcup_{i<j}V_i$, and initialise $v$ as the `current endpoint'. At every subsequent stage up until we have a path of length $2k-1$, do the following. Grow the random path by adding the $c_0$-out-neighbour $s$ of the current endpoint, and randomly picking an 
in-neighbour $w$ of $s$ which has the following properties: (1) $w$ isn't on the path built so far, (2) $w$ has a $c_0$-out-neighbour $x$ so that $x$ isn't on the path built so far, (3) the edges $(w,s)$ and $(v,x)$ both exist in $D$ and they have their colours outside $\bigcup_{i<j} C_i$ and the colours occurring along the path thus far. 
Once all $2k$ vertices are picked (yielding an antidirected path of length $2k-1$), the current endpoint picks $v$ as its next in-neighbour (noting that this is an edge of a colour we haven't seen thus far by the invariants of the algorithm). Define $V_j$ to be the $2k$ vertices we picked and define $C_j$ to be the $k$ colours occurring on the antidirected $2k$-cycle other than $c_0$. 

\par The only thing we need to record about this random process is that for any fixed colour sequence $\vec{c}:=c_1,\ldots, c_{k-1}$, the probability that the random process creating $V_j,C_j$ picks $\vec{c}$ as the first $k-1$ colours occurring along $C_j$ is at most $1/(n-100kr)^{k-1}$, as the process picks from at least $n-100kr$ choices at each stage, which follows from $D$ being $(n,\eps)$-typical.

\par Consider now a fixed $\bar{C}\subseteq C(D)\setminus \{c_0\}$ of size $k$, labelled as $\bar{c}_1,\ldots, \bar{c}_k$. We will show that with probability $1-n^{-10k}$, the required property holds for at least $r/2$ values $i\in [r]$. By a union bound over all $\bar{C}$, we can then conclude that the random choices above can be made to satisfy all $\bar{C}$ simultaneously.

Consider a random $C_i=(c_1,\ldots, c_k)$ as defined above. Call a vertex $v$ a \textit{candidate} if there exists a $c_0,\bar{c}_1,c_0$ (antidirected) path starting on $v$ (note there are at least $n-10\eps n$ many candidates $v$ by typicality). For each candidate $v$, consider $P_v$, the maximal length random antidirected walk starting at $v$ with colour sequence (starting with an out-edge) $$c_0,\bar{c}_1,c_0,c_1,c_0,\bar{c}_2,c_0,c_2,\ldots, c_0,c_{k-1},c_0,\bar{c}_k,c_0$$
where we terminate the (random) walk $P_v$ early if any of the neighbours of the appropriate colour are not available. Observe that if an instance of this random walk, $P_v$, does not terminate early and actually is an antidirected path, then its vertices can form a $\bar{V}$ with the desired matching properties. In this case, we say $P_v$ is \textit{good}. Define $\mathcal{L}_{\bar{C}}$ to be the set of sequences $(c_1,\ldots, c_{k-1})$ for which there are at least $n/10$ candidates $v$ for which $P_v$ is good.
\begin{claim}
     $|\mathcal{L}_{\bar{C}}|\geq 99n^{k-1}/100$.
\end{claim}
Let us show how to conclude assuming the claim. For every $i\in [r]$ we have that $C_i$ picks a sequence from $\mathcal{L}_{\bar{C}}$ as its first $k-1$ elements with probability at least $1-(n^{k-1}-|\mathcal{L}_{\bar{C}}|)\cdot \left(\frac{1}{n-100kr}\right)^{k-1}\geq 3/4$. With Azuma's inequality, we can deduce, with probability at least $1-n^{-10k}$, that at least $r/2$ of the random choices $C_i$ pick an element of $\mathcal{L}_{\bar{C}}$. For each such $C_i$ coming from $\mathcal{L}_{\bar{C}}$, we then have $n/10$ potential choices of $\bar{V}$ (that are vertex sets of the good $P_v$s) satisfying the desired properties, but we need to do a final refinement step to pick $\alpha' n$ many disjoint sets $\bar{V}_i$ as in the statement of \textbf{C7}. Take a maximal subcollection of the good $P_v$ that are vertex disjoint, call $\mathcal{J}$ the union of their vertex sets, and suppose towards a contradiction that $|\mathcal{J}|\leq 2k\alpha' n$. Note that the number of vertices that can be at distance $\leq 4k$ to $\mathcal{J}$ in the graph $D'$ that is the restriction of $D$ to edges of colour $C_i\cup \bar{C}\cup \{c_0\}$ is at most $|\mathcal{J}|(5k)^{5k}\leq n/100$. This means there is a vertex $v$ for which $P_v$ is good and $v$ is far away enough from $\mathcal{J}$ that $P_v$ has to be vertex-disjoint with $\mathcal{J}$, contradicting maximality as desired.
\par It remains to prove the claim, whose proof is essentially identical to that of Claim $8$ in \cite[Section 10]{montgomery2023proof}. That is, we first define $\mathcal{R}$ to be the set of alternating-paths of length $4k-1$ starting at any (candidate) vertex $v$ with colour sequence
$$c_0,\bar{c}_1,c_0,c_1,c_0,\bar{c}_2,c_0,c_2,\ldots, c_0,c_{k-1},c_0,\bar{c}_k,c_0$$
for some choices of $c_1,\ldots, c_{k-1}$. By typicality, we can conclude that $|\mathcal{R}|\geq (1-100\eps)n^k$, as for each potential choice of the starting vertex, or the next colour $c_i$, the number of potential new endpoints from which the $c_0,\bar{c}_{i+1}$ continuation of the path would not exist (or lead back into a vertex of the path built so far) is at most $100\eps n$ by typicality of $D$. The claim about the size of $\mathcal{L}_{\bar{C}}$ then follows from the observation that
$$|\mathcal{R}|\leq |\mathcal{L}_{\bar{C}}|n + (n^{k-1}-|\mathcal{L}_{\bar{C}}|)\cdot (n/10) $$
(justified by how each sequence in $\mathcal{L}_{\bar{C}}$ contributes at most $n$ elements to $\mathcal{R}$ by way of choosing $v$, and how the sequences outside $\mathcal{L}_{\bar{C}}$ can contribute at most $n/10$ to $|\mathcal{R}|$ by definition) and then rearranging to verify $|\mathcal{L}_{\bar{C}}|\geq 99n^{k-1}/100$, as desired.


\section{Proof of Theorem~\ref{thm:RSBcoveringstephyper}}\label{appendix:B}

This section is devoted to a proof of Theorem~\ref{thm:RSBcoveringstephyper}, following similar results in \cite{montgomery2023proof} and \cite{muyesser2022random}. We start by giving the usual nibble-type statement that we will use.

\begin{lemma}\label{lem:mainnibble}
    Let $1/n\llpoly \eps \llpoly \eta \llpoly  p\leq 1$ and let $\mathcal{H}$ be a simple 3-partite 3-uniform hypergraph which is $(n,p,\eps)$-regular. Then, $\mathcal{H}$ contains a matching with $(1-\eta)n$ edges.
\end{lemma}
\begin{proof}
    We deduce this from a result of Molloy and Reed, see Theorem~1 from~\cite{molloy2000near}. Applying that theorem with $k=3$, $\Delta=(1+\eps) p n$ gives us a decomposition of our $(n,p,\eps)$-regular hypergraph into $\Delta+c_k\Delta^{1-\frac{1}{k}}\log^4\Delta\leq (1+\eps)pn+n^{3/4}$ matchings (for some constant $c_k$). By the pigeonhole principle one of these matchings has at least $\frac{e(\mathcal{H})}{(1+\eps)pn+n^{3/4}}\geq \frac{((1-\eps)n)(1-\eps)pn)}{(1+2\eps)pn}\geq (1-\eta)n$ edges as required.
\end{proof}

The following is essentially a theorem of Thomason, see \cite[Lemma 3.5]{muyesser2022random} for a proof.
\begin{lemma}[\cite{muyesser2022random}]\label{Theorem_Thomassen}
Let $G$ be an $(n,\delta, \gamma)$-typical bipartite graph.
Then for every $A'\subseteq A, B'\subseteq B$, we have $e(A',B')=\delta|A'||B'|\pm 5\sqrt{(\delta + \gamma) n^3} + \gamma n^2$.
\end{lemma}

\begin{lemma}\label{lem:typicality}
    Let $1/n\llpoly \eps \llpoly \eta \llpoly  p,q\leq 1$. Let $H=(A,B,C)$ be a simple 3-partite 3-uniform hypergraph that is $(n,p,\eps)$-typical. Let $A'\subseteq A$ be $q$-random. Then, with probability at least $1-1/n^3$, the bipartite graph between $B$ and $C$ consisting of edges passing through $A'$ is  $(n,qp,2\eps)$-typical.
\end{lemma}
\begin{proof}
    For some $c,c'\in C$, let  $d_{A'}(c)=e(A',B,c)$ be the degree of $c$ into $A'$ and let $d_{A'}(c,c')$ denote the pair degree of $(c,c')$ into $A'$. We have $d_A(c)=(1\pm \eps)pn$ and  $d_A(c,c')=(1\pm \eps)p^2n$ by the typicality assumptions.
\par Note that $\mathbb E(d_{A'}(c))=qd_{A}(c)= (1\pm \eps)qp n$ and $\mathbb E(d_{A'}(c,c'))=q^2 d_A(c,c')= (1\pm \eps)q^2p^2 n$ for all $c,c'$. Set $\gamma:=n^{-1/5}$. By Chernoff's bound and a union bound, with probability at least $1-1/n^4$, for all $c$ we have $d_{A'}(c)=\mathbb E(d_{A'}(c))\pm \gamma n.$ Note that $d_{A'}(c,c')$ is $2$-Lipschitz. This is because for each $a$, there is at most one $b$ with $abc$ an edge, and one $b$ with $abc'$ an edge (using linearity of $H$). Hence by McDiarmid's inequality and a union bound, we have that with probability at least $1-1/n^4$, for each pair $c,c'\in C$, $d_{A'}(c,c')= \mathbb E(d_{A'}(c,c'))\pm \gamma n$. Corresponding bounds hold for $b,b'\in B$. With probability at least $1-4/n^4$ all these properties hold simultaneously. Whenever these properties all hold, we have that the bipartite graph $(B,C)$ consisting of edges through $A'$ is $(n,qp, 2\eps)$-typical (as $\gamma\ll \eps$) as desired.
\end{proof}

\begin{lemma}\label{Lemma_one_random_set_nearly_regular} Let $1/n\llpoly \eps \llpoly \eta \llpoly  p,q\leq 1$. Let $H=(A,B,C)$ be a simple 3-partite 3-uniform hypergraph that is $(n, p, \eps)$-typical. Let $A'\subseteq A$ be $q$-random. Then, with probability at least $1-1/n^2$, the following holds. For any $B'\subseteq B$, there are at most $\eta n$ vertices $c\in C$ with $e_{H}(A',B', c)\neq pq |B'|\pm \eps^{1/4} n$.
\end{lemma}
\begin{proof}
   By Lemma~\ref{lem:typicality}, with probability at least $1-1/n^3$, we have that the bipartite graph between $B$ and $C$ consisting of edges passing through $A'$ is  $(n,qp,2\eps)$-typical. Supposing this property holds, by Lemma~\ref{Theorem_Thomassen}, for any $B'\subseteq B, C'\subseteq C$, we have $e_{H}(A',B', C')= qp|B'||C'|\pm 4\eps n^2$. Let $C^-$ be the set of vertices $c \in C$ with $e_{H}(A',B', c)< pq|B'|- \eps^{1/4} n$. We have that $e_{H}(A',B', C^-)< pq|B'||C^-|- {\eps}^{1/4} n|C^-|$ and $e_{H}(A',B', C^-)\geq pq|B'||C^-|- 4\eps n^2$ implying $|C^-| \leq \eps^{1/2} n$. Similarly letting $C^+$ be the set of vertices with $e_{H}(A',B', c)>pq|B'|+ \eps^{1/4} n$, we get $|C^+|\leq \eps^{1/2} n$. This implies the lemma.
\end{proof}

\begin{lemma}\label{Lemma_2_random_1_deterministic} Let $1/n\llpoly \eps \llpoly \eta \llpoly  p,q\leq 1$. Let $H=(A,B,C)$ be a simple 3-partite 3-uniform hypergraph that is $(n,p,\eps)$-typical. Let $A'\subseteq A$ be $q$-random, and let $B'\subseteq B$ be $q$-random, where $A'$ and $B'$ are not necessarily independent. Then, with probability at least $1-n^{-2}$, the following holds.
\par For any $C'\subseteq C$ of size $(1\pm \eps)qn$, there is a matching covering all but $\eta n$ vertices in $H[A', B', C']$.
\end{lemma}
\begin{proof}
Fix some $\eta'$ with $\eps\llpoly \eta'\llpoly \eta$. With probability $\geq 1-3/n^2$, the hypothesis of Lemma~\ref{Lemma_one_random_set_nearly_regular} holds with $\eta'$ in place of $\eta$ for three applications, one using $(A, B, C)$ in the given order, one switching the roles of $B$ and $C$, and one permuting the roles to the order $(B,C,A)$. 
Furthermore both $A'$ and $B'$ have size $qn\pm n^{0.6}$ (by Chernoff's bound). This means that for all but $3\eta'n$ vertices, the degrees of vertices in $H[A', B', C']$ is $q^2pn\pm \eps^{1/4}n$. Deleting such vertices, we obtain a graph where the degrees are $q^2pn\pm 4\eta'n$. Applying Lemma~\ref{lem:mainnibble} with $4\eta'$ playing the role of $\eta$, we find a matching covering all but $\eta n$ vertices as desired.
\end{proof}

We can now give the main proof of this section.
\begin{proof}[Proof of Theorem~\ref{thm:RSBcoveringstephyper}]
    Let $\bar{q}$ satisfy $q-\eta/4\leq 3\bar{q}\leq q-\eta/10$ and $\bar{q}n\in \N$. Note that, as $q_A,q_B,q_C\geq 2\bar{q}$ we can take disjoint sets $A_1\cup A_2\subset A'$, $B_1\cup B_2\subset B'$ and $C_1\cup C_2\subset C'$ so that $A_1,A_2, B_1,B_2,C_1$ and $C_2$ are
each $\bar{q}$-random subsets (of $A$, $B$ or $C$).
By Lemma~\ref{Lemma_2_random_1_deterministic} applied three times, and by Chernoff's bound (and then using $3\bar{q}\leq q-\eta/10$), with high probability, we have the following properties.

\stepcounter{propcounter}
\begin{enumerate}[label = {{\textbf{\Alph{propcounter}\arabic{enumi}}}}]
\item For any $\bar{A}\subset A$ with $|\bar{A}|\geq \bar{q}n$, there is a matching in $\mathcal{H}[\bar{A},B_1,C_1]$ with size at least $(\bar{q}-\eta/4)n$.\label{prop-fin-cover-1}
\item For any $\bar{B}\subset B$ with $|\bar{B}|\geq \bar{q}n$, there is a matching in $\mathcal{H}[A_1,\bar{B},C_2]$ with size at least $(\bar{q}-\eta/4)n$.\label{prop-fin-cover-2}
\item For any $\bar{C}\subset C$ with $|\bar{C}|\geq \bar{q}n$, there is a matching in $\mathcal{H}[A_2,B_2,\bar{C}]$ with size at least $(\bar{q}-\eta/4)n$.\label{prop-fin-cover-3}
\item $|A_1\cup A_2|,|B_1\cup B_2|,|C_1\cup C_2|\leq (2\bar{q}+\eta/20)(1+\eps)n\leq (2\bar{q}+\eta/10)n\leq (q-\bar{q})n$. \label{prop-fin-cover-4}
\end{enumerate}

We now claim that we have the property in the theorem.
To see this, take arbitrary sets $\bar{A}\subset A$, $\bar B\subset B$, $\bar{C}\subset C$ with size $qn$ such that $A'\cup B'\cup C'\subset \bar{A}\cup \bar{B}\cup \bar{C}$. Let $\bar{A}'=\bar{A}\setminus (A_1\cup A_2)$, $\bar{B}'=\bar{B}\setminus (B_1\cup B_2)$ and $\bar{C}'=\bar{C}\setminus (C_1\cup C_2)$, noting that, by \ref{prop-fin-cover-4}, we have $|\bar{A}'|,|\bar{B}'|,|\bar{C}'|\geq \bar{q}n$.
By \ref{prop-fin-cover-1},~\ref{prop-fin-cover-2} and~\ref{prop-fin-cover-3}, there are matchings in $\mathcal{H}[\bar{A}',B_1,C_1]$, $\mathcal{H}[A_1,\bar{B}',C_2]$, and $\mathcal{H}[A_2,B_2,\bar{C}']$, each with size at least $(\bar{q}-\eta/4)n$. Combining these gives a matching with at least $(3\bar{q}-3\eta/4)n\geq (q- \eta)n$ edges.
\end{proof}

\end{document}